\documentclass[a4paper]{article}

\RequirePackage{amsthm,amsmath}
\RequirePackage[numbers]{natbib}

\usepackage{amssymb}
\usepackage{latexsym}
\usepackage{color}
\usepackage{xr}

\usepackage{comment}
\excludecomment{discuss}
\excludecomment{discuss2}
\usepackage[dvips]{graphicx}

\newif\ifcol
\colfalse
\ifcol
\newcommand{\colorr}{\color[rgb]{0.8,0,0}}
\newcommand{\colorg}{\color[rgb]{0,0.5,0}}
\newcommand{\colorb}{\color[rgb]{0,0,0.8}}
\newcommand{\colord}{\color[rgb]{0.8,0.3,0}}
\else
\newcommand{\colorr}{\color{black}}
\newcommand{\colorg}{\color{black}}
\newcommand{\colorb}{\color{black}}
\newcommand{\colord}{\color{black}}
\fi

\newtheorem{lemma}{Lemma}[section]
\newtheorem{proposition}{Proposition}[section]
\newtheorem{theorem}{Theorem}[section]
\newtheorem{remark}{Remark}[section]
\newtheorem{example}{Example}[section]

\makeatletter
\@addtoreset{equation}{section}
\makeatother

\begin{document}

\title{Misspecified diffusion models with high-frequency observations and an application to neural networks}
\author{Teppei Ogihara
\thanks{Graduate School of Information Science and Technology, University of Tokyo, 7-3-1 Hongo, Bunkyo-ku, Tokyo 113--8656, Japan.
Email: ogihara@mist.i.u-tokyo.ac.jp}
}
\date{}
\maketitle

\noindent
{\bf Abstract.}
We study the asymptotic theory of misspecified models for diffusion processes with noisy nonsynchronous observations.
Unlike with correctly specified models, the original maximum-likelihood-type estimator has an asymptotic bias under the misspecified setting and fails to achieve an optimal rate of convergence.
To address this, we consider a new quasi-likelihood function that arrows constructing a maximum-likelihood-type estimator that achieves the optimal rate of convergence.
Study of misspecified models enables us to apply machine-learning techniques to the maximum-likelihood approach.
{\colord With these techniques}, we can efficiently study the microstructure of a stock market by using rich information of high-frequency data.
Neural networks have particularly good compatibility with the maximum-likelihood approach, 
so we will consider an example of using a neural network for simulation studies and empirical analysis of high-frequency data from the Tokyo Stock Exchange. 
{\colord We demonstrate that the neural network outperforms polynomial models in volatility predictions for major stocks in Tokyo Stock Exchange.}


\section{Introduction}

\begin{discuss}
{\colorr そもそもIVやQCVを推定せずに$b$の関数形を推定する理由をわかりやすく書く→Podolskij達の研究の後に続けるのが自然か
→とりあえず最後の方に予測ができることを書いてあるからとりあえずいい}
\end{discuss}


High-frequency financial data, such as data on all intraday transactions from a stock market, are increasingly available.
These data contain lot of information about the intraday stock market, so they are expected to contribute to the analysis of stock microstructures.
More complicated structures exist in high-frequency data than in low-frequency data (e.g. daily or weekly time series of stock transactions), which increases the difficulty of statistical analysis.
One problem is that observation noise is intensified with frequency.
To explain empirical evidence, when we model stock price data as a continuous stochastic process, we must assume that the observations contain additional noise.
Another significant problem with analysis of high-frequency data is that nonsynchronous observation occurs; namely, we observe the prices of different securities at different time points.
Nonsynchronous observations make it more difficult to construct estimators of the covariation of pairs of stock.
Covariation estimation in the presence of both noise and nonsynchronicity has been studied in many papers, and various consistent estimators have proposed.
See, for example, Barndorff-Nielsen et al.~\cite{bar-etal11}, Christensen, Kinnebrock, and Podolskij~\cite{chr-etal10}, and Bibinger et al.~\cite{bib-etal14}.

These problems are related to the complex structure of observations. A model of (efficient) stock price dynamics also has a complex structure
that includes intraday periodicity, volatility clustering, asymmetry of return distributions, and other complications.
Though these individual complications have been investigated in several papers, 
a comprehensive model that explains all of these complications simultaneously has not yet been found.
On the other hand, statistical machine learning has been showing great success in the analysis of complicated nonlinear structures, 
with the structure being found from training with rich data. Image recognition is the most notable success in this field of problems.
The task of detecting an object in a picture was long considered to be difficult for computers because parametric models require infeasibly complicated nonlinear analysis to construct.
However, deep neural networks have achieved object recognition by training with huge amounts of data.

Toward progress on analysis of market microstructures, we consider nonsynchronous observations contaminated by market microstructure noise.
For this, let $(Y_t)_{t\geq 0}$ be a a multi-dimensional stochastic process satisfying the following equation:
\begin{equation}\label{nonpara-sde}
Y_t=Y_0+\int^t_0\mu_sds+\int^t_0b_{s,\dagger}dW_s, \quad t\in[0,T],
\end{equation}
where $(W_t)_{t\geq 0}$ is a multi-dimensional standard Wiener process and where $\mu_t$ and $b_{t,\dagger}$ are stochastic processes.
We consider a statistical model of $(Y_t)_t$ with nonsynchronous observations contaminated by market microstructure noise.
Ogihara~\cite{ogi18} studied maximum-likelihood- and Bayes-type estimation and showed estimators with asymptotic mixed normality when $b_{t,\dagger}$ satisfies 
\begin{equation}\label{spe-condition}
b_{t,\dagger}=b(t,X_t,\sigma_\ast) \quad (t\in [0,T]) \quad {\rm a.s.,}
\end{equation}
where $b(t,x,\sigma)$ is a given function, $\sigma_\ast$ is an unknown parameter, and 
$X_t$ is an explanatory process (in practice for our problem, the stock prices of other stocks, accumulated trading volume and so on).
Asymptotic efficiency of the estimators was also proved by showing local asymptotic normality when the diffusion coefficients are deterministic and each noise follows a normal distribution.

However, as mentioned above, it is difficult to find a parametric model satisfying (\ref{spe-condition}) in the practice of high-frequency data analysis.
Models for which this assumption is unsatisfied are called {\it misspecified models}.
To apply a neural network to the above situation, we approximate $\Sigma_{t,\dagger}=b_{t,\dagger}b_{t,\dagger}^\top$ by $\Sigma(t,X_t,\beta)$
where $\top$ denotes the transpose operator and the function $\Sigma$ is given by a neural network with a parameter $\beta$ (see Section~\ref{nn-example-section} for the precise definition). 
In this case, it is natural to consider misspecified models.
{\colord Study of misspecified model is important in general situations because there is always a gap between a parametric model chosen by a statistician and the model of the real data.} 
Misspecified models have not been well-studied for diffusion-type processes with high-frequency observations in a fixed interval, 
even for models in which neither nonsynchronicity nor market microstructure noise is included.
For the case with the end time $T$ of observations approaching infinity,
Uchida and Yoshida~\cite{uch-yos11} studied an ergodic diffusion process $(X_t)_{t\geq 0}$ with observations $(X_{kh_n})_{k=0}^n$, where $h_n\to 0$, $nh_n\to \infty$, and $nh_n^2\to 0$.
In that case, the rate of convergence of the maximum-likelihood-type estimator for a parameter in the diffusion coefficients is $\sqrt{nh_n}$, which is different from the rate $\sqrt{n}$ seen in correctly specified cases.

In this paper, we study the asymptotic properties of a maximum-likelihood-type estimator $\hat{\sigma}_n$ with a misspecified model containing noisy, nonsynchronous observations.
In contrast with the results for the correctly specified model of Ogihara~\cite{ogi18}, 
it is shown that the original maximum-likelihood-type estimator has an asymptotic bias for a misspecified model,
resulting in a failure to achieve the optimal rate of convergence. 
Identifiability does not hold in general, and therefore we cannot ensure convergence of the maximum-likelihood-type estimator in the parameter space.
However, if we consider a certain '{\it distance}' $D$ in the functional space of diffusion coefficients, 
we can show convergence of the value of a $D$ between $\Sigma(t,X_t,\hat{\sigma}_n)$ and $\Sigma_{t,\dagger}$ to the minimum value of $D$ between $\Sigma(t,X_t,\sigma)$ and $\Sigma_{t,\dagger}$ in the parametric family.
If we assume uniqueness of the parameters that minimize $D$, then we obtain asymptotic mixed normality of the estimator. 
In this case, the limit of $\hat{\sigma}_n$ is a random variable, which prevents directly using martingale central limit theorems (Theorems 2.1 and 3.2 in Jacod~\cite{jac97}),
which are typically used to show the asymptotic mixed normality of estimators.
To deal with this problem, we develop techniques (notably, Proposition~\ref{random-param-convergence}) using series expansions of the inverse co-volatility matrix to obtain asymptotic mixed normality.
This result is general and seems to be useful for showing the asymptotic mixed normality of estimators with random limits in many situations.

We can apply the maximum-likelihood approach to any machine learning methodology that constructs a parametric model $(\Sigma(t,X_t,\sigma))_\sigma$.
Additionally, the maximum-likelihood approach has good compatibility with neural networks in several respects.
\begin{enumerate}
\item A neural network can be expressed as a parametric family $\{\Sigma(t,X_t,\beta)\}_\beta$, so construction of a maximum-likelihood-type estimator is immediate, as described in Ogihara~\cite{ogi18}.
\item Optimization methods of neural networks can still be applied by setting the minus quasi-likelihood function as the loss function.
In particular, backpropagation still works in this setting.
\item The increase in optimization cost is relatively mild when increasing the dimensionality of parameters
because we need only the gradient of loss function with respect to variables at the output layer, by virtue of backpropagation 
(see Section~\ref{nn-example-section} for details). 
\end{enumerate}

\begin{discuss2}
{\colorr As we will see later in Section \ref{nn-example-section}, the number of variables in the output layer is equal to 
the dimension of Wiener process times the dimension of the latent process$Y$, which can be set relatively small even for high dimensional case.}
\end{discuss2}

\begin{discuss}
{\colorr RNNを含めた時系列に対するニューラルネットワークの先行研究に少し触れるか。非同期やノイズ等の高頻度観測の問題には触れていない。SDEの構造を使うことで高速かつ安定に推定することが可能．
→何か言われたらコメントするくらいでいい}
{\colorr サポートベクタマシンでは将来時刻を参照してしまい（この辺りは数式使ってより詳しく書く）、確率過程では技術的に難しくなってしまう。また，$H_n$にカーネルを入れることで非線形最適化が必要となりSVMのメリットである計算の高速化が実現できない．}
\end{discuss}

It is noteworthy that this study enables us to apply machine learning methodologies to intraday stock high-frequency data analysis
under conditions of both nonsynchronicity and market microstructure noise.
By training the structure using high-frequency data, we can forecast stock price volatilities and covariations.
We will apply the proposed method by training with data from major stocks in the Tokyo Stock Exchange.
By training the model using three-month data of each stock and the proposed method, 
we can forecast the volatility function $\Sigma_{t,\dagger}$ of each day for the next month.

The reminder of this paper is organized as follows.
In Section~\ref{setting-section}, we explain our settings and provide an example of a neural network using the quasi-likelihood function.
A general asymptotic theory of misspecified models is discussed in Section~\ref{misspe-theory-section}.
We specify the limit of the estimator by using a distance $D$, which is related to Kullback--Leibler divergence in some sense.
The maximum-likelihood-type estimator has an asymptotic bias and so does not achieve an optimal rate of convergence. We propose a modified estimator that does achieve an optimal rate.
Section~\ref{simulation-section} studies simulation of neural networks for the case in which 
the latent diffusion process is a one-dimensional Cox--Ingersoll--Ross process,
and the case in which the latent process is a two-dimensional CIR-type process with intraday periodicity. 
Using a neural network, we train the function $\Sigma_{t,\dagger}$ without any information about the parametric model.
We study Japanese stock high-frequency data in Section~\ref{empirical-section}. 
Proofs are provided in Section~\ref{proofs-section}.

\section{Settings and an example neural network}\label{setting-section}

\subsection{Parametric estimation under misspecified settings}\label{setting-subsection}

\begin{discuss2}
{\colorg リバイズ時：stable convergenceを示すとき，$\Omega=\Omega^{(0)}\times \Omega^{(1)}$とかOgi18のセッティングにする．} 
\end{discuss2}
Let $\gamma,\gamma_W\in\mathbb{N}$ and $(\Omega,\mathcal{F},P)$ be a probability space with a filtration ${\bf F}=\{\mathcal{F}_t\}_{0\leq t\leq T}$ for some $T>0$.
We consider a $\gamma$-dimensional ${\bf F}$-adapted process $Y=\{Y_t\}_{0\leq t\leq T}$ satisfying an integral equation (\ref{nonpara-sde}),
where $\{W_t\}_{0\leq t\leq T}$ is a $\gamma_W$-dimensional standard ${\bf F}$-Wiener process and where
$\{\mu_t\}_{0\leq t\leq T}$ and $b_\dagger=\{b_{t,\dagger}\}_{0\leq t\leq T}$ are $\mathbb{R}^\gamma$- and $\mathbb{R}^\gamma\otimes \mathbb{R}^{\gamma_W}$-valued ${\bf F}$-progressively measurable processes, respectively.
\begin{discuss}
{\colorr stable convergenceの議論をしないならlatent processとノイズのfiltrationを分ける必要はない}
\end{discuss}

We assume that the observations of processes occur in a nonsynchronous manner and are contaminated by market microstructure noise.
That is, we observe the sequence $\{\tilde{Y}^k_i\}_{0\leq i\leq {\bf J}_{k,n},1\leq k\leq \gamma}$, 
where $\{{\bf J}_{k,n}\}_{1\leq k\leq \gamma,n\in\mathbb{N}}$ are positive integer-valued random variables, $\{S^{n,k}_i\}_{i=0}^{{\bf J}_{k,n}}$ are random times,
$\{\epsilon^{n,k}_i\}_{i\in\mathbb{Z}_+,1\leq k\leq 2}$ is an independent identical distributed random sequence, and
\begin{equation}
\tilde{Y}^k_i=Y^k_{S^{n,k}_i}+\epsilon^{n,k}_i. 
\end{equation}

Let $\gamma_X\in\mathbb{N}$ and let $\top$ denote the transpose operator for matrices (including vectors).
Let $\bar{E}$ denote the closure of a subset $E$ of a Euclidean space.
We consider estimation of the co-volatility matrix 
$\Sigma_{t,\dagger}=b_{t,\dagger}b_{t,\dagger}^\top$ by using a functional $\Sigma(t,X_t,\sigma)$ 
with a $\gamma_X$-dimensional c\`adl\`ag stochastic process $X=(X_t)_{t\in[0,T]}$ and a parameter $\sigma$.
We observe possibly noisy data: $\tilde{X}_j^l=X^l_{T^{n,l}_j}+\eta^{n,l}_j$ for $0\leq j\leq {\bf K}_{l,n}$ and $1\leq l\leq \gamma_X$,
where $\{{\bf K}_{l,n}\}_{1\leq l\leq \gamma_X,n\in\mathbb{N}}$ are positive integer-valued random variables, $\{T^{n,l}_j\}_{j=0}^{{\bf K}_{l,n}}$ are random times,
and $\{\eta^{n,l}_j\}_{j=0}^{{\bf K}_{l,n}}$ are random variables which may be identically equal to zero.

We can arbitrarily set the explanatory variable $X$ and the function $\Sigma(t,x,\sigma)$.
For example, we can use other stock price processes, a price process of a stock index, the accumulated volume of stock trades, 
$Y$ itself, or some combination of these.
Let $\Sigma(t,x,\sigma)$: $[0,T]\times \mathcal{O}\times \bar{\Lambda}\to \mathbb{R}^\gamma\otimes \mathbb{R}^\gamma$ be some known continuous function, 
where $\mathcal{O}\subset \mathbb{R}^{\gamma_X}$ is an open set and the parameter space $\Lambda\subset \mathbb{R}^d$ is a bounded open set with $d\in\mathbb{N}$.
\begin{discuss}
{\colorr ソボレフの不等式を使うためには$\Lambda$は有界である必要があり，微分があるからopen setでの定義が必要．またMLEのdefでも有界性が必要．
ニューラルネットワークではCIRが真の関数でも$\Sigma$のdef域は全体となり，$\mathcal{O}$を考える必要はないが，ここではgeneral theoryだから$\mathcal{O}$を使う．}
\end{discuss}
Letting $\Pi_n=(\{S^{n,k}_i\}_{k,i},\{T^{n,l}_j\}_{l,j})$, we assume that $\mathcal{F}_T$, $(\Pi_n)_{n\in\mathbb{N}}$ and $\{\epsilon^{n,k}_i\}_{n,k,i}$ are mutually independent.
Let
\begin{equation*}
\mathcal{G}_t=\mathcal{F}_t\bigvee \mathfrak{B}(\{\Pi_n\}_n) \bigvee \mathfrak{B}(A\cap \{S^{n,k}_i\leq t\} ; A\in \mathfrak{B}(\epsilon^{n,k}_i), k\in \{1,2\}, i\in\mathbb{Z}_+, n\in\mathbb{N}),
\end{equation*}
where $\mathfrak{B}(\cdot)$ denotes the minimal $\sigma$-field related to measurable sets or random variables in the parenthesis, 
and  $\mathcal{H}_1\bigvee \mathcal{H}_2$ denotes the minimal $\sigma$-field that contains $\sigma$-fields $\mathcal{H}_1$ and $\mathcal{H}_2$.
\begin{discuss}
{\colorr $\eta$に独立なものを入れたいなら$\mathcal{F}^{(0)}_t$にいれてちゃんと$E_m[\#^{-1}\sum_j\eta^{n,k}_j]=\bar{R}_n(\ell_n^{-1})$となるようにしていく．}
\end{discuss}
Moreover, we assume that $\eta^{n,l}_j1_{\{T^{n,l}_j\leq t\}}$ is $\mathcal{G}_t$-measurable and that both $E[\epsilon^{n,k}_0]=0$ and $E[(\epsilon^{n,k}_0)^2]=v_{k,\ast}$,
where $1_A$ is the indicator function for a set $A$ and $v_{k,\ast}$ is positive constant for $1\leq k\leq \gamma$.
\begin{discuss}
{\colorr $Y_0$は元々$\sigma$や$v_\ast$には依らないから仮定する必要はない．}
\end{discuss}

We consider a maximum-likelihood-type estimator of $\sigma$ based on a quasi-likelihood function. Construction is based on that described in Ogihara \cite{ogi18}.
Let $\{b_n\}_{n\in\mathbb{N}}$ and $\{\ell_n\}_{n\in\mathbb{N}}$ be sequences of positive numbers satisfying $b_n\geq 1$, $\ell_n\in\mathbb{N}$, 
$b_n\to\infty$, $\ell_n/b_n^{-1/3-\epsilon}\to \infty$, and $b_n^{1/2-\epsilon}/\ell_n\to \infty$ as $n\to\infty$ for some $\epsilon>0$.
\begin{discuss}
{\colorr $\ell_n$のレートは一致性ではもう少し改善できる可能性もある．}
\end{discuss}
For technical reasons, we construct a quasi-log-likelihood function by partitioning the whole observation interval $[0,T]$ into distinct local intervals $\{s_{m-1},s_m)\}_{m=1}^{\ell_n}$.
Here the sequence $\{b_n\}_{n\in\mathbb{N}}$ is the order of sampling frequency, that is,
\begin{equation*}
0<P\mathchar`-\lim_{n\to\infty} (b_n^{-1}{\bf J}_{k,n})<\infty
\end{equation*}
almost surely for $1\leq k\leq \gamma$.

We use several notations for clarity: $k_n=b_n\ell_n^{-1}$,
$K_0^j=-1$, $K_m^j=\#\{i\in \mathbb{N};S_i^{n,j}<s_m\}$, $k^j_m=K^j_m-K^j_{m-1}-1$, $I_{i,m}^k=[S^{n,k}_{i+K_{m-1}^k},S^{n,k}_{i+1+K_{m-1}^k})$
and $M(l)=\{2\delta_{i_1,i_2}-1_{\{|i_1-i_2|=1\}}\}_{i_1,i_2=1}^l$, where $\delta_{ij}$ is Kronecker's delta. 
For an interval $J=[a,b)$, we write $|J|=b-a$. $\mathcal{E}_l$ denotes a unit matrix of size $l$.
For a matrix $A$, we denote its $(i,j)$ element by $[A]_{ij}$.
Let ${\rm diag}((u_i)_{i=1}^K)$ be a $K\times K$ diagonal matrix with elements $[{\rm diag}((u_i)_i)]_{jk}=u_j\delta_{jk}$ for a vector $(u_i)_{i=1}^K$.
For matrices $(M_i)_{i=1}^l=(([M_i]_{jk})_{1\leq j,k\leq K_i})_{i=1}^l$,
 ${\rm diag}((M_i)_{i=1}^l)$ be a $L_l\times L_l$ matrix with elements 
\begin{equation*}
[{\rm diag}((M_i)_{i=1}^l)]_{jk}=\left\{
\begin{array}{ll}
[M_i]_{j-L_{i-1},k-L_{i-1}} & {\rm if} \ L_{i-1}<j,k\leq L_i \ {\rm for \ some} \ i \\
0 & {\rm otherwise}
\end{array}
\right.
\end{equation*}
where $L_0=0$ and $L_i=\sum_{i'=1}^iK_{i'}$ for $1\leq i\leq l$.

Let us consider an observable approximation $\hat{X}_m$ of $X_{s_{m-1}}$ defined by
\begin{equation*}
\hat{X}_m=\bigg(\#\{j;T^{n,k}_j\in [s_{m-1},s_m)\}^{-1}\sum_{j;T^{n,k}_j\in [s_{m-1},s_m)}\tilde{X}_j^k\bigg)_{1\leq k\leq \gamma_X}.
\end{equation*}
Let 
$\Sigma_m(\sigma)=\Sigma(s_{m-1},\hat{X}_{m-1},\sigma)$, $M_{j,m}=M(k^j_m)$,
$Z_{m,l}^i=\tilde{Y}_{l+1+K_{m-1}^i}^i-\tilde{Y}_{l+K_{m-1}^i}^i$ and $Z_m=((Z_{m,l}^1)_l^\top,\cdots,(Z_{m,l}^\gamma)_l^\top)^\top$.
Then, roughly speaking, we obtain approximation:
\begin{eqnarray}
E[Z_{m,i}^kZ_{m,j}^k|\mathcal{G}_{s_{m-1}}]&\approx & [\Sigma_{s_{m-1},\dagger}]_{kk}|I_{i,m}^k|\delta_{ij}+v_{k,\ast}[M_{k,m}]_{ij}, \nonumber \\
E[Z_{m,i}^kZ_{m,j}^l|\mathcal{G}_{s_{m-1}}]&\approx & [\Sigma_{s_{m-1},\dagger}]_{kl}|I_{i,m}^k\cap I_{j,m}^l| \nonumber 
\end{eqnarray}
for $k\neq l$. Therefore, by setting
\begin{eqnarray}
{\bf S}_m(B,v)&=&\left(\begin{array}{ccc} [B]_{11}\{|I^1_{i,m}\cap I^1_{j,m}|\}_{ij} & \cdots &  [B]_{1\gamma}\{|I^1_{i,m}\cap I^\gamma_{j,m}|\}_{ij} \\ 
\vdots & \ddots & \vdots \\
{[B]}_{\gamma 1}\{|I^\gamma_{i,m}\cap I^1_{j,m}|\}_{ij} & \cdots & [B]_{\gamma \gamma}\{|I^\gamma_{i,m}\cap I^\gamma_{j,m}|\}_{ij}
\end{array} \right) \nonumber \\
&&+\left(\begin{array}{ccc} v_1M_{1,m} & & \\ & \ddots & \\ & & v_\gamma M_{\gamma,m} \end{array} \right) \nonumber 
\end{eqnarray}
for a $\gamma\times \gamma$ matrix $B$ and $v=(v_1,\cdots, v_\gamma)$, and $S_m(\sigma,v)={\bf S}_m(\Sigma_m(\sigma),v)$,
we define a quasi-log-likelihood function
\begin{equation*}
H_n(\sigma,v)=-\frac{1}{2}\sum_{m=2}^{\ell_n}(Z_m^{\top}S_m^{-1}(\sigma,v)Z_m+\log \det S_m(\sigma,v)).
\end{equation*}
The function $H_n$ takes two parameters: the first one is $\sigma$, which is the parameter for the estimating function of $\Sigma_\dagger$ and is the parameter of interest.
The second parameter is $v$, which is the parameter for noise variance. Though we can consider simultaneous maximization of $\sigma$ and $v$,
we fix an estimate of $v$ in advance. With this approach, we can apply our results to the case of non-Gaussian noise.

\begin{description}
\item{[V]} There exist estimators $\{\hat{v}_n\}_{n\in\mathbb{N}}$ of $v_\ast$ such that $\hat{v}_n\geq 0$ almost surely
and $\{b_n^{1/2}(\hat{v}_n-v_\ast)\}_{n\in\mathbb{N}}$ is tight.
\end{description}
It is easy to obtain a suitable $\hat{v}_n$. For example, let $\hat{v}_n=(\hat{v}_n^1,\cdots, \hat{v}_n^\gamma)$ and
\begin{equation*}
\hat{v}_{k,n}=(2{\bf J}_{k,n})^{-1}\sum_{i,m}(Z_{i,m}^k)^2.
\end{equation*}
Then $\hat{v}_n$ satisfies [V] if $\{b_n{\bf J}_{k,n}^{-1}\}_{n\in\mathbb{N}}$ is tight and $\sup_{n,k,i}E[(\epsilon_i^{n,k})^4]<\infty$.

We fix $\hat{v}_n$ such that it satisfies $[V]$. We define a maximum-likelihood-type estimator
\begin{equation*}
\hat{\sigma}_n={\rm argmax}_{\sigma}H_n(\sigma,\hat{v}_n).
\end{equation*}

$H_n$ is constructed on the basis of the local Gaussian approximation of $Z_{m,i}^k$. 
This approximation is typically considered valid only when the observation noise $\epsilon_i^{n,k}$ follows a normal distribution.
However, we can see in the proof that this approximation is also valid and $\hat{\sigma}_n$ still works for the case of non-Gaussian noise.

\begin{discuss}
{\colorr Ogihara 2017と違う条件は, identifiability, $X_t$の条件, $b_{t,\dagger}\neq b(t,X_t,\sigma_{\ast})$.
極限の導出ではidentifiailityは使わない.}
\end{discuss}

\subsection{An example of a neural network}\label{nn-example-section}

In Section \ref{setting-subsection}, we construct the maximum-likelihood-type estimator for the parametric model $(\Sigma(t,X_t,\sigma))_\sigma$.
If we can find a parametric model that contains a good approximation of $\Sigma_{t,\dagger}$ with a low-dimensionality parameter,
then we can estimate $\Sigma_{t,\dagger}$ by using the above method.
However, as indicated in the introduction, it is not an easy task to find a good parametric model of this type.
In contrast, it seems effective to learn $\Sigma_{t,\dagger}$ by machine learning methods, 
since typical high-frequency data contain high volume data.
In particular, a neural network is strongly compatible with the quasi-likelihood approach and is useful in practice.

\begin{discuss}
{\colorr ネットワークの図を入れるか→無くても問題ない．
nnのパラメータはgeneral caseと区別するため$\beta$で書く．}
\end{discuss}

We start from input data $x=(x_1,\cdots, x_{\gamma_X})\in\mathbb{R}^{\gamma_X}$. 
The neural network here consists of three types of layers: an input layer, hidden layers, and an output layer.
The hidden layers consist of elements $(u^k_j)$, which are inductively defined as
\begin{equation*}
u^0_{j(0)}=x_{j(0)}, \quad u^k_{j(k)}=h\bigg(\sum_{l=1}^{L_k}\beta_{l,j(k)}^ku^{k-1}_l\bigg),
\end{equation*}
for $1\leq j(k)\leq L_k, 0\leq k\leq K-1$,
where $K\geq 2$, $L_0=\gamma_X$, $(L_k)_{k=1}^{K-1}\subset \mathbb{N}$, parameters $(\beta_{l,j}^k)$ are elements of $\mathbb{R}$,
and $h:\mathbb{R}\to \mathbb{R}$ is a continuous function (the so-called activating function).
In our simulation, we use Swish, proposed in Ramachandran, Zoph, and Le~\cite{ram-etal17} and defined as $h(x)=x/(1+e^{-x})$.

The output layer is given by $[b]_{ij}=\sum_l\beta^K_{i,j,l}u_l^{K-1}$ for $1\leq i\leq j\leq \gamma$. 
Let $[b]_{ij}=[b]_{ji}$ for $i>j$.
The estimator $\Sigma(t,\sum_m\hat{X}_m1_{[s_{m-1},s_m)}(t),\beta)$ of $\Sigma_{t,\dagger}$ is given by setting $\Sigma(t,x,\beta)=bb^\top(t,x,\beta)+\epsilon\mathcal{E}_\gamma$
for $\beta=(\beta_{l,j}^k)$ with some small constant $\epsilon\geq 0$.
Then, choosing an activating function $h$ and the structure of hidden layers defines a parametric family of non-linear functions.
{\colord As we will see in Remark~\ref{univ-approx-rem}, this family approximates any continuous function for sufficiently large $L_1$ with $K=2$.
In addition, this family expresses exponentially large complexity against the number $K$ of layers, as studied in Montufar et al.~\cite{monEtAl14}.}

Let $H_n(\beta,v)$ be a quasi-log-likelihood function constructed by the above $\Sigma(t,x,\beta)$.
From this, we can construct a maximum-likelihood-type estimator by letting
\begin{equation*}
\hat{\beta}_n={\rm argmax}_\beta H_n(\beta,\hat{v}_n).
\end{equation*}

Calculation of $\hat{\beta}_n$ is not easy. However, several optimization techniques have been proposed for use in neural networks,
and some of them can be applied to our case.
In particular, the use of back propagation to obtain the gradient $\partial H_n/\partial \beta_{i,j}^k$
is valid for our model, and we can quickly compute gradients by this method.

Let $h\in C^1(\mathbb{R})$. The chain rule yields
\begin{eqnarray}
\frac{\partial H_n}{\partial \beta_{l,j}^k}=\frac{\partial H_n}{\partial u^k_j}\cdot \frac{\partial u^k_j}{\partial \beta_{l,j}^k}
=\frac{\partial H_n}{\partial u^k_j}h'\bigg(\sum_{l=1}^{L_k}\beta_{l,j}^ku_l^{k-1}\bigg)u_l^{k-1}.
\end{eqnarray}
Then, setting $\Delta_{k,j}=\partial H_n/\partial u^k_j$ and applying the chain rule again yields
\begin{equation}
\Delta_{k,j}=\sum_{i=1}^{L_{k+1}}\Delta_{k+1,i}\frac{\partial u_{k+1,i}}{\partial u_{k,j}}=\sum_{i=1}^{L_{k+1}}\Delta_{k+1,i}\beta_{j,i}^{k+1}h'\bigg(\sum_{l=1}^{L_{k+1}}\beta_{l,j}^{k+1}u_{j,l}^k\bigg).
\end{equation}
Therefore, we can inductively calculate $\Delta_{k,j}$; that is, fast calculation of the gradients is possible.

Calculation of $H_n$ or its derivatives requires $\ell_n$ times calculation of inverse matrices $(S_m)_m$, with the matrix sizes of order $k_n$.
Back propagation offers the advantage that we need to calculate only the derivatives $\partial H_n/\partial [b]_{ij}$ for the output layer.

\section{Asymptotic theory for misspecified model of diffusion-type processes}\label{misspe-theory-section}

Ogihara \cite{ogi18} studied a correctly specified model given by
\begin{equation}\label{specified-condition}
\Sigma_{t,\dagger}\equiv \Sigma(t,X_t,\sigma_\ast) \quad {\rm for} \ {\rm some} \ {\rm nonrandom} \ \sigma_\ast\in \Lambda
\end{equation}
and showed the asymptotic mixed normality of $\hat{\sigma}_n$ and local asymptotic normality in a special case. 
In machine-learning theory, which includes neural networks,
we typically consider models that do not necessarily satisfy (\ref{specified-condition}); as mentioned earlier,  we call these misspecified models.

In the study of misspecified models, we sometimes see different asymptotic behavior than with correctly specified models.
For example, Uchida and Yoshida~\cite{uch-yos11} studied ergodic diffusion $X=(X_t)_{t\geq 0}$ with observations $(X_{kh_n})_{k=0}^n$,
where $h_n\to 0$, $nh_n\to \infty$ and $nh_n^2\to 0$ as $n\to\infty$.
They showed that the rate of convergence of a maximum-likelihood-type estimator for parameters in diffusion coefficients
is $\sqrt{nh_n}$, which is different than the rate $\sqrt{n}$ found for correctly specified models.

In our case, we also observe different phenomena than in the correctly specified case.
In particular, the maximum-likelihood-type estimator $\hat{\sigma}_n$ cannot attain an optimal rate of convergence
due to the existence of asymptotic bias.
Despite this, we can construct an estimator that attains the optimal rate by modifying the asymptotic bias.

When we consider a misspecified parametric model including a neural network, the limit of a parameter 
which maximizes $H_n$ is not guaranteed to be unique in general, meaning that we cannot ensure convergence of the maximum-likelihood-type estimator.
Consistency should be studied not in the parameter space but in the space of co-volatility functions.
If we define a function $D(\Sigma_1,\Sigma_2)$ over the space of co-volatility matrices by (\ref{D-def}) later, we obtain
\begin{equation}\label{D-conv}
D(\Sigma(\hat{\sigma}_n),\Sigma_\dagger)\overset{P}\to \min_\sigma D(\Sigma(\sigma),\Sigma_\dagger).
\end{equation}
Intuitively, $D$ is a kind of extension of Kullback--Leibler divergence (see (\ref{Hn-conv-toD})),
and we can obtain equivalence between $D$ and the $L^2([0,T]\times \Omega)$ norm in the sense of $(\ref{L2-equivalence})$.

\begin{discuss2}
{\colorg $D$で推定することの意味を反論されないようにもっと詳しく書きたいか→仕上げの時に余力があれば}
\end{discuss2}

\subsection{Consistency}

In this section, we study results related to the consistency of $\hat{\sigma}_n$.
Since convergence of $\hat{\sigma}_n$ is not guaranteed, we characterize the convergence
by means of a function $D(\Sigma_1,\Sigma_2)$.

Here, we make some assumptions about the latent stochastic process $X,Y$ and the market microstructure noise $\epsilon_i^{n,k}$.
Let $E_\Pi[{\bf X}]=E[{\bf X}|\{\Pi_n\}_n]$ for a random variable ${\bf X}$, 
\begin{discuss}
{\colorr $\Pi_n$と$Y$等の独立性は仮定しているから$\{\Pi_n\}_n$で条件づけても$=E[]_{S^{n,j}_i=s^j_i}$の形にはできる}
\end{discuss}
let $|A|^2=\sum_{i,j}[A]_{ij}^2$,
let ${\rm Abs}(A)=(|[A]_{ij}|)_{ij}$ and let $\lVert A\rVert$ be the operator norm for a matrix $A$.
For random variables $\{{\bf X}_n\}_{n\in\mathbb{N}}$ and a sequence $\{c_n\}_{n\in\mathbb{N}}$ of positive numbers, 
we write ${\bf X}_n=O_p(c_n)$ if $\{c_n^{-1}{\bf X}_n\}_{n\in\mathbb{N}}$ is $P$-tight,
and denote ${\bf X}_n=o_p(c_n)$ if $c_n^{-1}{\bf X}_n\overset{P}\to 0$ as $n\to\infty$.
For a vector $x=(x_1,\cdots,x_k)$ we use the notation $\partial_x^l=(\frac{\partial^l}{\partial x_{i_1}\cdots \partial x_{i_l}})_{i_1,\cdots, i_l=1}^k$.

We assume that $\Lambda\subset \mathbb{R}^d$ and that $\Lambda$ satisfies Sobolev's inequality; that is, for any $p>d$, there exists some $C>0$ such that 
$\sup_{\sigma\in\Lambda}|u(\sigma)|\leq C\sum_{k=0,1}(\int_\Lambda |\partial_\sigma^ku(\sigma)|^pd\sigma)^{1/p}$ for any $u\in C^1(\Lambda)$.
Notably, this holds when $\Lambda$ has a Lipschitz boundary. See Adams and Fournier \cite{ada-fou03} for more details.
\begin{discuss2}
{\colorg 余：ソボレフの代わりにC spaceの収束でも十分かも}
\end{discuss2}

Moreover, we assume the following conditions.
\begin{description}
\item{[A1]} 1. $\partial_\sigma^l\Sigma$ exists and is continuous for $l\in \{0,1\}$ on $[0,T]\times \mathcal{O}\times \bar{\Lambda}$.
There exists a locally bounded function $L(x,y)$ such that
\begin{equation}\label{Lip-conti-ineq}
|\partial_\sigma^l\Sigma(s,x,\sigma)-\partial_\sigma^l\Sigma(t,y,\sigma)|\leq L(x,y)(|t-s|+|y-x|)
\end{equation}
for any $s,t\in[0,T]$, $x,y\in\mathcal{O}$, $\sigma\in\Lambda$ and $l\in\{0,1\}$.
\begin{discuss}
{\colorr $\partial_\sigma\Sigma$はソボレフを使うために必要．}
\end{discuss}
\\
2. $\Sigma(t,x,\sigma)$ is symmetric and positive definite and $\lVert \Sigma'^{-1/2}{\rm Abs}(\Sigma-\Sigma')\Sigma'^{-1/2}\rVert(t,x,\sigma)<1$ for any $(t,x,\sigma)\in [0,T]\times \mathcal{O}\times \bar{\Lambda}$,
where $\Sigma'={\rm diag}(([\Sigma]_{kk})_{k=1}^\gamma)$.
\begin{discuss}
{\colorr $\Sigma$をパラメトライズしたらANNの時少し下を浮かせればできるし，そうでなくても$b'=\sqrt{\epsilon+bb^\top}$を考えればできる．
$b_{\dagger}b_{\dagger}^{\top}$の正定値性はおそらくいらない}
\end{discuss}
\\
3. $\mu_t$ is locally bounded; That is, there exists a monotonically increasing sequence $\{T_l\}_{l\in\mathbb{N}}$ of stopping times such that $\lim_{l\to\infty}T_l=T$ almost surely
and $\{\mu_{t\wedge T_l}\}_{0\leq t\leq T}$ is bounded for each $l$.\\
4. $\sup_{n,k,i}E[(\epsilon^{n,k}_i)^q]<\infty$ for any $q>0$ and $\sup_{0\leq s<t\leq T}(E[|X_t-X_s|^2]/|t-s|)<\infty$. \\ 
5. $P[\min_{1\leq m\leq \ell_n}\#\{j;T^{n,k}_j\in [s_{m-1},s_m)\}\geq 1]\to 1$ as $n\to\infty$ and 
\begin{equation*}
\bigg\{\ell_n^{q/2}\max_{m,k}\bigg(\#\{j;T^{n,k}_j\in [s_{m-1},s_m)\}^{-q}E_{\Pi}\bigg[\bigg|\sum_{j;T^{n,k}_j\in [s_{m-1},s_m)}\eta^{n,k}_j\bigg|^q\bigg]\bigg)\bigg\}_n
\end{equation*}
is tight for any $q>0$.
\begin{discuss2}
{\colorg $q$の条件変えたがこれでいいかリバイズ時にチェック}
\end{discuss2}
\begin{discuss}
{\colorr 過去のバージョンでは$q=2$だけでやっていたが，証明ではOgi18のLemma 4.4の$\hat{\Psi}_{2,n}$の評価をやっており，ソボレフ型で示しているのでおそらく任意の$q$に対して条件が必要だろう．} 
\end{discuss}
\\
6. There exist progressively measurable processes $\{b^{(j)}_t\}_{0\leq t\leq T,0\leq j\leq 1}$ and $\{\hat{b}^{(j)}_t\}_{0\leq t\leq T, 0\leq j\leq 1}$ such that
$\sup_{t\in[0,T]} E[|b^{(j)}_t|^q\vee |\hat{b}^{(j)}_t|^q]<\infty$, 
\begin{equation*}
\sup_{s,t\in [0,T];s< t}\left(E[|b^{(j)}_t-b^{(j)}_s|^q\vee |\hat{b}^{(j)}_t-\hat{b}^{(j)}_s|^q]^{1/q}|t-s|^{-1/2}\right)<\infty
\end{equation*}
for $0\leq j\leq 1$ and any $q>0$, and
\begin{discuss}
{\colorr 期待値をとらないものは暴れているからこの評価の局所化は難しそう．停止時刻を付けた条件にすることはできそうだが，
実際には$b_{t,\dagger}$や$X_t$を局所化してＭＬＥの収束がなりたてばいいから問題は起きないだろう．}
\end{discuss}
\begin{equation*}
b_{t,\dagger}=b_{0,\dagger}+\int^t_0b^{(0)}_sds+\int^t_0b^{(1)}_sdW_s, \quad b^{(1)}_t=b^{(1)}_0+\int^t_0\hat{b}^{(0)}_sds+\int^t_0\hat{b}^{(1)}_sdW_s
\end{equation*}
for $t\in [0,T]$. 
\begin{discuss}
{\colorg 次：$q$は任意ではなくてある$q$が存在すればいいような気もするがどうか}
\end{discuss}
\begin{discuss}
{\colorr 
モーメント条件から確率積分が定義できることもわかる.
$X_t$で$Y_t$と独立なブラウン運動の成分を定義できるように, $W$の次元を一般にしておく. $\tilde{W}$として定義すると, $S,T$との独立性などの議論がごちゃごちゃしやすい.
}
\end{discuss}
\end{description}
\begin{discuss}
{\colorr 基本的にはミススぺにして, identifiabilityをなくしたという状況か.}
\end{discuss}

Some of these conditions are standard conditions and easy to check. The market microstructure noise $\eta_j^{n,k}$ for $X$ requires
point 5 of (A1). 
Roughly speaking, this condition is satisfied when the sum of $\eta^{n,k}_j$ has order equivalent to the square root of the number of $T^{n,k}_j$ in $[s_{m-1},s_m)$.
This is satisfied if the sampling frequency of $\{T^{n,k}_j\}$ is of order $b_n$ and $\eta_j^{n,k}$ satisfies certain independence, martingale, or mixing conditions.
Decomposition of $X$ in point 6 of (A1) is used when we estimate the difference between $E_m[Z_mZ_m^\top]$ and $\Sigma_{s_{m-1},\dagger}$, which appears in asymptotic representation of $H_n$.
To satisfy the condition $\lVert \Sigma'^{-1/2}{\rm Abs}(\Sigma-\Sigma')\Sigma'^{-1/2}\rVert<1$, it is sufficient that $\Sigma$ is symmetric and positive definite and $\gamma\leq 2$.
This condition is restrictive when the dimension $\gamma$ of $Y$ is large. We need this condition to obtain the expansion of $S_m^{-1}$ in Lemma~\ref{invS-eq-lemma},
which is repeatedly used in the proof of the main results.

We assume some additional conditions for the sampling scheme.
Let $r_n=\max_{i,k}|S^{n,k}_i-S^{n,k}_{i-1}|$ and $\underbar{r}_n=\min_{i,k}|S^{n,k}_i-S^{n,k}_{i-1}|$.
For $\eta \in (0,1/2)$, let $\mathcal{S}_\eta$ be the set of all sequences $\{[s'_{n,l},s''_{n,l})\}_{n\in\mathbb{N},1\leq l\leq L_n}$ of intervals on $[0,T]$ satisfying
$\{L_n\}_{n\in\mathbb{N}}\subset \mathbb{N}$, $[s'_{n,l_1},s''_{n,l_1})\cap [s'_{n,l_2},s''_{n,l_2})=\emptyset$ for $n,l_1\neq l_2$, $\inf_{n,l}(b_n^{1-\eta}(s''_{n,l}-s'_{n,l}))>0$
and $\sup_{n,l}(b_n^{1-\eta}(s''_{n,l}-s'_{n,l}))<\infty$.
\begin{description}
\item{[A2]} There exist $\eta \in (0,1/2)$, $\kappa>0$, $\dot{\eta}\in (0,1]$, and positive-valued stochastic processes $\{a_t^j\}_{t\in[0,T],1\leq j\leq \gamma}$ such that 
$\sup_{t\neq s}(|a^j_t-a^j_s|/|t-s|^{\dot{\eta}})<\infty$ almost surely, $b_n^{-1/2+\kappa}k_n(b_n^{-1}k_n)^{\dot{\eta}}\to 0$, and
\begin{equation}\label{A2-conv}
k_nb_n^{-\frac{1}{2}+\kappa}\max_{1\leq l\leq L_n}\bigg|b_n^{-1}(s''_{n,l}-s'_{n,l})^{-1}\#\{i;[S^{n,j}_{i-1},S^{n,j}_i)\subset [s'_{n,l},s''_{n,l})\}-a^j_{s'_{n,l}}\bigg|
\end{equation}
converges to zero in probability as $n\to\infty$ for $\{[s'_{n,l},s''_{n,l})\}_{1\leq l\leq L_n,n\in\mathbb{N}}\in \mathcal{S}_{\eta}$ and $1\leq j\leq \gamma$.  
Moreover, $(r_nb_n^{1-\epsilon}) \vee (b_n^{-1-\epsilon}\underbar{r}_n^{-1})\overset{P}\to 0$ for any $\epsilon>0$.
\end{description}

Condition [A2] is about the law of large numbers of $s_i^{n,k}$ in each local interval $[s'_{n,l},s''_{n,l})$.
This conditions ensures that the intensity of the observation count converges to some intensity $a_t^j$ with order $(k_nb_n^{-1/2+\kappa})^{-1}$.
For covariation estimators under the existence of market microstructure noise,
Christensen, Kinnebrock, and Podolskij~\cite{chr-etal10} studied a pre-averaging method, 
and Barndorff-Nielsen et al.~\cite{bar-etal11} studied a kernel-based method.
Though our quasi-likelihood approach does not use such data-averaging methods,
the proofs of Lemmas 5.1 and 5.2 in Ogihara~\cite{ogi18} show that we can replace the length of observation intervals in the covariance matrix
with their averages. Doing so makes it important to identify the limit of local sampling counts.
One case where [A2] holds is when observation times are generated by mixing processes as seen below.
\begin{example}
Let $\{N^k_t\}_{t\geq 0}$ be an exponential $\alpha$-mixing point process with stationary increments for $1\leq k\leq \gamma$.
Assume that $E[|N^k_1|^q]<\infty$ for any $q>0$ and $1\leq k\leq \gamma$.
Set $S^{n,k}_i=\inf\{t\geq 0 ; N^k_{b_nt}\geq i\}$. Then  [A2] is satisfied with $a^j_t\equiv E[N^j_1]$ (constants) 
Moreover, [B2] (defined later) is satisfied if, additionally, $k_nb_n^{-3/5+\epsilon}\to 0$ for some $\epsilon>0$.
See Example 2.1 in Ogihara~\cite{ogi18} for the details.
\end{example}

\begin{discuss}
{\colorr
まず$E[N^j_t]=tE[N^j_1]$. (∵$t$が有理数ならOK. 右連続性から$\{t_n\}_n\subset \mathcal{Q}, t_n\to t$に対し, $E[N^j_{t_n}]=t_nE[N^j_1]$なのでOK)

よって,(\ref{A2-conv})の左辺の絶対値の中は
\begin{eqnarray}
b_n^{-1}(s''_{n,l}-s'_{n,l})^{-1}(N^j_{b_ns''_{n,l}}-N^j_{b_ns'_{n,l}})-E[N^j_1]
=b_n^{-1}(s''_{n,l}-s'_{n,l})^{-1}(N^j_{b_ns''_{n,l}}-N^j_{b_ns'_{n,l}}-E[N^j_{b_ns''_{n,l}}-N^j_{b_ns'_{n,l}}]) \nonumber
\end{eqnarray}
となるのでRosenthal-type inequalitiesから
\begin{eqnarray}
&&E\bigg[\bigg(k_nb_n^{-1/2}\max_{1\leq l\leq L_n}\bigg|b_n^{-1}(s''_{n,l}-s'_{n,l})^{-1}\sum_{i;[S^{n,j}_{i-1},S^{n,j}_i)\subset (s'_{n,l},s''_{n,l})}1-a^j_{s'_{n,l}}\bigg|\bigg)^q\bigg] \nonumber \\
&\leq &k_n^qb_n^{-q/2}\sum_lE\bigg[\bigg|\cdot \bigg|^q\bigg]
\leq k_n^qb_n^{-q/2}L_n\max_l(b_n^{-1}(s''_{n,l}-s'_{n,l})^{-1})^qE\bigg[\bigg|\cdot \bigg|^q\bigg]. \nonumber
\end{eqnarray}
\begin{eqnarray}
E\bigg[\bigg|\cdot \bigg|^q\bigg]
&\leq& C_q\bigg\{\bigg(\sum_{i=1}^{[b_n^{\eta}]}\int^1_0(\alpha^{-1}(u)\wedge [b_n^{\eta}])^{q-1}Q^q_{X'_i}(u)du\bigg)
\vee \bigg(\sum_{i=1}^{[b_n^{\eta}]}\int^1_0(\alpha^{-1}(u)\wedge [b_n^{\eta}])Q^2_{X'_i}(u)du\bigg)^{q/2}\bigg\} \nonumber \\
&\leq &C_q[b_n^{\eta}]^{q/2}\int^1_0(\alpha^{-1}(u))^{q-1}Q^q_{X'_i}(u)du\leq C_qb_n^{\eta q/2}
\bigg(\int^1_0(\alpha^{-1}(u))^{2q-2}du\bigg)^{1/2}\bigg(\int^1_0Q^{2q}_{X'_i}(u)du\bigg)^{1/2}. \nonumber
\end{eqnarray}
ただし, $q\geq 1$, $\alpha^{-1}(u)=\sum_{k=0}^{\infty}1_{\{\alpha_k> u\} }$, $Q_{X'}(s)=\inf\{t\geq 0, P[|X'|>t]\leq s\}$．
\begin{equation*}
Q_{X}^q(s)=\inf\{t>0;P[|X|>t^{1/q}]\leq s\}=\sup\{t>0; P[|X|>t^{1/q}]>s \}=\int^{\infty}_01_{\{P[|X|>t^{1/q}]>s\}}dt
\end{equation*}
より，
\begin{equation*}
\int^1_0Q^{2q}_{X'_i}(u)du=\int^{\infty}_0P[|X'_i|>t^{1/(2q)}]dt=\int^{\infty}_0P[|X'_i|^{2q}>t]dt=E[|X'_i|^{2q}], 
\end{equation*}
またOgiYos2014のProp6の証明より，$\int^1_0(\alpha^{-1}(u))^{q'}du\leq q'\sum_{k=0}^{\infty}(k+1)^{q'-1}\alpha_k^n$なので，
与式$=O((k_nb_n^{-1/2})^qb_n^{1-\eta}b_n^{-\eta q/2})\to 0$.
(A2)を示すだけの場合でも最後の$b_n^{1-\eta}$を消せるだけ$q$を大きくする必要があるから$N^k_1$のモーメント条件は必要．

(B2)の証明も上の式を使えばできる．
}
\end{discuss}

Here, we study results related to consistency under Conditions [A1] and [A2].
Let $\tilde{a}_t=(\tilde{a}_t^1,\cdots,\tilde{a}_t^\gamma)$, $\tilde{a}^j_t=a_t^j/v_{j,\ast} \ (1\leq j\leq \gamma)$.
Let $\Sigma_\dagger=(\Sigma_{t,\dagger})_{0\leq t\leq T}$, $\Sigma_t(\sigma)=\Sigma(t,X_t,\sigma)$, $\Sigma(\sigma)=(\Sigma_t(\sigma))_{0\leq t\leq T}$,
and $\mathcal{D}(t,A)=([A_t]_{ij}(a^i_ta^j_t)^{1/2}v_{i,\ast}^{-1/2}v_{j,\ast}^{-1/2})_{1\leq i,j\leq \gamma}$
for $A=(A_t)_{t\geq 0}\subset \mathbb{R}^\gamma\otimes \mathbb{R}^\gamma$. Moreover, we define
\begin{eqnarray}\label{D-def}
&& \\
D(\Sigma_1,\Sigma_2)&=&\int^T_0\bigg\{\frac{1}{4}{\rm tr}((\mathcal{D}(t,\Sigma_2)-\mathcal{D}(t,\Sigma_1))\mathcal{D}(t,\Sigma_1)^{-1/2}) \nonumber \\
&&\quad \quad -\frac{1}{2}{\rm tr}(\mathcal{D}(t,\Sigma_2)^{1/2})+\frac{1}{2}{\rm tr}(\mathcal{D}(t,\Sigma_1)^{1/2})\bigg\}dt \nonumber \\
&=&\frac{1}{4}\int^T_0{\rm tr}((\mathcal{D}(t,\Sigma_2)^{1/2}-\mathcal{D}(t,\Sigma_1)^{1/2})^2\mathcal{D}(t,\Sigma_1)^{-1/2})dt \nonumber
\end{eqnarray}
for c\`adl\`ag functions $\Sigma_j=(\Sigma_{j,t})_{0\leq t\leq T} \ (j=1,2)$
with positive definite matrices $(\Sigma_{1,t})_t$ and nonnegative definite matrices $(\Sigma_{2,t})_t$. 

\begin{discuss}
{\colorr $\mathcal{Y}(b)=-D(b,b_{\dagger})$.
ここは$\mathcal{Y}$は$\Sigma(\sigma)=\Sigma_\dagger$の時最大，$D$は$L^2$キョリと同値で最小になっているから符号はこのままでいい．}
\end{discuss}
\begin{theorem}\label{consistency-theorem}
Assume [A1], [A2] and [V]. Then (\ref{D-conv}) holds as $n\to \infty$.
\end{theorem}

The function $D$ is deduced as a limit of the quasi-log-likelihood function.
It is noteworthy that $D$ is related to the Kullback--Leibler divergence in the following sense.
Let $\ell_n(\theta)$ be the log-likelihood function of a sequence $(X_i)_{i=1}^n$ of independent identically distributed random variables, 
where the probability density function of $X_1$ is $p(x,\theta_0)$ for a parameter $\theta$.
Then, under suitable regularity conditions, we obtain
\begin{equation*}
n^{-1}(\ell_n(\theta)-\ell_n(\theta'))\overset{P}\to {\rm KL}(p(\theta),p(\theta_0))-{\rm KL}(p(\theta'),p(\theta_0)),
\end{equation*}
where ${\rm KL}(p,q)$ is the Kullback--Leibler divergence of densities $p$ and $q$.
In contrast, (\ref{Hn-limit-eq}) later yields
\begin{equation}\label{Hn-conv-toD}
b_n^{-1/2}(H_n(\sigma,\hat{v}_n)-H_n(\sigma',\hat{v}_n))\overset{P}\to D(\Sigma(\sigma),\Sigma_\dagger)-D(\Sigma(\sigma'),\Sigma_\dagger).
\end{equation}
We can therefore regard $D$ as an extension of the Kullback--Leibler divergence in some sense.

Moreover, under boundedness of $a^j_t+(a^j_t)^{-1}$, $\lVert \Sigma_t(\sigma)\rVert$, $\lVert \Sigma_{t,\dagger}\rVert$, and $\lVert \Sigma_t^{-1}(\sigma) \rVert$, there exist positive constants $C_1$ and $C_2$ such that
\begin{equation}\label{L2-equivalence}
C_1\int^T_0|\Sigma(t,X_t,\sigma)-\Sigma_{t,\dagger}|^2dt\leq D(\Sigma(\sigma),\Sigma_\dagger)\leq C_2\int^T_0|\Sigma(t,X_t,\sigma)-\Sigma_{t,\dagger}|^2dt,
\end{equation}
where $C_1$ and $C_2$ depend on only the upperbounds of $a^j_t+(a^j_t)^{-1}$ and $\lVert \Sigma^{-1}(\sigma) \rVert$.
A proof is given in the appendix.
From the above, $D$ is {\it equivalent} to an $L^2$ norm. 

\begin{discuss}
{\colorr 一様非退化性等を仮定しなくても$\sup_{j,t}(|b^j_t|\vee |b^j_t|^{-1}),\inf_{j,t}a^j_t,v_{\ast}$のみに依存する確率変数$C_1,C_2$に対して上の式が成り立つ．
下からの評価はIdentifiabilityの証明が引き続き使えることからわかる. 上からの評価はLem A.6の赤字を使えばＯＫ．}
\end{discuss}

\begin{remark}\label{univ-approx-rem}
In the setting of Section~\ref{nn-example-section},
assume the conditions of Theorem~\ref{consistency-theorem} and (\ref{L2-equivalence}) and further assume that 
$\Sigma_{t,\dagger}\equiv\Sigma_\dagger(t,X_t)$ for some $X=(X_t)_{t\in [0,T]}$ and some continuous function $\Sigma_\dagger(t,x)$.
Let $\epsilon,\delta>0$. For any compact set $S\subset \mathcal{O}$, we have
\begin{equation}
\min_\beta \sup_{t\in [0,T],x\in S}|\Sigma(t,x,\beta)-\Sigma_\dagger(t,x)|<\sqrt{\frac{\delta}{2T}}
\end{equation}
for $K=2$ and sufficiently large $L_1$ ($K,L_k$ are defined in Section~\ref{nn-example-section}).
This property is called universal approximation property. See Mhaskar and Micchelli~\cite{mha-mic92}, Pinkus~\cite{pin99}, and Sonoda and Murata~\cite{son-mur17} for the details.

Therefore, applying Theorem~\ref{consistency-theorem} and (\ref{L2-equivalence}) yields
\begin{equation*}
P\bigg[\int^T_0|\Sigma(t,X_t,\hat{\beta}_n)-\Sigma_{t,\dagger}|^2dt\geq \delta\bigg]<\epsilon
\end{equation*}
for sufficiently large $n$. 
\end{remark}
\begin{remark}
In the correctly specified setting of Ogihara~\cite{ogi18}, $\mathcal{Y}_1(\sigma)=D(\Sigma(\sigma),\Sigma(\sigma_\ast))$ holds for $\mathcal{Y}_1(\sigma)$ in Section 2.2 of \cite{ogi18}.
The representation (2.8) of $\mathcal{Y}_1(\sigma)$ is obtained by calculating the elements of $\mathcal{D}(\Sigma(\sigma))^{1/2}$ for $\gamma=2$.
\end{remark}

\subsection{Optimal rate of convergence}\label{opt-conv-section}

In this section, we study the optimal rate of convergence.
Ogihara~\cite{ogi18} showed that local asymptotic normality holds for a correctly specified model having nonrandom diffusion coefficients,
with the optimal rate of convergence equal to $b_n^{1/4}$ for estimators of the parameter in the diffusion coefficients,
and further that the maximum-likelihood-type estimator $\hat{\sigma}_n$ attains the optimal rate.
In contrast, we will see that $\hat{\sigma}_n$ cannot attain the rate $b_n^{1/4}$ in the misspecified setting
due to an asymptotic bias term.
We can attain the optimal convergence rate if we construct a maximum-likelihood-type estimator $\check{\sigma}_n$ by using a bias-modified quasi-log-likelihood function.

To obtain the optimal convergence rate, we need stronger versions of [A1] and [A2].
We give those here, calling them [B1] and [B2], respectively.
\begin{description}
\item{[B1]} [A1] is satisfied, $\sup_{0\leq s<t\leq T}(E[E[X_t-X_s|\mathcal{F}_s]^2]/|t-s|^2)<\infty$, and $\partial_\sigma^l\partial_x\Sigma(t,x,\sigma)$ exists and is continuous 
on $[0,T]\times \mathcal{O}\times \bar{\Lambda}$ for $l\in\{0,1\}$.
Moreover, there exists a locally bounded function $L(x,y)$ such that
\begin{equation*}
|\partial_\sigma^l\partial_x\Sigma(t,x,\sigma)-\partial_\sigma^l\partial_x\Sigma(t,y,\sigma)|\leq L(x,y)|x-y|
\end{equation*}
for any $t\in[0,T]$, $x,y\in\mathcal{O}$, $\sigma\in\Lambda$, and $l\in\{0,1\}$.\\
\end{description}
\begin{discuss}
{\colorr Lemma \ref{HtoTildeH-lemma}で$\partial_x\Sigma$が必要．$\partial_x\Sigma$のりぷしっつも．Proposition \ref{Hn-diff-prop}の証明の最後は$\Sigma(t,\cdot)-\Sigma(s,\cdot)$はりぷしっつを使えばいい．
$\partial_x\Sigma$のりぷしっつが必要．Theorem \ref{optConvTheorem}の証明のソボレフで$\partial_\sigma\Sigma$が必要．
[B1]3.4.はLemma \ref{ZtildeZ-diff-lemma}とProposition \ref{Hn-diff-prop}の証明の最後と(\ref{log-likelihood-eq-lemma-eq4})の評価だけで使う．いずれも二乗評価でOK.
$E[X_t-X_s|\mathcal{F}_s]$の評価はProposition \ref{Hn-diff-prop}の証明の最後で使う．

基本的には$b_t$と$b_{t,\dagger}$に分けてなめらかさをなくして, identifiabilityをなくしたという状況か.
収束レートを出すだけなら$b$の一階微分だけあれば十分. $X_t$の条件はLemma 4.3ではそもそも$b$が出てこないので不要．$bb^{\top}(t,X_t)$の下からの評価があれば適用できる．
Lemma 4.4に対応する評価で$X_t\to X_{s_{m-1}}$への誤差評価があるからその分の仮定は必要．

Lemma \ref{HtoTildeH-lemma}を示すときに$\Sigma$の滑らかさは必要になるので仮定する．ReLUは無理そうだが，softplusはいける．
一致性はReLUでもいけるからとりあえずそれでいい．
ここで$b$の滑らかさを使わずに示せるならProposition \ref{Hn-diff-prop}の最後の$t\in [s_{m-1},s_m)$と$s_{m-1}$の差の評価は
$\Sigma$が一階のソボレフ空間に入っていて，$\Sigma_l\to \Sigma$に対し，
\begin{equation*}
\sup_{\sigma,m}E[\partial_\sigma^j\Sigma_l(t,X_t,\sigma)-\partial_\sigma^j\Sigma_l(s,X_s,\sigma)|\mathcal{F}_s]=O_p(t-s)
\end{equation*}
を仮定すればよい．ReLUでこれを示せるかわからないが伊藤の公式の二階微分がほぼつぶれているので行けるかもしれない．

}
\end{discuss}
\begin{description}
\item{[B2]} There exist positive valued stochastic processes $\{a^j_t\}_{t\in[0,T],1\leq j\leq \gamma}$ such that for any $q>0$ and $\epsilon>0$,
we have $(r_nb_n^{1-\epsilon})\vee (b_n^{-1-\epsilon}\underbar{r}_n^{-1})\overset{P}\to 0$, $\ell_nb_n^{-3/7-\epsilon}\to \infty$,
$E[\sup_{t\neq s}(|a^j_t-a^j_s|^q/|t-s|^q)]<\infty$, and $E[\sup_{j,t}(|a^j_t|+1/|a^j_t|)^q]<\infty$, we have that
\begin{equation*}
\sup_n\sup_{[s'_n,s''_n)}E\bigg[\bigg(\sqrt{b_n(s''_n-s'_n)}\bigg(\frac{\#\{i;[S^{n,j}_{i-1},S_i^{n,j})\subset [s'_n,s''_n)\}}{b_n(s''_n-s'_n)}-a_{s'_n}^j\bigg)\bigg)^q\bigg]
\end{equation*}
is finite for any $1\leq j\leq \gamma$,
where the second supremum is taken over all sequences $\{s'_n\}_n,\{s''_n\}_n\subset [0,T]$ such that $s'_n<s''_n$ and $\sup_n(\ell_n(s''_n-s'_n))<\infty$.
\end{description}
\begin{discuss}
{\colorr 期待値の中にmaxを入れる必要がないのは，後で隙間を作って入れればいいから．
$k_n\leq b_n^{3/5-\epsilon}$はProof of Thm \ref{optConvTheorem}の$b_n^{-1/4}\Lambda_2(\sigma_\ast)$の評価で必要．
これがなければ$k_n\leq b_n^{2/3-\epsilon}$で十分か．
}
\end{discuss}
We can see that [A2] is satisfied whenever [B2] is.
\begin{discuss}
{\colorr $(r_nb_n^{1-\epsilon})\vee (b_n^{-1-\epsilon}\underbar{r}_n^{-1})\overset{P}\to 0$はOK. $\dot{\eta}=1$, $\eta<1/2$ such that ある$\kappa>0$があって$k_nb_n^{-1/2+2\kappa-\eta/2}\to 0$とする.
$b_n^{-1/2}k_nb_n^{-1}k_n\to 0$はOK. $T(s'_n,s''_n)=\big\{b_n^{-1}(s''_n-s'_n)^{-1}\#\{i;[S^{n,j}_{i-1},S_i^{n,j})\subset [s'_n,s''_n)\}-a_{s'_n}^j\big\}$とすると
\begin{equation*}
E[\max_{1\leq l\leq L_n}|T(s'_{n,l},s''_{n,l})|]\leq \bigg(\sum_lE[|T(s'_{n,l},s''_{n,l})|^q]\bigg)^{1/q}
\leq C\sup_n\sup_{[s'_n,s''_n)}E[|T(s'_n,s''_n)|^q]^{1/q}l_n^{1/q}(b_nb_n^{-1+\eta})^{-1/2}
\end{equation*}
より，$\ell_n^{1/q}\leq b_n^\kappa$となる$q>0$を取れば
\begin{equation*}
k_nb_n^{-1/2+\kappa}\max_{1\leq l\leq L_n}|T(s'_{n,l},s''_{n,l})|=O_p(k_nb_n^{-1/2+2\kappa-\eta/2})\overset{P}\to0.
\end{equation*}
}
\end{discuss}

Here, we see the bias of the quasi-log-likelihood function $H_n$.
Let $A(a)={\rm diag}((a_j^{1/2})_j)$ for $a_j\geq 0$. Let
\begin{eqnarray}
E_m(a,B,C,v)&=&{\rm tr}({\bf S}_m^{-1}(B,v){\bf S}_m(C,v)) \nonumber \\
&&-(1/2)Tb_n^{1/2}\ell_n^{-1}{\rm tr}(A(a)(C-B)A(a)(A(a)BA(a))^{-1/2}), \nonumber \\
F_m(a,B,v)&=&\log \det {\bf S}_m(B,v)-Tb_n^{1/2}\ell_n^{-1}{\rm tr}((A(a)BA(a))^{1/2}), \nonumber \\
G_m(a_1,a_2,B,C,v)&=&E_m(a_1,a_2,B,C,v)+F_m(a_1,a_2,B,v) \nonumber
\end{eqnarray}
for a $\gamma\times \gamma$ symmetric, positive definite matrix $B=(B_{ij})_{ij}$ and a $\gamma\times \gamma$ symmetric matrix $C=(C_{ij})_{ij}$.
\begin{discuss}
{\colorr 
$E_m$の第二項の符号は，$-{\rm tr}()/2\approx 1/2 \times \mbox{第二項}$とすればOgihara 2017と合うのでプラスでいい．$F_m$の第二項は$-1/2\log \approx -1/2\mbox{第二項}$となるのでマイナス．

最終形は$b_n^{1/2}\int^{s_m}_{s_{m-1}}\cdots dt$だが，これだと$B_{m,n}$のマルチンゲール評価が使えないのでそれを使った後最終形に持っていく．
$\tilde{F}_m(X,v_1)-\tilde{F}_m(X,v_2)$の評価では$b$の部分の$v_{\ast}$は消えるのでこの定義で良い．
}
\end{discuss}

Let $\tilde{\Sigma}_m=\tilde{\Sigma}_m(\sigma)=\Sigma(s_{m-1},X_{s_{m-1}},\sigma)$, $\tilde{\Sigma}_{m,\dagger}=\Sigma_{s_{m-1},\dagger}$ and $\tilde{S}_{m,\dagger}={\bf S}_m(\tilde{\Sigma}_{m,\dagger},v_\ast)$. 
We write
\begin{equation*}
\Delta_n(\sigma_1,\sigma_2):=-\frac{1}{2}\sum_m{\rm tr}((\tilde{S}_m^{-1}(\sigma_1)-\tilde{S}_m^{-1}(\sigma_2))(\tilde{Z}_m\tilde{Z}_m^{\top}-\tilde{S}_{m,\dagger}))
\end{equation*}
for $\sigma_1,\sigma_2\in\bar{\Lambda}$.

\begin{proposition}\label{Hn-diff-prop}
Assume [B1], [B2], and [V]. Let $\{\sigma_{j,n}\}_{j=1,2,n\in \mathbb{N}}$ be $\bar{\Lambda}$-valued random variables. Then
\begin{eqnarray}\label{qLF-diff-est}
&&b_n^{-1/4}(H_n(\sigma_{1,n},\hat{v}_n)-H_n(\sigma_{2,n},\hat{v}_n)) \\
&&\quad =b_n^{-1/4}\Delta_n(\sigma_{1,n},\sigma_{2,n})-b_n^{1/4}(D(\Sigma(\sigma_{1,n}),\Sigma_\dagger)-D(\Sigma(\sigma_{2,n}),\Sigma_\dagger)) \nonumber \\
&&\quad \quad +\frac{1}{2}b_n^{-1/4}\sum_{j=1}^2(-1)^j\sum_mG_m(\tilde{a}_{s_{m-1}},\tilde{\Sigma}_m(\sigma_{j,n}),\tilde{\Sigma}_{m,\dagger},v_{\ast})+o_p(1). \nonumber
\end{eqnarray}
\end{proposition}

\begin{discuss}
{\colorr 
\begin{eqnarray}
H_n(\sigma,\hat{v}_n)&\approx& -\frac{1}{2}{\rm tr}({\bf S}_m^{-1}(B,v_\ast)(ZZ^\top -{\bf S}_m(C)))
-\frac{1}{2}{\rm tr}({\bf S}_m^{-1}(B,v_\ast){\bf S}_m(C,v_\ast)) \nonumber \\
&&-\frac{1}{2}\log\det {\bf S}_m(B,v_\ast) \nonumber
\end{eqnarray}
より
\begin{eqnarray}
H_n(\sigma_{1,n},\hat{v}_n)-H_n(\sigma_{2,n},\hat{v}_n)
&\approx &\Delta_n -(D(\Sigma(\sigma_{1,n}),\Sigma_\dagger)-D(\Sigma(\sigma_{2,n}),\Sigma_\dagger)) \nonumber \\
&&-\frac{1}{2}G_m(\Sigma(\sigma_{1,n}))+\frac{1}{2}G_m(\Sigma(\sigma_{2,n})) \nonumber
\end{eqnarray}
より符号はよさそう．
}
\end{discuss}

The third term in the right-hand side of (\ref{qLF-diff-est}) is a bias term, which does not appear in the correctly specified model.
For the correctly specified case (\ref{specified-condition}), Ogihara~\cite{ogi18} showed 
$b_n^{1/4}(\hat{\sigma}_n-\sigma_\ast)=O_p(1)$ under suitable conditions.
However, in a general misspecified model, we cannot obtain this relation due to the bias term $G_m$.
For example, let $\gamma=1$, $S^{n,j}_i\equiv i/n$ for $0\leq i\leq n$, $t_k=k\pi/(k^j_m+1)$, let $\Sigma$ be smooth and $\sigma_\ast$ be a $\bar{\Lambda}$-valued random variable
such that $D(\Sigma(\sigma_\ast),\Sigma_\dagger)=\min_{\sigma\in \bar{\Lambda}}D(\Sigma(\sigma),\Sigma_\dagger)$ and $\hat{\sigma}_n\overset{P}\to\sigma_\ast$.
\begin{discuss}
{\colorr $\hat{\sigma}_n\overset{P}\to\sigma\ast$は常に取れるわけではないが，$\sigma_\ast$が一意ならOKか．}
\end{discuss}
Then, by differentiating both sides of the equation in the above proposition, we obtain, roughly,
\begin{equation*}
b_n^{-1/4}\partial_\sigma H_n(\sigma_\ast)=-\frac{b_n^{-1/4}}{2}\sum_m\partial_\sigma G_m(\tilde{a}_{s_{m-1}},\tilde{\Sigma}_m(\sigma_\ast),\tilde{\Sigma}_{m,\dagger},v_\ast)+O_p(1)
\end{equation*}
since $\{\sup_{\sigma,\sigma'}|b_n^{-1/4}\partial_\sigma\Delta_n(\sigma,\sigma')|\}_n$ is tight as seen in (\ref{Delta-est}).
On the other hand, we also obtain
\begin{eqnarray}
b_n^{-1/4}\partial_\sigma H_n(\sigma_\ast)&=&b_n^{-1/4}(\partial_\sigma H_n(\sigma_\ast)-\partial_\sigma H_n(\hat{\sigma}_n)) \nonumber \\
&\approx &(-b_n^{-1/2}\partial_\sigma^2H_n(\sigma_\ast))b_n^{1/4}(\hat{\sigma}_n-\sigma_\ast)
\approx \Gamma_2 (b_n^{1/4}(\hat{\sigma}_n-\sigma_\ast)) \nonumber
\end{eqnarray}
by (\ref{Gamma-conv}), {\colord where $\Gamma_2=\partial_\sigma^2D(\Sigma(\sigma),\Sigma_\dagger)|_{\sigma=\sigma_\ast}$.}
\begin{discuss}
{\colorr $\sigma_\ast$が一意でないときは，$-b_n^{-1/2}\int^{\sigma_{1,\ast}}_{\sigma_{2,\ast}}\partial_\sigma^2H_n(\sigma_t)dt$がつぶれて$b_n(\hat{\sigma}_n-\sigma_\ast)$のオーダーが出ないだろう．}
\end{discuss}

We show
\begin{eqnarray}\label{bias-limit}
&& \\
\Gamma_2 (b_n^{1/4}(\hat{\sigma}_n-\sigma_\ast))
&=&\frac{b_n^{-9/4}}{2}\sum_m(\tilde{\Sigma}_{m,\dagger}-\tilde{\Sigma}_m(\sigma_\ast))\partial_\sigma \tilde{\Sigma}_m(\sigma_\ast) \nonumber \\
&&\times \sum_{k=1}^{k_m}\int^{t_k}_{t_{k-1}}\int^{t_k}_x\frac{-4v_\ast\sin y(t_k-t_{k-1})^{-1}dydx}{(\tilde{\Sigma}_m(\sigma_\ast)b_n^{-1}+2v_\ast (1-\cos y))^3}+o_p(1) \nonumber
\end{eqnarray}
by a calculation in Section~\ref{Section3-proof-section}.

Since $\sin y\geq 2y/\pi$ for $0\leq y\leq \pi/2$ and $1-\cos x\leq x^2/2$ for $x\geq 0$, we obtain
\begin{eqnarray}
&&b_n^{-9/4}\sum_m\sum_{k=1}^{k_m}\int^{t_k}_{t_{k-1}}\int^{t_k}_x\frac{-4v\sin y(t_k-t_{k-1})^{-1}dydx}{(\tilde{\Sigma}_m(\sigma_\ast)b_n^{-1}+2v_\ast (1-\cos y))^3} \nonumber \\
&&\quad \leq -\frac{8v_\ast}{\pi}b_n^{-9/4}\sum_m\sum_{k=1}^{[k_m/2]}\int^{t_k}_{t_{k-1}}\int^{t_k}_x\frac{\sin ydydx}{t_k-t_{k-1}}(\tilde{\Sigma}_m(\sigma_\ast)b_n^{-1}+v_\ast t_k^2)^{-3} \nonumber \\
&&\quad \leq -\frac{8v_\ast}{\pi}b_n^{-9/4}\sum_m\sum_{k=1}^{[k_mb_n^{-1/2}]}\frac{t_{k-1}}{2}\frac{\pi}{k_m+1}\frac{b_n^3}{(\tilde{\Sigma}_m(\sigma_\ast)+v_\ast \pi^2)^3} \nonumber \\
&&\quad \leq -\sum_m\frac{2\pi v_\ast b_n^{3/4}[k_mb_n^{-1/2}]^2}{(\tilde{\Sigma}_m(\sigma_\ast)+v_\ast \pi^2)^3(k_m+1)^2} \nonumber \\
&&\quad \leq  -\frac{v_\ast b_n^{-1/4}\ell_n}{(\max_m\tilde{\Sigma}_m(\sigma_\ast)+v_\ast \pi^2)^3}\neq O_p(1). \nonumber
\end{eqnarray}
\begin{discuss}
{\colorr 
\begin{eqnarray}
\int^{t_k}_{t_{k-1}}\int^{t_k}_xydydx&=&\int^{t_k}_{t_{k-1}}\bigg(\frac{t_k^2}{2}-\frac{x^2}{2}\bigg)dx=\bigg[\frac{t_k^2x}{2}-\frac{x^3}{6}\bigg]^{t_k}_{t_{k-1}}
=\bigg(\frac{t_k^2}{3}-\frac{t_{k-1}}{6}(t_k+t_{k-1})\bigg)(t_k-t_{k-1}) \nonumber \\
&=&\bigg(\frac{t_k}{3}+\frac{t_{k-1}}{6}\bigg)(t_k-t_{k-1})^2. \nonumber 
\end{eqnarray}
}
\end{discuss}
Because of this, we cannot ensure that $b_n^{1/4}(\hat{\sigma}_n-\sigma_\ast)$ is $O_p(1)$.

As seen above, we cannot obtain the optimal rate of convergence of the maximum-likelihood-type estimator $\hat{\sigma}_n$, due to the bias term of $H_n$.
In the following, we consider how to remove the bias.
First, we consider an estimator $(B_{m,n})$ of $\Sigma_{s_{m-1},\dagger}$, using the function $g$ from Jacod et al. \cite{jac-etal09}.
The function $g:[0,1]\to \mathbb{R}$ is continuous and piecewise $C^1$, with $g(0)=g(1)=0$, and $\int^1_0g(x)dx >0$.
We let $g^j_l=g(l/(k^j_m+1))$, $\Psi_1=\int^1_0g(x)^2dx$, and $\Psi_2=\int^1_0g'(x)^2dx$.
For example, if we let $g(x)=x\wedge (1-x)$ on $0\leq x\leq 1$, then $\Psi_1=1/12$ and $\Psi_2=1$. 
Here, let $\hat{a}_m=(\hat{a}_m^1,\cdots,\hat{a}_m^\gamma)$,
$\hat{a}_m^i=k^i_m\hat{v}_{i,n}^{-1}(Tk_n)^{-1}1_{\{\hat{v}_{i,n}>0\}}$ and let 
$B_{m,n}$ be a $\gamma \times \gamma$ matrix satisfying
\begin{equation*}
[B_{m,n}]_{ij}=\frac{\ell_n}{T\Psi_1}\bigg\{\bigg(\sum_{l=1}^{k^i_m}g_l^iZ_{m,l}^i\bigg)\bigg(\sum_{l=1}^{k^j_m}g_l^jZ_{m,l}^j\bigg)-\frac{\hat{v}_{i,n}}{k^i_m}\Psi_21_{\{i=j\}}\bigg\}.
\end{equation*}
We next define a bias-corrected quasi-log-likelihood function $\check{H}_n(\sigma)$ as
\begin{equation*}
\check{H}_n(\sigma)=H_n(\sigma,\hat{v}_n)+\frac{1}{2}\sum_mG_m(\hat{a}_m,\Sigma_m(\sigma),B_{m,n},\hat{v}_n). \nonumber 
\end{equation*}
By setting 
$\hat{D}_{m,n,\dagger}=((\hat{a}_m^i\hat{a}_m^j)^{1/2}[B_{m,n}]_{ij})_{ij}$ 
and $\hat{D}_{m,n}(\sigma)=((\hat{a}_m^i\hat{a}_m^j)^{1/2}[\Sigma_m(\sigma)]_{ij})_{ij}$,
$\check{H}_n(\sigma)$ can be simplified to
\begin{eqnarray}
\check{H}_n(\sigma)&=&-\sum_m\bigg\{\frac{1}{2}{\rm tr}(S_m^{-1}(\sigma,\hat{v}_n)(Z_mZ_m^\top-B_{m,n})) \nonumber \\
&&\quad  +\frac{\sqrt{T}}{4}\ell_n^{-1/2}{\rm tr}((\hat{D}_{m,n}(\sigma)+\hat{D}_{m,n,\dagger})\hat{D}_{m,n}(\sigma)^{-1/2})\bigg\}. \nonumber
\end{eqnarray}

Letting $\check{\sigma}_n={\rm argmax}_\sigma \check{H}_n(\sigma)$, we obtain the optimal rate of convergence for $\check{\sigma}_n$.
\begin{theorem}\label{optConvTheorem}
Assume [B1], [B2] and [V]. Then 
\begin{equation*}
\{b_n^{1/4}(D(\Sigma(\check{\sigma}_n),\Sigma_{\dagger})-\min_{\sigma} D(\Sigma(\sigma),\Sigma_{\dagger}))\}_{n\in\mathbb{N}}
\end{equation*}
is tight.
\end{theorem}
\begin{discuss}
{\colorr 
\begin{eqnarray}
-\partial_b\check{H}_n=\sum_m\bigg\{\frac{1}{2}{\rm tr}(\partial_b S_m^{-1}(\sigma,\hat{v}_n)(Z_mZ_m^\top-B_{m,n}))+\frac{\sqrt{T}}{4}\ell_n^{-1/2}{\rm tr}(\partial_b\hat{\Sigma}_m(\sigma)^{1/2}(\mathcal{E}-\hat{\Sigma}_m(\sigma)^{-1/2}\hat{B}_m\hat{\Sigma}_m(\sigma)^{-1/2}))\bigg\}. \nonumber
\end{eqnarray}
}
\end{discuss}

\subsection{Fast calculation of the estimator}\label{fastCalc-section}

The estimator $B_{m,n}$ of $\Sigma_{s_{m-1},\dagger}$ is constructed with the aim of bias correction.
However, it is also useful for fast calculation of a parametric estimator.

We define
\begin{equation*}
\dot{H}_n(\sigma)=-\frac{T\ell^{-1}_n}{4}\sum_{m=2}^{\ell_n}{\rm tr}((\hat{D}_{m,n,\dagger}+\hat{D}_{m,n})\hat{D}_{m,n}^{-1/2}),
\end{equation*}
and a new parametric estimator by  $\dot{\sigma}_n={\rm argmax}_\sigma \dot{H}_n(\sigma)$.

In the calculation of $\hat{\sigma}_n$ and $\check{\sigma}_n$, we must repeatedly calculate of $\partial_\sigma H_n$ or $\partial_\sigma \check{H}_n$.
At each step, we must calculate the inverse matrix of $S_m(\sigma,\hat{v}_n)$, which has size of order $k_n$.
However, we do not need to calculate the inverse of a large matrix to find $\partial_\sigma \dot{H}_n$.
Instead, we can calculate $B_{m,n}$ only once, taking $k_m^j$ terms.
This drastically shortens the calculation time.

Unfortunately, $\dot{\sigma}_n$ does not give the optimal rate of convergence in general, although we obtain the following result.
\begin{theorem}\label{fast-estimator-thm}
Assume [B1], [B2], and [V]. Then $\{\ell_n^{1/2}(D(\Sigma(\dot{\sigma}_n),\Sigma_\dagger)-\min_\sigma D(\Sigma(\sigma),\Sigma_\dagger))\}_n$ is tight.
\end{theorem}

\begin{remark}
In practical calculation of $\dot{\sigma}_n$, we need to calculate $\partial_\sigma \dot{H}_n$. Since
\begin{eqnarray}
\partial_\sigma\dot{H}_n&=&-\frac{\ell^{-1}_n}{4}\sum_{m=2}^{\ell_n}{\rm tr}(\hat{D}_{m,n,\dagger}\partial_\sigma \hat{D}_{m,n}^{-1/2}+\partial_\sigma \hat{D}_{m,n}^{1/2}) \nonumber \\
&=&-\frac{\ell^{-1}_n}{4}\sum_{m=2}^{\ell_n}{\rm tr}(\partial_\sigma \hat{D}_{m,n}^{1/2}(\mathcal{E}_\gamma-\hat{D}_{m,n}^{-1/2}\hat{D}_{m,n,\dagger}\hat{D}_{m,n}^{-1/2})), \nonumber
\end{eqnarray}
we can obtain $\partial_\sigma \dot{H}_n$ from calculating $\partial_\sigma \hat{D}_{m,n}^{1/2}$.
Let $U_{(m)}$ be an orthogonal matrix and $\Lambda_{(m)}={\rm diag}((\lambda_{j,(m)})_j)$ such that $U_{(m)}\Lambda_{(m)}U_{(m)}^\top=\hat{D}_{m,n}$. 
Then $\hat{D}_{m,n}^{1/2}\hat{D}_{m,n}^{1/2}=\hat{D}_{m,n}$ implies
\begin{equation*}
\partial_\sigma \hat{D}_{m,n}^{1/2}\hat{D}_{m,n}^{1/2}+\hat{D}_{m,n}^{1/2}\partial_\sigma \hat{D}_{m,n}^{1/2}=\partial_\sigma \hat{D}_{m,n}.
\end{equation*}
From that, we have
\begin{equation*}
[U_{(m)}^\top\partial_\sigma \hat{D}_{m,n}^{1/2}U_{(m)}]_{ij}\lambda_{j,(m)}^{1/2}+\lambda_{i,(m)}^{1/2}[U_{(m)}^\top \partial_\sigma \hat{D}_{m,n}^{1/2}U_{(m)}]_{ij}=[U_{(m)}^\top\partial_\sigma \hat{D}_{m,n}U_{(m)}]_{ij},
\end{equation*}
and hence
\begin{equation*}
[U_{(m)}^\top\partial_\sigma \hat{D}_{m,n}^{1/2}U_{(m)}]_{ij}=[U_{(m)}^\top\partial_\sigma \hat{D}_{m,n}U_{(m)}]_{ij}/(\lambda_{i,(m)}^{1/2}+\lambda_{j,(m)}^{1/2}),
\end{equation*}
which is useful in practical calculation of $\partial_\sigma \hat{D}_{m,n}^{1/2}$.
\end{remark}

\begin{discuss}
{\colorr nオーダーの計算はBmnのみ. $\ell'_n$を大きめにとれば結構いいかも

\begin{equation*}
\partial_\sigma D^{1/2}U\Lambda^{1/2}U^\top+U \Lambda^{1/2}U^\top \partial_\sigma D^{1/2}=\partial_\sigma D,
\end{equation*}
\begin{equation*}
U^\top\partial_\sigma D^{1/2}U\Lambda^{1/2}+ \Lambda^{1/2}U^\top \partial_\sigma D^{1/2}U=U^\top \partial_\sigma D U.
\end{equation*}
}
\end{discuss}

\subsection{Asymptotic mixed normality}\label{asym-mixed-normality-section}

In this section, we discuss the asymptotic mixed normality of $\check{\sigma}_n$. 
In a correctly specified model, the maximum-likelihood-type estimator $\hat{\sigma}_n$ satisfies the asymptotic mixed normality as seen in Theorem 2.1 of~\cite{ogi18}.
In our misspecified model, we need further assumptions related to the smoothness of $\Sigma(t,x,\sigma)$ and 
the convergence of $\check{\sigma}_n$ in the parameter space $\Lambda$.
Let $\mathcal{N}_k$ denote a general $k$-dimensional standard normal random variable on an extension of some probability space, and let it be independent of any random variables
in the original probability space.

\begin{description}
\item{[C]} 1. There exists a $\Lambda$-valued random variable $\sigma_\ast$ such that $\check{\sigma}_n\overset{P}\to\sigma_\ast$.\\
2. $\partial_\sigma^2\Sigma$ exists and is continuous on $(t,x,\sigma)$, and there exists a locally bounded function $L(x,y)$ such that 
(\ref{Lip-conti-ineq}) is satisfied with $l=2$.\\
3. $\Gamma_2=\partial_\sigma^2D(\Sigma(\sigma),\Sigma_\dagger)|_{\sigma=\sigma_\ast}$ is positive definite almost surely.
\end{description}

\begin{discuss}
{\colorr $X_n$: $\Lambda$-valued, $\inf_n d(X_n,\Lambda^c)>0$ a.s.でも収束部分列が取れるとは限らない．
$X_1=1_{[0,1/2)}$, $X_2=1_{[0,1/4)}+1_{[1/2,3/4)}$, $X_3=1_{[0,1/8)}+1_{[1/4,3/8)}+1_{[1/2,5/8)}+1_{[3/4,7/8)}$
とすると常に$P[|X_k-X_l|\geq 1]=1/2$で収束部分列がとれない．
}
\end{discuss}

If $\Lambda$ has a unique point that minimizes $D(\Sigma(\sigma),\Sigma_\dagger)$ almost surely, we can check the first condition of [C].
\begin{lemma}\label{uniqueness-lemma}
Assume [B1], [B2], [V] and that a $\Lambda$-valued random variable $\sigma_\ast$ is a unique minimizing point of $D(\Sigma(\sigma),\Sigma_\dagger)$ on $\bar{\Lambda}$ almost surely.
Then $\check{\sigma}_n\overset{P}\to\sigma_\ast$ as $n\to\infty$.
\end{lemma}

By simple calculation, we obtain
\begin{equation*}
\Gamma_2=\frac{1}{4}\int^T_0{\rm tr}(\partial_\sigma^2\mathcal{D}_{t,\sigma_\ast}^{1/2}(\mathcal{E}_\gamma-\mathcal{D}_{t,\sigma_\ast}^{-1/2}\mathcal{D}_{t,\dagger}\mathcal{D}_{t,\sigma_\ast}^{-1/2})
+2\mathcal{D}_{t,\dagger}\mathcal{D}_{t,\sigma_\ast}^{-1/2}(\partial_\sigma \mathcal{D}_{t,\sigma_\ast}^{1/2}\mathcal{D}_{t,\sigma_\ast}^{-1/2})^2)dt,
\end{equation*}
where $\mathcal{D}_{t,\sigma}=\mathcal{D}(t,\Sigma(\sigma))$ and $\mathcal{D}_{t,\dagger}=\mathcal{D}(t,\Sigma_\dagger)$.
\begin{discuss}
{\colorr 
\begin{eqnarray}
\Gamma_2&=&\frac{1}{4}\int^T_0\partial_\sigma^2{\rm tr}(\mathcal{D}_{t,\dagger}\mathcal{D}_{t,\sigma}^{-1/2}+\mathcal{D}_{t,\sigma}^{1/2})dt
=\frac{1}{4}\int^T_0\partial_\sigma{\rm tr}(\partial_\sigma \mathcal{D}_{t,\sigma}^{1/2}-\mathcal{D}_{t,\dagger}\mathcal{D}_{t,\sigma}^{-1/2}\partial_\sigma \mathcal{D}_{t,\sigma}^{1/2}\mathcal{D}_{t,\sigma}^{-1/2})dt \nonumber \\
&=&\frac{1}{4}\int^T_0\partial_\sigma{\rm tr}(\partial_\sigma^2 \mathcal{D}_{t,\sigma}^{1/2}-\mathcal{D}_{t,\dagger}\mathcal{D}_{t,\sigma}^{-1/2}\partial_\sigma^2 \mathcal{D}_{t,\sigma}^{1/2}\mathcal{D}_{t,\sigma}^{-1/2}
+\mathcal{D}_{t,\dagger}\mathcal{D}_{t,\sigma}^{-1/2}(\partial_\sigma \mathcal{D}_{t,\sigma}^{1/2}\mathcal{D}_{t,\sigma}^{-1/2})^2)dt \nonumber
\end{eqnarray}
よりＯＫ．
}
\end{discuss}
Let 
\begin{eqnarray}
\Gamma_{1,\sigma}&=&\frac{1}{2}\int^T_0\big\{{\rm tr}(\mathcal{D}_{t,\sigma}^{-\frac{1}{2}}(\partial_\sigma \mathcal{D}_{t,\sigma}^{\frac{1}{2}})^2)+\mathfrak{K}_1(\mathcal{D}_{t,\sigma},\partial_\sigma \mathcal{D}_{t,\sigma},\mathcal{D}_{t,\dagger}-\mathcal{D}_{t,\sigma}) \nonumber \\
&&\quad \quad \quad+\mathfrak{K}_2(\mathcal{D}_{t,\sigma},\partial_\sigma \mathcal{D}_{t,\sigma},\mathcal{D}_{t,\dagger}-\mathcal{D}_{t,\sigma})\big\}dt, \nonumber
\end{eqnarray}
where for a symmetric, positive definite matrix $B$ and matrices $A$ and $C$, $\varphi_B(A)$ is defined as in Lemma~\ref{varphi-lemma},
\begin{equation*}
\mathfrak{K}_1(B, A, C)={\rm tr}(B^{-1/2}\varphi_B(A)^2B^{-1/2}\varphi_B(C))+{\rm tr}((\varphi_B(A)B^{-1/2})^2CB^{-1/2}),
\end{equation*}
\begin{eqnarray}
\mathfrak{K}_2(B,A,C)&=&{\rm tr}(CB^{-1/2}\varphi_B(A)B^{-1/2}\varphi_B(C)B^{-1/2}\varphi_B(A)B^{-1/2}) \nonumber \\
&&+{\rm tr}(\varphi_B(\varphi_B(A)\varphi_B(C))^2B^{-1/2}). \nonumber
\end{eqnarray}

\begin{discuss}
{\colorr 第一項は${\rm tr}((\partial_\sigma {\bf B}_t^{1/2})^2{\bf B}_t^{-1/2})$に一致する. おそらく
\begin{eqnarray}
\Gamma&=&\int^T_0\bigg\{{\rm tr}((\partial_\sigma {\bf B}_t{\bf B}_t^{-1/2})^2)+2{\rm tr}((\partial_\sigma {\bf B}_t{\bf B}_t^{-1/2})^2({\bf B}_{t,\dagger}-{\bf B}_t){\bf B}_t^{-1/2})
+ {\rm tr}((\partial_\sigma {\bf B}_t{\bf B}_t^{-1/2}({\bf B}_{t,\dagger}-{\bf B}_t){\bf B}_t^{-1/2})^2)\bigg\}dt \nonumber \\
&=&\int^T_0{\rm tr}((\partial_\sigma {\bf B}_t{\bf B}_t^{-1/2}(\mathcal{E}_\gamma+({\bf B}_{t,\dagger}-{\bf B}_t){\bf B}_t^{-1/2}))^2)dt. \nonumber
\end{eqnarray}
→この式はたぶん違う．
}
\end{discuss}

\begin{discuss}
{\colorr 一次元で$\tilde{a}_t^1\equiv 1$の時，
\begin{equation*}
\varphi_B(A)=\frac{A}{2B^{1/2}},\quad \mathfrak{K}_1=\frac{A^2C}{4B^{5/2}},\quad \mathfrak{K}_2=\frac{9A^2C^2}{64B^{7/2}},
\end{equation*}
\begin{eqnarray}
\Gamma_2&=&\frac{1}{4}\int^T_0\bigg[\partial_\sigma^2 \Sigma_\ast^{1/2}-\frac{\Sigma_\dagger \partial_\sigma^2\Sigma_\ast^{1/2}}{\Sigma_\ast}
+\frac{2(\partial_\sigma \Sigma_\ast^{1/2})^2\Sigma_\dagger}{\Sigma_\ast^{3/2}}\bigg]dt, \nonumber \\
\Gamma_1&=&\frac{1}{2}\int^T_0\bigg[\frac{(\partial_\sigma \Sigma_\ast^{1/2})^2}{\Sigma_\ast^{1/2}}+\frac{(\partial_\sigma \Sigma_\ast)^2(\Sigma_\dagger-\Sigma_\ast)}{4\Sigma_\ast^{5/2}}
+\frac{9(\partial_\sigma \Sigma_\ast)^2(\Sigma_\dagger -\Sigma_\ast)^2}{64\Sigma_\ast^{7/2}}\bigg]dt \nonumber
\end{eqnarray}
となり，$\Sigma_\ast\equiv \Sigma_\ast$の時$\Gamma_1=\Gamma_2$. $\Sigma_\ast$が$t$に依らない時Gloter and Jacodの$\Gamma$とも一致する．

$\Gamma_1$と$\Gamma_2$が一致しない例はすぐには判らないが、$\Sigma=\sigma^2$のような形にしておけば$\partial_\sigma^2\Sigma_\ast^{1/2}=0$なので$\Gamma_2$はシンプルになり，
$\Sigma_\dagger^{1/2}=at+b$のような形の時に$(\Sigma_\ast-\Sigma_\dagger)^2$を含む$\Gamma_1$が常に$\Gamma_2$と一致するとはとても思えないから取りあえずそれでいい．
}
\end{discuss}

Because of [C], for any $\epsilon>0$, there exists a positive integer $N$ such that
\begin{equation*}
P[\{\sigma_\ast+t(\check{\sigma}_n-\sigma_\ast)\}_{0\leq t\leq 1}\in\Lambda]<\epsilon
\end{equation*}
whenever $n\geq N$. 
If $\{\sigma_\ast+t(\check{\sigma}_n-\sigma_\ast)\}_{0\leq t\leq 1}\subset \Lambda$, then Taylor's formula yields
\begin{equation*}
-\partial_\sigma \check{H}_n(\sigma_\ast)=\partial_\sigma \check{H}_n(\check{\sigma}_n)-\partial_\sigma \check{H}_n(\sigma_\ast)=\partial_\sigma^2\check{H}_n(\xi_n)(\check{\sigma}_n-\sigma_\ast),
\end{equation*}
where $\xi_n\in \{\sigma_\ast+t(\check{\sigma}_n-\sigma_\ast)\}_{0\leq t\leq 1}$.
\begin{discuss}
{\colorr 
連続関数$\xi\mapsto \{\partial_\sigma \check{H}_n(\sigma_\ast)+\partial_\sigma^2\check{H}_n(\xi)(\check{\sigma}_n-\sigma_\ast)\}^2$の最小点だから
可測選択定理より$\xi_n$: r.v.にとれる．

$\int^1_0\partial_\sigma^2\check{H}_n dt$でやると誤差評価でより高次の微分が必要なので，可測性が言えるなら$\xi$でやった方がいい．
}
\end{discuss}

Then, if
\begin{equation}
b_n^{-1/2}\partial_\sigma^2\check{H}_n(\xi_n)\overset{P}\to \Gamma_2
\end{equation}
and 
\begin{equation}\label{score-st-conv}
b_n^{-1/4}\partial_\sigma \check{H}_n(\sigma_\ast)\to^{s\mathchar`-\mathcal{L}}\Gamma_{1,\sigma_\ast}^{-1/2}\mathcal{N}_\gamma,
\end{equation}
we obtain
\begin{equation*}
b_n^{1/4}(\check{\sigma}_n-\sigma_\ast)\to^{s\mathchar`-\mathcal{L}} \Gamma_2^{-1}\Gamma_{1,\sigma_\ast}^{1/2}\mathcal{N}_\gamma.
\end{equation*}

In Proposition 7.2 in~\cite{ogi18}, the martingale central limit theorem of Jacod~\cite{jac97} is used to show the corresponding result with (\ref{score-st-conv}).
However, $\sigma_\ast$ is random in our model, meaning that $\partial_\sigma \check{H}_n(\sigma_\ast)$ is not a martingale.
To deal with this problem, we prepare the following scheme to show stable convergence for the case of a random parameter.
This scheme ought to be useful in many situations involving random parameters.

Let $({\bf A}, {\bf G}, {\bf P})$ be a probability space.
Let $\mathcal{X}$ be a complete, separable metric space, let $\{{\bf Z}_n(x)\}_{n\in\mathbb{N},x\in \mathcal{X}}$
be random fields on $({\bf A},{\bf G})$, let $\{{\bf Z}(x)\}_{x\in\mathcal{X}}$ be a continuous random field on an extension $({\bf A}',{\bf G}')$ of $({\bf A},{\bf G})$.
Let $\to ^{{\bf G}\mathchar`-\mathcal{L}}$ denote ${\bf G}$-stable convergence.

\begin{proposition}\label{random-param-convergence}
Let $({\bf Z}_{n,k})_{n,k\in \mathbb{N}}$ and $(V_{k,k'})_{k,k'\in\mathbb{N}}$ be random variables on $({\bf A}, {\bf G})$, 
let $(f_k)_{k\in\mathbb{N}}$ be continuous functions on $\mathcal{X}$,
$K_0=0$, and let $(K_m)_{m\in\mathbb{N}}$ be a monotonically increasing sequence of positive integers satisfying $K_m\to \infty$ as $m\to\infty$.
Assume the following.
\begin{enumerate}
\item For any $m$, $(V_{k,k'})_{1\leq k,k'\leq K_m}$ is a symmetric, nonnegative definite matrix almost surely and
\begin{equation*}
({\bf Z}_{n,k})_{1\leq k\leq K_m}\to ^{{\bf G}\mathchar`-\mathcal{L}}\sqrt{(V_{k,k'})_{1\leq k,k'\leq K_m}}N_m
\end{equation*}
as $n\to \infty$, where $N_m$ is a $K_m$-dimensional standard normal random variable independent of ${\bf G}$.
\item For any $x\in\mathcal{X}$, there exists an open set $U_x$ satisfying the following property: 
given any $\epsilon>0$ and $\delta>0$ there exists a positive integer $M$ such that 
\begin{equation}\label{Zf-tail-est}
P\bigg[\sup_{x'\in U_x}\bigg(\bigg|\sum_{k=K_M+1}^{K_{M'}}{\bf Z}_{n,k}f_k(x')\bigg|+|\mathcal{V}_{M'}(x')-\mathcal{V}_M(x')|\bigg)> \delta \bigg]<\epsilon
\end{equation}
for any $M'>M$ and $n\in\mathbb{N}$, where $\mathcal{V}_M(x)=\sum_{k,k'=1}^{K_M}V_{k,k'}f_k(x)f_{k'}(x)$.
\end{enumerate}
\begin{discuss}
{\colorr $x\mapsto {\bf Z}_n(x)$の連続性とかがないと${\bf Z}_n(Y)$の可測性が言えないか→2.から言えるか}
\end{discuss}
Then ${\bf Z}_n(x):=P\mathchar`-\lim_{M\to\infty}\sum_{k=1}^{K_M}{\bf Z}_{n,k}f_k(x)$ exists for any $x$. 
Moreover, if ${\bf Z}_n(x)\to^{{\bf G}\mathchar`-\mathcal{L}}{\bf Z}(x)$ as $n\to \infty$ for any $x\in\mathcal{X}$,
then ${\bf Z}_n({\bf Y})\to^{{\bf G}\mathchar`-\mathcal{L}}{\bf Z}({\bf Y})$ as $n\to\infty$
for any $\mathcal{X}$-valued ${\bf G}$-measurable random variable ${\bf Y}$.
\end{proposition}
\begin{discuss}
{\colorr ${\bf Z}_n({\bf Y})$の展開には共分散行列の固有の形を使っているからこれ以上スペシフィックな形で書くのは難しいか．
${\bf Z}_n(x)$のa.s.収束はPropositionからは言えないが適用時には(\ref{invS-expansion})の$\tilde{S}_m^{-1}$の展開が収束していることから言える．
$x=0, {\bf Z}(x)=\Gamma_{1,\sigma}^{1/2}\mathcal{N}$だから${\bf Z}(x)$の連続性もわかる．}
\end{discuss}

The proof is given in Section~\ref{Section3-proof-section}.
{\colord When we have the stable convergence of $\partial_\sigma \check{H}_n(\sigma)$ for a deterministic parameter $\sigma$, 
this proposition helps us to extend the convergence to the case with a random parameter. Indeed,}
by using Proposition~\ref{random-param-convergence}, we obtain the asymptotic mixed normality of our estimator.

\begin{theorem}\label{mixed-normality-thm}
Assume [B1], [B2], [V], and [C]. Then 
\begin{equation*}
b_n^{1/4}(\check{\sigma}_n-\sigma_\ast)\to^{s\mathchar`-\mathcal{L}} \Gamma_2^{-1}\Gamma_{1,\sigma_\ast}^{1/2}\mathcal{N},
\end{equation*}
where $\mathcal{N}$ is a $\gamma$-dimensional standard normal random variable on an extension of $(\Omega,\mathcal{F},P)$
and independent of $\mathcal{F}$.
\end{theorem}

\begin{remark}
For correctly specified models, we obtain 
\begin{equation*}
\Gamma_{1,\sigma_\ast}=\Gamma_2=\frac{1}{2}\int^T_0{\rm tr}(\mathcal{D}_{t,\sigma_\ast}^{1/2}(\partial_\sigma \mathcal{D}_{t,\sigma_\ast}^{1/2})^2)dt
\end{equation*}
since $\mathcal{D}_{t,\dagger}\equiv \mathcal{D}_{t,\sigma_\ast}$.
This value corresponds to the asymptotic variance $-\partial_\sigma^2\mathcal{Y}_1(\sigma_\ast)$ of the maximum-likelihood-type estimator 
in Section 2.2 of \cite{ogi18} when $\gamma=2$.
\end{remark}

\begin{discuss2}
{\colorr 同様の操作はノイズのないモデルにおいては適用が困難であり，ノイズ付共分散行列の性質をうまく利用している点が興味深い．}
\end{discuss2}

\section{Simulation studies of neural networks}\label{simulation-section}

\begin{discuss}
{\colorr 目的：この手法の意義を伝えたい。ニューラル・ネットワークを入れる意味など．あとは実際のパフォーマンスをシンプルに伝える．参考文献はReiss 2011,Jacod et al 2009.

ストーリー：
ニューラルネットのパフォーマンスを調べるためまずシミュした．ちゃんと真の関数を推定できているようだ．次に実証で学習してみた．
ノイズ処理してボラティリティ構造の学習をできるものは他にはない．多項式モデルと比べて良い推定・予測ができる→時間がかなり速いとかを売りにすべきか．
最小分散を作ってみました．→最悪なくても良い．流れはこれで問題ないと思う．
あとはNNがなぜ良いかをもっとアピールする．
}

\end{discuss}

In this section, we simulate some diffusion processes $Y$ with $b_{t,\dagger}=b_\dagger(t,Y_t)$ for some function $b_\dagger$, 
and calculate the proposed estimators via neural network modeling in Section \ref{nn-example-section}.
We will verify whether the function $b_{t,\dagger}$ is well approximated.

\subsection{One-dimensional Cox--Ingersoll--Ross processes}\label{CIR-simu-section}

First, we consider the case where the process $Y$ is the Cox--Ingersoll--Ross (CIR) process derived in \cite{cox-etal85}, that is, 
$Y$ is a one-dimensional diffusion process satisfying a stochastic differential equation
\begin{equation}
dY_t=(\alpha_1-\alpha_2 Y_t)dt+\sigma_\ast \sqrt{Y_t}dW_t,
\end{equation}
where $\sigma_\ast>0$ is the parameter. We assume that $2\alpha_1>\sigma_\ast^2$ to ensure $\inf_{t\in [0,T]}Y_t>0$ almost surely.
Let $(\epsilon_i^n)_{i\in\mathbb{Z}_+}$ be independent identically distributed Gaussian random variables with variance $v_\ast>0$.
Sampling times are given by $S^n_i=\inf\{t\geq 0 : N_{nt}\geq i\}\wedge T$ with a Poisson process $\{N_t\}_{0\leq t\leq T}$ with parameter $\lambda$.

We set $X_t=(t,Y_t)$, $h(x)=x/(1+e^{-x})$, $K=3$, $L_1=L_2=10$, and $\epsilon=0.0001$ for Section~\ref{nn-example-section}.
We generate $100$ sample paths of $Y$, set the number of epochs to $3,000$ and randomly pick a simulated path at each optimization step.
We use ADADELTA, proposed in Zeiler~\cite{zei12}, with weight decay having parameter $0.005$ for optimization.
We set $\ell_n=[n^{0.45}]$ and $\hat{v}_n=(2{\bf J}_{1,n})^{-1}\sum_{i,m}(Z_{i,m}^1)^2$.
We also calculate the maximum-likelihood-type estimator $\hat{\sigma}_n^{{\rm {\bf model}}}$ for a parametric model $\Sigma(t,x,\sigma)=\sigma^2x$,
and the associated quasi-log-likelihood $H_n^{\rm {\bf model}}$.

\begin{figure}[htb]
\begin{center}
\includegraphics[width=90mm]{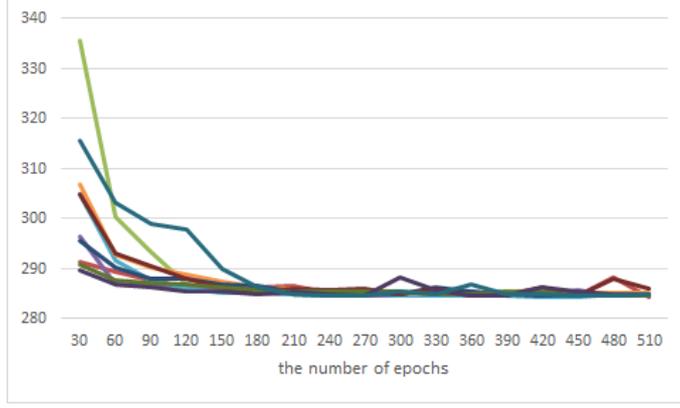}
\end{center}
\caption{Transition of $H_n^{\rm {\bf model}}-H_n(\hat{\beta}_n)$ for ten training results}
\label{fig1}
\end{figure}

Figure~\ref{fig1} shows the levels of the loss functions $H_n^{\rm {\bf model}}-H_n(\hat{\beta}_n)$ for ten training results with different initial values. 
The parameters are set as $T=1$, $n=5000$, $\lambda=1$, $\alpha_1=\alpha_2=1$, and $\sigma_\ast=1$.
For each trial, the loss function seems to converge to a specific value as the number of epochs increases.
Since $\hat{\sigma}_n^{{\rm {\bf model}}}$ is calculated on the basis of a parametric model that includes the true model,
the average value of $-H_n^{\rm {\bf model}}$ is less than $-H_n(\hat{\beta}_n)$, as expected.
However, we can see that $-H_n(\hat{\beta}_n)$ reaches a value close to the average of $-H_n^{\rm {\bf model}}$, and is calculated without any information of the parametric model.

\begin{figure}[htb]
\begin{center}
\includegraphics[width=90mm]{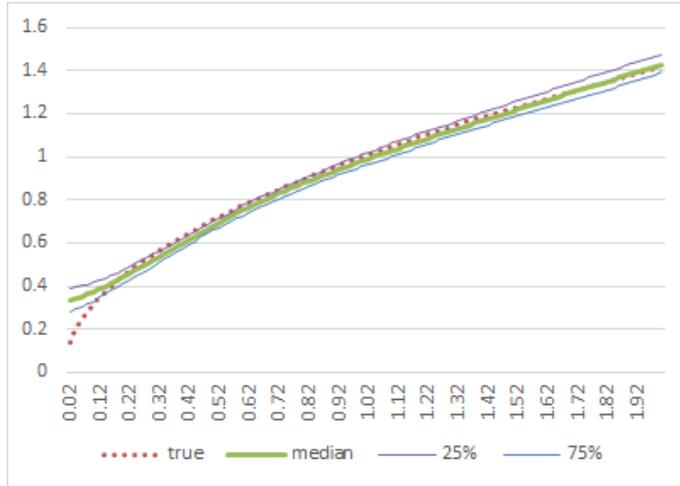}
\end{center}
\caption{Trained function $\sqrt{\Sigma}(0,x,\hat{\beta}_n)$. The horizontal line shows the value of $x$, and the vertical line shows the value of $\sqrt{\Sigma}$.}
\label{fig2}
\end{figure}

Figure~\ref{fig2} shows {\colord quartiles of trained functions $\sqrt{\Sigma}(0,x, \hat{\beta}_n)$ for $100$ training results}.
The true function is $\sqrt{\Sigma_\dagger}(t,x)\equiv \sqrt{x}$ indicated by the dotted line.
We can see that the true function is well-trained except near the origin.
Since $\inf_tY_t>0$ almost surely, the errors of trained functions are relatively large near the origin.
Table~\ref{table1} shows the average values of
\begin{equation*}
{\rm MSE}_k(\beta)=\sqrt{\frac{1}{20}\sum_{i=1}^{20}\big|\Sigma(0,x_i^{(k)},\beta)-\Sigma_\dagger(x_i^{(k)})\big|^2}
\end{equation*}
for $100$ trials, where $x_i^{(1)}=0.1i$ and $x_i^{(2)}=0.1+0.1i$.
We also calculate similar quantities ${\rm MSE}_k^{\rm model}$ for $\hat{\sigma}_n^{{\rm model}}$.
{\colord The performance of $\hat{\sigma}_n^{{\rm model}}$ is better since it is calculated by using the structure of parametric model.
Though $\hat{\beta}_n$ can be calculated without the structure of parametric model,
we can see that $\hat{\beta}_n$ achieves good performance for large number of epochs.}

\begin{table}
\caption{MSE median of $\hat{\beta}_n$ for each number of epochs}
\label{table1}
\begin{center}
\footnotesize
\begin{tabular}{|c|ccccc|} \hline
Epochs & $100$ & $300$ & $500$ & $1000$ & $3000$ \\ \hline
${\rm MSE}_1(\hat{\beta}_n)$ & $0.2938$ & $0.1383$ & $0.1175$ & $0.1094$ & $0.0932$ \\ 
${\rm MSE}_2(\hat{\beta}_n)$ & $0.3219$ &  $0.1471$ & $0.1198$ & $0.1097$ & $0.0967$  \\ \hline
${\rm MSE}_1^{{\rm model}}$ & \multicolumn{5}{|c|}{$0.0769$} \\ 
${\rm MSE}_2^{{\rm model}}$ & \multicolumn{5}{|c|}{$0.0829$} \\ \hline
\end{tabular} 
\end{center}
\end{table}

\begin{discuss}
{\colorr model MLEのMSE
\begin{equation*}
{\rm MSE}=\sqrt{\frac{1}{20}\sum_i|\sigma^2x_i-x_i|^2}=|\sigma^2-1|\sqrt{\frac{1}{20}\sum_ix_i^2}.
\end{equation*}
}
\end{discuss}

{\colord Figure~\ref{bias} shows the quartiles of $H_n^{\rm {\bf model}}-H_n(\check{\beta}_n)$ for $100$ training results with different initial values,
where $\check{\beta}_n$ is the bias-corrected estimator in Section~\ref{opt-conv-section} for the neural network model in Section \ref{nn-example-section}.
Table~\ref{bias_table} shows the numerical values for certain numbers of epochs.
Though $\check{\beta}_n$ achieves the optimal rate $b_n^{-1/4}$, the level of $H_n(\check{\beta}_n)$ is the same as that of $H_n(\hat{\beta}_n)$ or even worse.
This seems to hapen because the estimation accuracy of the estimator $B_{m,n}$ for $\Sigma_{s_{m-1},\dagger}$ is not so good as seen in Lemma \ref{Bmn-error-est-lemma}.}

\begin{discuss}
{\colorr 
\begin{equation*}
\sum_j\sum_mG_m=O_p(1+b_n^{-5/8}k_n)\cdot \ell_n=O_p(\ell_n+b_n^{3/8})
\end{equation*}
だから$\hat{\sigma}_n$の収束レートは$O_p(\ell_n+b_n^{3/8})b_n^{-1/2}\approx n^{0.05}$でかなり悪い．
}
\end{discuss}

\begin{figure}[htb]
\begin{center}
\includegraphics[width=90mm]{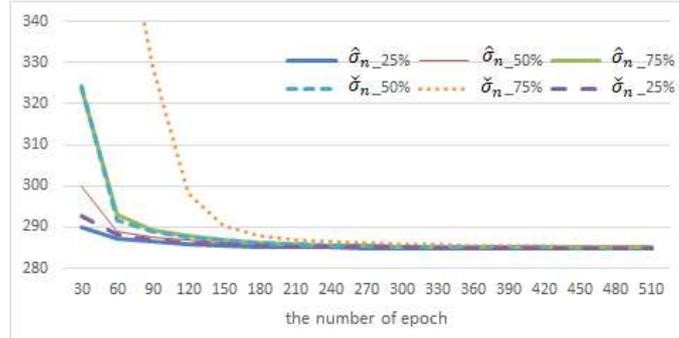}
\end{center}
\caption{Transition of the quartiles of $H_n^{\rm {\bf model}}-H_n(\beta)$}
\label{bias}
\end{figure}

\begin{table}
\caption{Average of $H_n^{\rm {\bf model}}-H_n(\beta)$ for each number of epochs}
\label{bias_table}
\begin{center}
\footnotesize
\begin{tabular}{|c|ccccc|} \hline
Epochs & $30$ & $120$ & $300$ & $600$ & $900$ \\ \hline
$H_n^{\rm {\bf model}}-H_n(\hat{\beta}_n)$ & $317.2$ & $288.8$ & $285.6$ & $285.2$ & $284.9$ \\ 
$H_n^{\rm {\bf model}}-H_n(\check{\beta}_n)$ & $474.6$ & $324.8$ & $293.4$ & $290.0$ & $289.0$ \\ \hline
\end{tabular} 
\end{center}
\end{table}

\begin{discuss}
{\colorr Ogi18 (4.2)より$\sum_mG_m\sim b_n^{1/2}k_n^{-1}b_n^{1/2}$なので，$b_n^{-1/4}\sum_mG_m\sim b_n^{3/4}k_n^{-1}$. 
$\hat{\sigma}$との収束レートの差は$n^{-1/4}$と$k_nb_n^{-1/2}$でそこまで小さくないが．
元々最適レートを達成していることも多い．逆にバイアスを推定する時のずれが余計なエラーを生んでいるかも．} 
\end{discuss}

We also calculated $H_n^{\rm {\bf model}}-H_n(\dot{\beta}_n)$ and MSE for $\dot{\beta}_n$ in Figure~\ref{fig3} and Table~\ref{table2},
where $\dot{\beta}_n$ is the estimator in Section \ref{fastCalc-section} for the neural network model in Section \ref{nn-example-section}.
In Figure~\ref{fig3}, the performances of $\dot{\beta}_n$ are poor {\colord compared to $\hat{\beta}_n$} for some initial values.
In contrast, the time needed for calculation was reduced, as 3,000 training runs took $8.33$ s for $\dot{\beta}_n$ in contrast with $333.01$ s for $\hat{\sigma}_n$
(CPU：Intel{\scriptsize \textcircled{\tiny {\rm R}}} Xeon{\scriptsize \textcircled{\tiny {\rm R}}} Processor E5-2680 v4, 35M Cache, 2.40 GHz).
{\colord So $\dot{\beta}_n$ is useful for reducing computational cost.}
\begin{figure}[htb]
\begin{center}
\includegraphics[width=90mm]{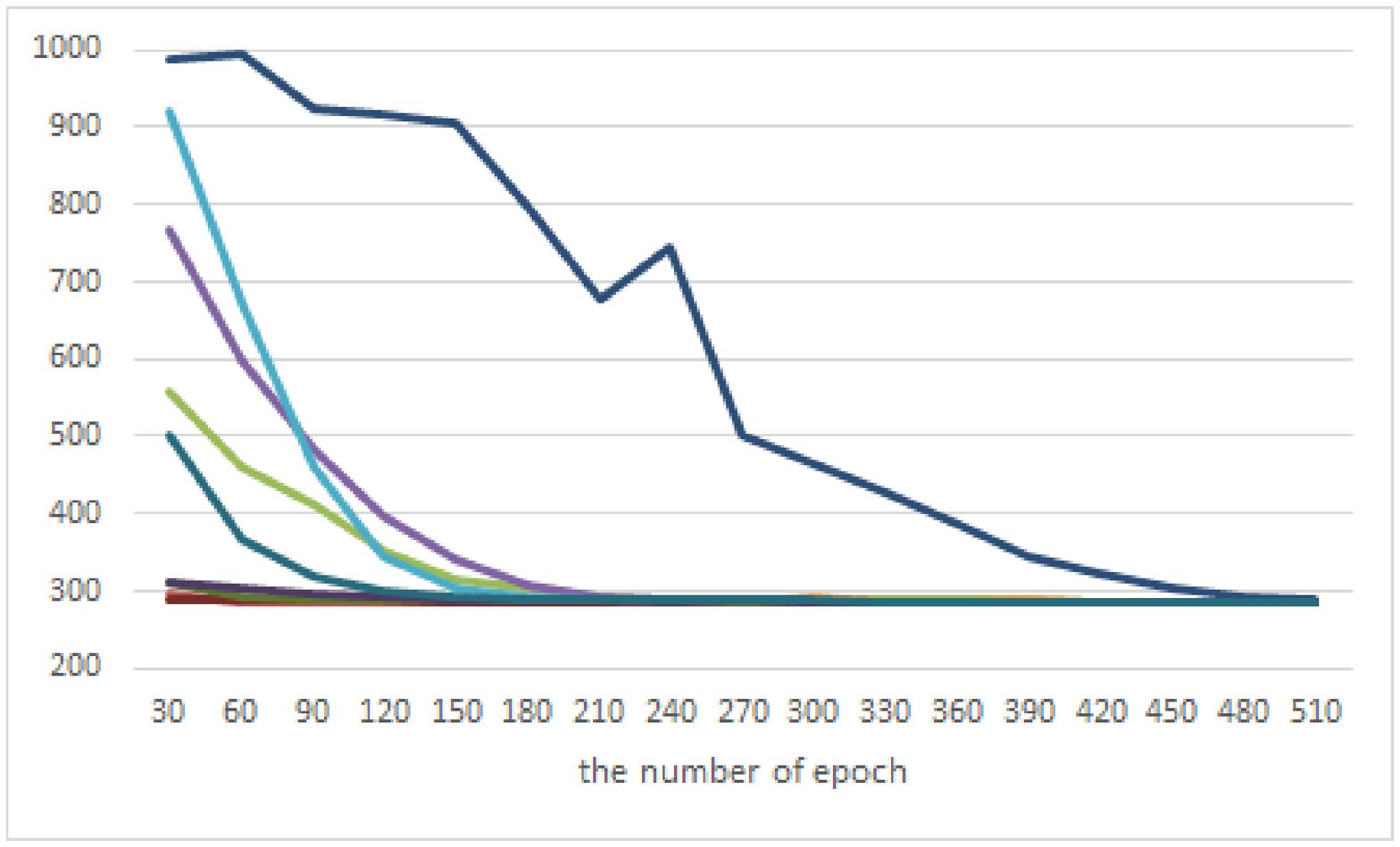}
\end{center}
\caption{Transition of $H_n^{\rm {\bf model}}-H_n(\dot{\beta}_n)$ for ten training results}
\label{fig3}
\end{figure}

\begin{table}
\caption{MSE median of the estimator $\dot{\beta}_n$}
\label{table2}
\begin{center}
\footnotesize
\begin{tabular}{|c|ccccc|} \hline
Epochs & $100$ & $300$ & $500$ & $1000$ & $3000$ \\ \hline
${\rm MSE}_1(\dot{\beta}_n)$ & $0.4902$ & $0.2368$ & $0.1545$ & $0.1441$ & $0.1349$  \\ 
${\rm MSE}_2(\dot{\beta}_n)$ & $0.5285$ &  $0.2610$ & $0.1664$ & $0.1506$ & $0.1393$   \\ \hline
\end{tabular} 
\end{center}
\end{table}

\subsection{Two-dimensional Cox--Ingersoll--Ross with intraday seasonality}

\begin{discuss}
{\colorr seasonalityの先行研究ではあまり簡単なモデルを仮定していない。seasonalityはシンプルにこう考える。}
{\colorr intraday periodicityの先行研究はAndersenBollerslev1997,HecqLaurentPalm2012}
\end{discuss}

We also simulate sample paths when $Y$ is a two-dimensional CIR-type process with intraday seasonality:
\begin{equation}
\left\{
\begin{array}{ll}
dY_t^1 & =(\alpha_1-\alpha_3 Y_t^1)dt+(at^2+bt+c)\sigma_{1,\ast}\sqrt{Y^1_t}dW_t^1 \\
dY_t^2 & =(\alpha_2-\alpha_4 Y_t^2)dt+(at^2+bt+c)\sqrt{Y^2_t}(\sigma_{3,\ast}dW_t^1 +\sigma_{2,\ast}dW_t^2) \\
\end{array}
\right.
\end{equation}
where $W_t=(W_t^1,W_t^2)$ is a two-dimensional standard Wiener process, and we define 
$\epsilon_i^{n,j}$, $S_i^{n,j}$, $\lambda_j$, $\hat{v}_{j,n}$ and $v_{j,\ast}$ for $j \in\{1,2\}$ in a way similar to the first example.
We set $\alpha_1=\alpha_2=\alpha_3=\alpha_4=1$, $a=1, b=-4/3, c=2/3$, $(\sigma_{1,\ast},\sigma_{2,\ast},\sigma_{3,\ast})=(1,\sqrt{0.75},0.5)$,
$\lambda_1=\lambda_2=1$ and $v_{1,\ast}=v_{2,\ast}=0.005$.
{\colord We calculate the estimators as the values of $a$, $b$ and $c$ are known for simplicity.}

\begin{figure}[htb]
\begin{center}
\includegraphics[width=90mm]{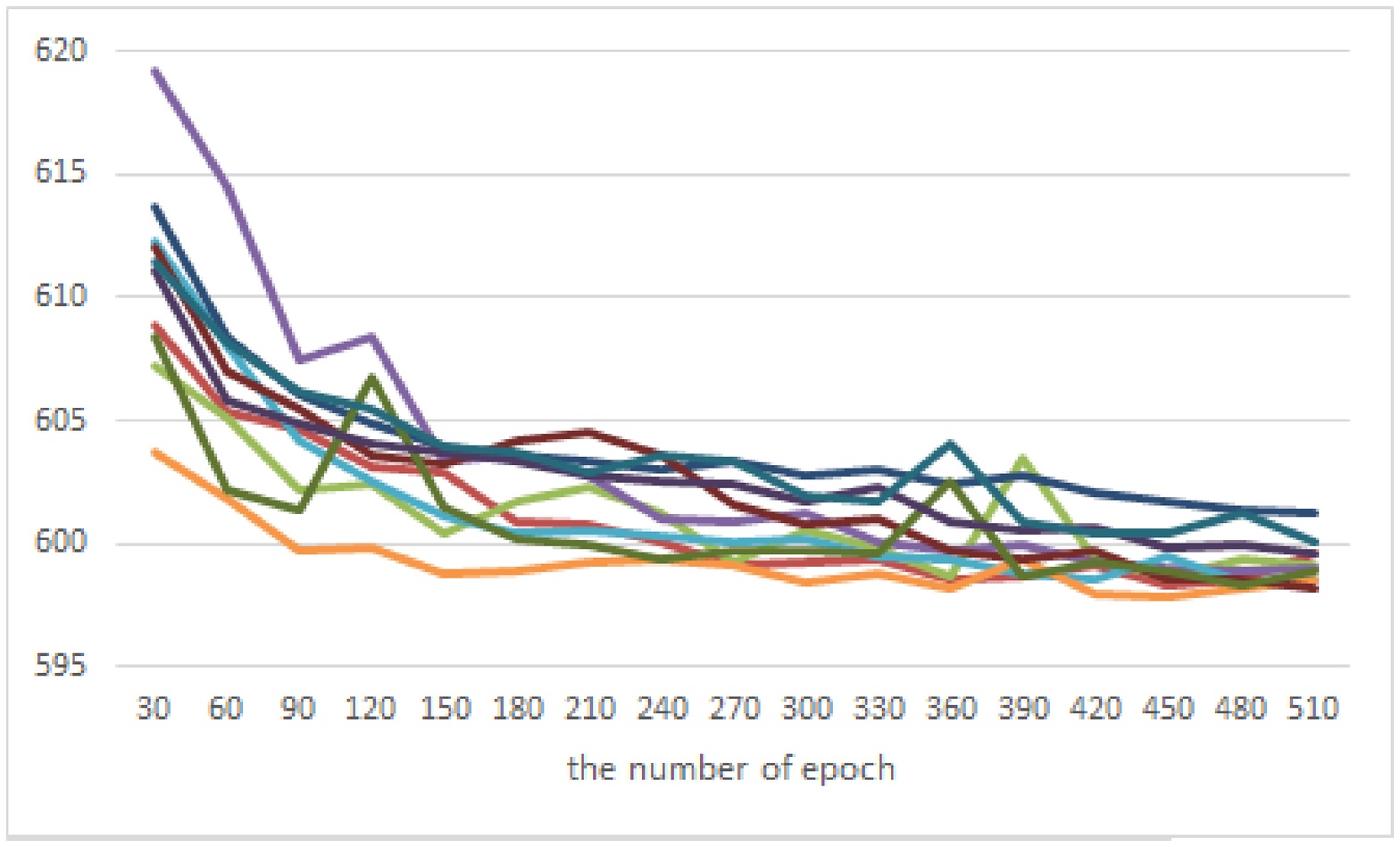}
\end{center}
\caption{Transition of $H_n^{\rm {\bf model}}-H_n(\hat{\beta}_n)$ for ten training results}
\label{fig4}
\end{figure}

\begin{table}
\caption{MSE median for each estimator}
\label{table3}
\begin{center}
\footnotesize
\begin{tabular}{|c|ccccc|} \hline
Epochs & $100$ & $300$ & $500$ & $1000$ & $3000$  \\ \hline
${\rm MSE}_1(\hat{\beta}_n)$ & $0.2402$ & $0.1763$ & $0.1521$ & $0.1264$ & $0.1066$ \\
${\rm MSE}_2(\hat{\beta}_n)$ & $0.2602$ &  $0.1868$ & $0.1580$ & $0.1289$ & $0.1067$ \\
${\rm MSE}_1(\dot{\beta}_n)$ & $0.2643$ & $0.2094$ & $0.1685$ & $0.1526$ & $0.1234$ \\
${\rm MSE}_2(\dot{\beta}_n)$ & $0.2873$ &  $0.2271$ & $0.1798$ & $0.1609$ & $0.1270$ \\ \hline
${\rm MSE}_1^{{\rm model}}$ & \multicolumn{5}{|c|}{$0.0946$} \\
${\rm MSE}_2^{{\rm model}}$ & \multicolumn{5}{|c|}{$0.1127$} \\ \hline
\end{tabular} 
\end{center}
\end{table}
\begin{discuss}
$x$が0.1ずれると$x=2$近辺が追加されるからmodel MLEのMSEが大きくずれているのは自然だろう．
NNのMSEがずれないのはおそらくmodel MLEと違って原点付近の精度が他の部分に比べていいわけではないから．
$x=1$付近はたぶんMLEより良いか．２次元でRの最適化が不安定になっていることも影響しているかもしれない．
\end{discuss}

\begin{figure}[htb]
\begin{center}
\includegraphics[width=90mm]{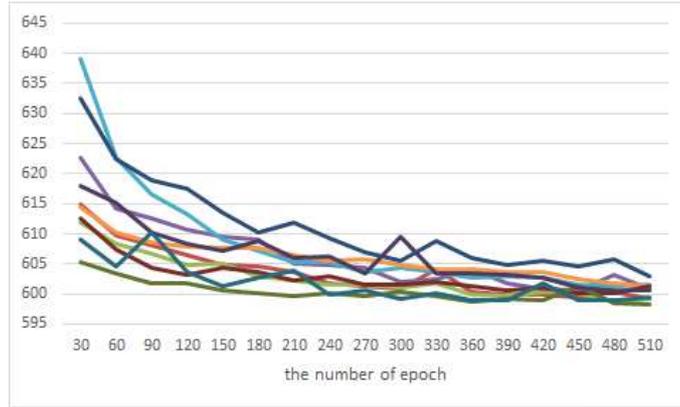}
\end{center}
\caption{$H_n^{\rm {\bf model}}-H_n(\dot{\beta}_n)$ for ten training results}
\label{fig5}
\end{figure}

Figure~\ref{fig4} shows the transition of $H_n^{\rm {\bf model}}-H_n(\hat{\beta}_n)$,
Figure~\ref{fig5} shows that of $H_n^{\rm {\bf model}}-H_n(\dot{\beta}_n)$,
and Table~\ref{table3} shows the MSE of each estimator, where $t_j=0.05\times j$ and
\begin{equation*}
{\rm MSE}_k(\beta)=\sqrt{\frac{1}{8000}\sum_{i,j,l=1}^{20}\sum_{m,m'=1}^2\left[\Sigma(t_j,x_i^{(k)},x_l^{(k)},\beta)-\Sigma_\dagger(t_j,x_i^{(k)},x_l^{(k)})\right]^2_{m,m'}}
\end{equation*}

\begin{discuss}
{\colorr two-dim CIRのmodel MLEなら
\begin{eqnarray}
{\rm MSE}_k&=&\sqrt{\frac{1}{8000}\sum_{i,j,l=1}^{20}\bigg(t_j^2-\frac{4}{3}t_j+\frac{2}{3}\bigg)^4
\left[(\hat{\sigma}_1^2-1)^2(x_i^{(k)})^2+(\hat{\sigma}_2^2+\hat{\sigma}_3^2-1)^2(x_l^{(k)})^2+(\hat{\sigma}_1\hat{\sigma}_3-0.5)^2x_i^{(k)}x_l^{(k)}\right]} \nonumber \\
&=&\sqrt{\frac{1}{20}\bigg\{\sum_{j=1}^{20}\bigg(t_j^2-\frac{4}{3}t_j+\frac{2}{3}\bigg)^4\bigg\}\left[(\hat{\sigma}_1^2-1)^2\bar{x^2}+(\hat{\sigma}_2^2+\hat{\sigma}_3^2-1)^2\bar{x^2}+(\hat{\sigma}_1\hat{\sigma}_3-0.5)^2\bar{x}^2\right]}. \nonumber
\end{eqnarray}
}
\end{discuss}

We can see that the calculation time of $\dot{\beta}_n$ is much shorter than that of $\hat{\beta}_n$:
training $3000$ times takes $12.74$ s for $\dot{\beta}_n$ but $1217.41$ s for $\hat{\beta}_n$.
This tendency becomes stronger as the dimensionality of $Y$ is increase.

\section{Empirical analysis}\label{empirical-section}
\begin{discuss2}
{\colorr 参考文献はAitEtAl JE2011, AitMiklandZhang, Hansen Lunde 2005(実証系のHansen＿RV),ChristensenKinnebrockPodolskij2010}
\end{discuss2}
\begin{discuss}
{\colorr 先行研究を見てもデータベースはここで明らかにしなくてもよさそう}
\end{discuss}

We trained on high-frequency data in the Tokyo Stock Exchange by using a neural network.
Our data comprise transaction data about five major stocks listed in the Tokyo Stock Exchange:
Nissan Motor Co., Ltd. (Nissan; trading symbol: 7201.T), Toyota Motor Corporation (Toyota; trading symbol: 7203.T), 
Honda Motor Co., Ltd. (Honda; trading symbol: 7267.T), Mitsubishi UFJ Financial Group, Inc. (MUFG; trading symbol: 8306.T),
and Nippon Telegraph and Telephone Corporation (NTT; trading symbol: 9432.T).

We train $\Sigma_{t,\dagger}$ by using the neural network $\Sigma(t,X_t,\beta)$ defined in Section~\ref{nn-example-section}, 
changing the explanatory process $X_t$,
and find useful factors for forecasting $\Sigma_{t,\dagger}$. 
We use transaction data over all trading days between January 2016 and December 2016.
We calculate $\hat{\beta}_n$ from three months of high-frequency data of each stock, and forecast $\Sigma_{t,\dagger}$ for each trading day in the subsequent month.
Though we cannot observe $\Sigma_{t,\dagger}$, examining (\ref{Hn-conv-toD}) shows that the differences of $H_n(\beta)$ are approximations of 
the differences of $D(\Sigma(\beta),\Sigma_\dagger)$, and $D$ evaluates a distance between $\Sigma(\beta)$ to $\Sigma_\dagger$.
The predictive capability of the estimators is thus assessed by $H_n(\beta)$. 
\begin{discuss2}
{\colorr 不要：We also approximate $\Sigma_{t,\dagger}$ by using the pre-averaged estimator
proposed by ... and calculated distance from each forecasted volatility $\hat{\Sigma}_n$ and $\Sigma_\dagger$ by MSE. }
\end{discuss2}

We set $h$, $K$, $L_n$, $\epsilon$, and $\hat{v}_n$ as in Section~\ref{CIR-simu-section} and use ADADELTA and weight decay with parameter $0.005$.
We randomly pick one-day data from the three months of data at each optimization step, setting the number of epochs to $3,000$.
We divide one-day data into the morning period (0900--1130: a.m.) and afternoon period (1230--1500: p.m.) because the Tokyo Stock Exchange is closed for lunch.
The estimator $\dot{\sigma}_n$ is used for training, and we adopt the estimator that attains the best value of $H_n$ over three optimization trials, randomly changing the initial value.
We use high-frequency data of the volatility index (VI) as the explanatory process $X_t$, this is the implied volatility calculated by using the option price of the Nikkei 225 index.
\begin{discuss2}
{\colorb 余：詳細な説明入れたい}
\end{discuss2}

For comparison, we also calculate an estimator obtained by the polynomial functions
\begin{equation}
\Sigma(t,x,\beta)=\bigg(\beta_0+\sum_{j=1}^p(\beta_jt^j+\beta_{j+p}x^j)\bigg)^2+\epsilon
\end{equation}
for $\beta=(\beta_j)_{j=0}^{2p}$. 
We change $p$ = $1$, $2$, and $3$ (for Poly1, Poly2, Poly3, respectively) and compare the performance with that when using the neural network (NeNet).
Since optimization of polynomial models are unstable compared to neural network, we calculate $\dot{\beta}_n$ ten times and select the best result.

\begin{table}
\caption{Average values of $-H_n(\dot{\beta}_n)/10000$ in one-dimensional prediction}
\label{table4}
\begin{center}
\footnotesize
\begin{tabular}{|r|rrrr|rrrr|} \hline
 & \multicolumn{4}{c}{a.m.}   & \multicolumn{4}{|c|}{p.m.}  \\ \cline{2-9}
\multicolumn{1}{|c|}{code} & NeNet & Poly1 & Poly2 & Poly3 & NeNet & Poly1 & Poly2 & Poly3 \\ \hline
7201 & -20.39 & -18.03 & -18.51 & -18.38 & -23.44 & -20.07 & -20.30 & -20.18 \\
7203 & 11.83 & 51.20 & 45.44 & 44.61 & 9.18 & 62.15 & 58.03 & 57.46 \\
7267 & 1.82 & 11.17 & 10.24 & 10.48 & -0.32 & 11.94 & 10.84 & 11.15 \\
8306 & -102.50 & -98.98 & -99.21 & -99.31 & -84.64 & -67.27 & -68.24 & -67.70 \\
9432 & 5.70 & 17.16 & 16.01 & 15.86 & 3.95 & 17.68 & 16.27 & 15.78 \\ \hline
\end{tabular} 
\end{center}
\end{table}

Table~\ref{table4} shows the average values of $-H_n(\dot{\beta}_n)/10000$ for each day from April 1st, 2016 to December 30th, 2016.
We can see that the neural network outperforms the polynomial models in each case (smaller is better).
The neural network can express complex models since we can stably optimize the parameter by using back propagation.

We also trained $\Sigma_{t,\dagger}$ for the two-dimensional process $Y$ that corresponds to each stock pair of above five stocks.
$X$ is the same as above, we let $K=3$, $L_1=L_2=30$ and set weight decay with parameter $0.01$.
Table~\ref{table5} shows the average values of $-H_n(\dot{\beta}_n)/10000$ under these settings.
\begin{table}
\caption{Average values of $-H_n(\dot{\beta}_n)/10000$ in two-dimensional prediction}
\label{table5}
\begin{center}
\footnotesize
\begin{tabular}{|r|rrrr|rrrr|} \hline
 & \multicolumn{4}{|c|}{a.m.}   & \multicolumn{4}{|c|}{p.m.}  \\ \cline{2-9} 
\multicolumn{1}{|c|}{code pair} & NeNet & Poly1 & Poly2 & Poly3 & NeNet & Poly1 & Poly2 & Poly3 \\ \hline 
7201-7203 & -8.65 & 34.77 & 25.82 & 23.05 & -14.73 & 30.69 & 27.50 & 26.85 \\ 
7201-7267 & -17.79 & -6.84 & -7.68 & -7.81 & -22.84 & -9.47 & -10.81 & -11.00 \\  
7201-8306 & -109.48 & -102.51 & -102.69 & -102.97 & -110.17 & -92.46 & -92.63 & -91.19 \\
7201-9432 & -13.81 & 0.53 & -0.60 & -1.78 & -18.72 & -2.98 & -5.28 & -5.30 \\
7203-7267 & 12.12 & 61.58 & 57.41 & 55.05 & 7.00 & 66.19 & 63.61 & 61.94 \\
7203-8306 & -80.31 & -32.45 & -37.51 & -39.19 & -80.33 & -10.04 & -12.85 & -16.64 \\
7203-9432 & 16.45 & 67.59 & 63.92 & 60.48 & 11.05 & 72.96 & 69.23 & 68.83 \\
7267-8306 & -88.94 & -73.24 & -74.43 & -74.65 & -88.16 & -58.67 & -61.78 & -61.41 \\
7267-9432 & 7.23 & 28.67 & 27.04 & 26.28 & 2.88 & 25.79 & 23.63 & 23.29 \\
8306-9432 & -84.86 & -67.33 & -68.41 & -69.17 & -84.36 & -51.64 & -52.22 & -53.64 \\ \hline
\end{tabular} 
\end{center}
\end{table}

\begin{discuss}
{\colorr
Weight decayは損失関数を$L'=L+\lambda|w|^2$とすることで微分すると、$\partial_w L'=\partial_w L+2\lambda w$となり、元の$w$を少し減らした更新値になる。
AdaDeltaはgradientを指数平滑で更新して過去のΔxの分散で調整した学習率を設定する。これによりAdaGradで見られるような急速かつ単調な学習率の低下を防ぐ。学習率は勾配が小さくてもほどほどに学習するように勾配の分散で割っている。
☆ニューラルネットはパラメータの次元をあげてもそれほど計算時間がかわらないことが大きなメリット。

Topix分析がわかると何がいいのか。→まずボラの構造がわかる。特に日内でどういう構造か。→当然予測精度改善につながるだろうし要因分析的なこともできる。
$X_t$は得られないけど、時刻や曜日の効果。ノンパラでやるよりモデルがわかっている方が何がノイズかわかるし、どういう時相関が高まるかわかる→より正確なシナリオへ。ファーマフレンチの高度化も。直近に重みを付けたり自由自在。	
要因を分解していればどこがノイズかどうかわかってデータから効率的に推定できるし、大規模共分散やるときの計算コストを削減できる。マルチファクターモデルというのはそういう発想。ただ、$X_t$に依存して変わるとわかるだけでも前日得データの$X_t$を入れればヒストリカル共分散より精度が上がる。
}
{\colorr

＜やること＞
\begin{enumerate}
\item 領域会議資料やメモから説明を強化する. Remark 3.1の論文調べる．
\item ２章以降を整えて論文にしていく. [A1]$\to $ Assumption (A1).とするか考える．各ステップで多次元場合を考える．Conclusion入れるか検討
\item 実証などで何を追加するか考えて追加．バイアス補正で推定の数値が改善するなら説得力がでてよい．
\item イントロは最後にもう一度きっちり考えて整える．領域会議資料とか特許スライドも参考に．あとは今までの論文のreferee commentとか見ながら
\end{enumerate}
高速アルゴリズムに触れた方がニューラルネットをやっている感じがでるか。今泉さんにイントロチェックしてもらうか。
どういう問題意識からどういう分析を行ったか等の流れをちゃんと書く。 $D$で測る正当性をきっちりいう。AS通らなかったら論文は２本に分けてもいい。
\\

ニューラルネットワークのアプロ―チの必要性とミススぺの意義がアピールできれば一定の評価はされるだろう。新しくて理論的にも有意義でデータ分析上も面白い。
つながりをうまくやること、ミススぺの難しさとどう解決したか、凄そうにみせたい。

\noindent
＜バイアス項の下からの評価＞
左右の対称性をなくしたらシミュレーションでバイアスを出せるかも。
理論的にバイアスの存在が言えるならそれがベスト。無理なら数値的に言う。
一般の$\sigma_{1,n}$と$\sigma_{2,n}$に対してスぺではなかったバイアスがでていることを示せば十分か．
やれるならやってもいいが，これで遅れても嫌だから時間が空いたときにやる程度にする．ある程度やってできなかったらそれはしょうがない
サンプリングスキームを変えてやる。Nを超えて比較。
\\

\noindent
＜実証方針＞
既存の推定量との比較は入れたいか。シミュレーションではMLEと比較するか、RMSEは入れる、ノイズ処理無しの場合も一応計算するか、あともう少し複雑なモデルも検討。
Bibinger et al.(2014)のグラフの形式は参考にする。日本株式市場の実証でもノイズ無しもやってみるか。比較分析は入れたいところだが比較できる手法がそれほど多くない。
何か面白い知見が得られれば十分だろう。みんなそこまできっちりやっているわけではない印象。理論がメインだからそれはそうか。層の比較や説明変数の変更等でバリエーションを出す。
FanLiYu2012: 200日シミュレーションして101日目から保有。低頻度は100日のヒストリカル共分散。高頻度は10日のtwo scale. ダウジョーンズ30銘柄でもうまくいく。さきがけの成果でも使えそうだし、論文としてもいいだろう。

（実証の参考文献）Bibinger et al2014、AitFanXiu。機械学習系は機械学習フォルダのTimeSeriesFinance\\

＜メモ＞\\
・「ニューラルネットワークは相性がどういいのか」→論文に対する感想としても現れるか。\\
・２×２行列の平方根の明示公式：https://\\
・☆論文を進めるときは査読者の疑問になるであろう点にこたえていくように進めていく\\
・ＳＤＥのＡＮＮはidentifiabilityがないことも重要。あとは滑らかさもない。\\
・ミススぺの先行研究では極限がrandomにならないから極限のuniquenessは仮定してしまっているか．\\
・ＡＮＮで一意性を言うのは厳しそう。\\
・さきがけでは機械学習をアピールするが、論文としてはミススぺを強調してＡＮＮを一例としてあげればいい。\\
・高次元でSDEのANNをやる時はブラウン運動の数を少なくすればbの関数を求める時に１０００×２０とかの計算ですむ。\\
・一意性はあまりきれいに書けないし、長くなるので、論文には入れない。今後の課題とする。とりあえずはANNが主題\\
・$\theta$は有限次元でやる．無限次元に近づくケースでも行けるかもしれないが，SVMでの数値計算を見るとカーネル関数の次元が大きすぎるとΓが退化しやすい。特に一次元は。\\
・taqどうするか決める。アクセプトには必要か。二か月分でいいか。まあとりあえず投稿して突っ込まれたら書くか\\
・ソボレフの不等式使うなら$\Lambda$の有界性は必要か\\
・softplusはReLUに近い。\\
・ＲＮＮは逐次的に求めていくが今回のアプローチは局所ガウスを仮定する→高速・安定。ガウス過程はガウスしか扱えず金融データには使えない。
→これらをアピールして尤度型のＳＤＥへのアプローチとして売り出してもいい。その場合イントロは鈴木先生に協力してもらいたい。A new approach to machine learning scheme for time series …\\
・ニューラルネットは化学反応の予測にも使われている
・ベイズ型推定を作れば非凸最適化の問題がおきないし，PLDを使えば作れるが、モーメントの計算がハードなので今回はやめておく。
・高速化のためには最尤推定をやめてワンステップにするという可能性があるが、ニューラルネットだから勾配法的な操作は必要。一意じゃないからワンステップも使えない。
・シーブ法とくらべたANNのメリットとしては、パラメータの多さのわりに計算が高速であること。特に疑似尤度の計算回数が少なくなる点で相性がいい。\\
・データ数とパラメータ次元の兼ね合い考える。NNパラメータ数多すぎるんじゃないか。
→Alexnetは6000万パラメータ。Imagenetの1400万のデータを判別している。実際はパラメータがスパースに入っているということだろう。
また、小さいNNで再学習というのもあるので、実際有効なものはそれほど多くないだろう。\\
・高速推定量は$n^{3/8}$以外をやってみても面白いかも。\\

＜推定量が$\Lambda$の境界に近づく場合（$\sigma_\ast\in\partial \Lambda$）＞
境界にあってもテイラー展開するわけではないし，$\sigma$の摂動はしないので問題ない．条件はすべて$\sigma\in\partial \Lambda$も含めてdefする．
そもそも$\Lambda$がopen setである必要もないか．最大値の存在が保証されればいい．
誤差逆伝搬では微分なので$\Lambda$はopenにするか\\

条件はドル囲いしなくても定理のStatement以外で整合性がとれているし，定理の中で違っていても別に問題ないからそのままにする．
$o_p$はMasuda Ueharaで同じのが使われているからこのままでいいか

}
\end{discuss}

\section{Proofs}\label{proofs-section}

This section contains proofs of the main results. Some preliminary lemmas are proved in Section~\ref{aux-results-section}.
Section~\ref{consistency-proof-section} addresses the consistency results used in Theorem~\ref{consistency-theorem}.
The results related to the optimal rate of convergence (Proposition~\ref{Hn-diff-prop} and Theorem~\ref{optConvTheorem}) are shown in Section~\ref{opt-conv-proof-section}.
Theorem~\ref{fast-estimator-thm} is proved in Section~\ref{fast-estimator-proof-section}. 
In Section~\ref{mixed-normality-proof-section}, the results for asymptotic mixed normality (Theorem~\ref{mixed-normality-thm}) are discussed.

\subsection{Preliminary results}\label{aux-results-section}
In this subsection, we will prove Lemma~\ref{HtoTildeH-lemma}, which enables us to replace $H_n$ with the more tractable $\tilde{H}_n$ in the proof of the main results.

For random variables $({\bf X}_n)_{n\in \mathbb{N}}$ and $({\bf Y}_n)_{n\in \mathbb{N}}$, we let ${\bf X}_n\approx {\bf Y}_n$ mean ${\bf X}_n-{\bf Y}_n\overset{P}\to 0$ as $n\to\infty$.
We use the symbols $C$ and $C_q$ for a generic positive constant that can vary from line to line ($C_q$ depends on $q$, with $q>0$).

We will prove the main results by replacing Conditions [A1] and [B1] with [A1$'$] and [B1$'$], respectively, using localization techniques to those used in proving Lemma 4.1 of Gobet~\cite{gob01}.

\begin{description}
\item{[$A1'$]} $[A1]$ is satisfied by $\mathcal{O}=\mathbb{R}^{\gamma_X}$, $\mu_t=0$ for any $t\in[0,T]$, the values $b_{t,\dagger}$, $b^{(j)}_t$, and $\hat{b}^{(j)}_t$ are bounded,
and there exists a positive constant $L$ such that
\begin{equation*}
\lVert \partial_\sigma^l\Sigma\rVert \vee \lVert \Sigma^{-1}\rVert(t,x,\sigma)\leq L, \quad \lVert \Sigma'^{-1/2}{\rm Abs}(\Sigma-\Sigma')\Sigma'^{-1/2}\rVert(t,x,\sigma) \leq 1-1/L,
\end{equation*}
\begin{equation*}
|\partial_\sigma^l\Sigma(s,x,\sigma)-\partial_\sigma^l\Sigma(t,y,\sigma)|\leq L(|x-y|+|s-t|)
\end{equation*}
for any $s,t \in [0,T],x,y\in \mathbb{R}^{\gamma_X},\sigma\in\Lambda$ and $l\in\{0,1\}$.

\end{description}
\begin{description}
\item{[$B1'$]} $[A1']$ and $[B1]$ are satisfied and there exists a positive constant $L'$ such that
$\lVert \partial_\sigma^l\partial_x\Sigma(t,x,\sigma)\rVert \leq L'$ and $|\partial_\sigma^l\partial_x \Sigma(t,x,\sigma)-\partial_\sigma^l\partial_x\Sigma(t,y,\sigma)|\leq L'|x-y|$ 
for any $t\in [0,T],x,y\in\mathbb{R}^{\gamma_X},\sigma\in \Lambda$ and $l\in\{0,1\}$.
\end{description}

\begin{discuss}
{\colorr $\partial_\sigma\Sigma$の有界性はLemmas~\ref{Phi-est-lemma},~\ref{Psi-est-lemma}で使うか．
そもそも$[A1']$を仮定する必要があるのはLemma \ref{ZSZ-est-lemma}を使うためで，モーメント条件を仮定していればこのLemmaは示せるからよい．
確率終息を扱う限りは$X_t$は有界な範囲だけで判断すればいい．$b_{t,\dagger},b_t^{(j)}$は絶対値がある$R$以上の部分を$0$に変えても条件は引き続き成立し，
それに対して成り立つなら問題ないだろう．Identifiabilityなどは成り立たなくなるが$\hat{\sigma}_n\to \sigma_\ast$などは引き続き成り立つ．
収束をいうときはGirsanov変換後のmeasureで収束を言えば元のmeasureでも収束するからよい．収束以外を示すときは注意が必要．
「任意の$\epsilon,\delta>0$に対しある$N$があって$Q[|X_n|\geq \delta]<\epsilon(n\geq N)$ならRadon-Nikodymで$f\in L^1$があって$P(A)=\int_AfdQ$なので
$P[|X_n|\geq \delta]<\epsilon$もOK.」
tightnessも互いに絶対連続なmeasureで引き継がれる．
}
\end{discuss}

\begin{discuss}
{\colorr $\tilde{S}$は$\Delta_n$のdefのところでdef済}
\end{discuss}
For a sequence $c_n$ of positive numbers, let us denote by $\{\bar{R}_n(c_n)\}_{n\in\mathbb{N}}$ a sequence of random variables (which may depend on $1\leq m\leq \ell_n$ and $\sigma\in \bar{\Lambda}$) satisfying
\begin{equation*}
E\bigg[\bigg(b_n^{-\delta}c_n^{-1}\frac{(\underbar{r}_n/b_n)^{p_2}(\underbar{k}_n/k_n)^{p_4}}{(r_n/b_n)^{p_1}(\bar{k}_n/k_n)^{p_3}}\sup_{\sigma,m}E_\Pi[|\bar{R}_n(c_n)|^q]^{1/q}\bigg)^{q'}\bigg]\to 0
\end{equation*}
as $n\to \infty$ for any $q,q',\delta>0$ and some constants $p_1,\cdots,p_4\geq 0$, where $\bar{k}_n=\max_{j,m}k^j_m$ and $\underbar{k}_n=\min_{j,m}k^j_m$.
\begin{discuss}
{\colorr ソボレフでモーメント評価もやっているし，やはり$\bar{R}_n$は定義しておいた方がいい．}
\end{discuss}

Let
\begin{equation*}
\tilde{H}_n(\sigma)=-\frac{1}{2}\sum_{m=2}^{\ell_n}(\tilde{Z}_m^{\top}\tilde{S}_m^{-1}(\sigma)\tilde{Z}_m+\log \det \tilde{S}_m(\sigma)).
\end{equation*}
In the proofs of the main results, we will replace $H_n$ with $\tilde{H}_n$.
In the following we introduce some lemmas to show that this replacement is acceptable.
Let $M_{m,\ast}={\rm diag}((v_{1,\ast}M_{1,m},\cdots, v_{\gamma,\ast}M_{\gamma,m}))$, 
\begin{equation}\label{D'-def-eq}
D'_m={\rm diag}((D'_{1,m},\cdots, D'_{\gamma,m})),\quad D'_{j,m}={\rm diag}((|I^j_{i,m}|)_i).
\end{equation}
\begin{lemma}\label{ZSZ-est-lemma}
Assume [A1$'$]. Let $\mathfrak{S}=\mathfrak{S}_{n,m}$ be a symmetric, $\mathcal{F}_{s_{m-1}}$-measurable random matrix of size $\sum_{j=1}^\gamma k^j_m$
satisfying $\lVert (D'_m+M_{m,\ast})\mathfrak{S}(D'_m+M_{m,\ast})\rVert\leq b_n^{-1}$ for $n\in \mathbb{N}$ and $1\leq m\leq \ell_n$.
Then there exists a random sequence $\{Q_{n,q}\}_{q\geq 2}$, not depending on $\mathfrak{S}$, such that $Q_{n,q}=\bar{R}_n(1)$ for $q\geq 2$, and the following hold.
\begin{enumerate}
\item $|E_m[(\tilde{Z}_m^{\top}\mathfrak{S}\tilde{Z}_m)^2]-2{\rm tr}((\mathfrak{S}\tilde{S}_{m,\dagger})^2)-{\rm tr}(\mathfrak{S}\tilde{S}_{m,\dagger})^2|\leq Q_{n,2}$,
$E_m[(\tilde{Z}_m^{\top}\mathfrak{S}\tilde{Z}_m)^4]\leq ((b_n^{-4}k_n^7)\vee (b_n^{-2}k_n^4))Q_{n,4}$ and $E_m[|\tilde{Z}_m^{\top}\mathfrak{S}\tilde{Z}_m|^q]\leq b_n^{-q}k_n^{2q}Q_{n,q}$ for $q>4$.
\item $E_{\Pi}[|\sum_m(Z_m-\tilde{Z}_m)^{\top}\mathfrak{S}(Z_m+\tilde{Z}_m)|^q]\leq (b_n^{-3}k_n^7)^{q/4}Q_{n,q}$ for $q\geq 4$.
\item $E_{\Pi}[|\sum_m(Z_m-\tilde{Z}_m)^{\top}\mathfrak{S}(Z_m+\tilde{Z}_m)|^2]\leq b_n^{-1}k_n^2Q_{n,2}$.
\end{enumerate}
\end{lemma}
This lemma is obtained similarly to Lemma 4.3 in~\cite{ogi18}. 
In contrast with Lemma 4.3 of~\cite{ogi18}, the smoothness of $b$ is not assumed here, and the assumptions for $b_\dagger$ are different from those in~\cite{ogi18}.
However, these differences do not affect the conclusion.

\begin{discuss}
{\colorr Ogi 2017 Lemma 4.3との違いは$b$のなめらかさを仮定しないことと，$b$と$b_{\dagger}$の仮定が違うこと．
$\mathfrak{S}$という一般形でやっているから$b$と$b_{\dagger}$の違いは問題ないし，$b$のlower boundくらいしか影響してこないか．→１．はＯＫ．
２．は$\mu\equiv 0$だからより楽．$\Psi_{2,m},\Psi_{3,m}$の評価は$\lVert \mathfrak{S}\rVert$の評価や$b_{\dagger}$の分解が使えるのでＯＫ．
$\lVert (D'_m+M_{m,\ast})\mathfrak{S}(D'_m+M_{m,\ast})\rVert $をtrの外に出すわけではないから$r_n/\underbar{r}_n$のロスは影響せずstrictな評価ができる．
\\

＜多変量の場合＞
Lemma 4.3が多変量で成り立つことをチェックする．Lemma A.8は証明を見れば二次元であることに依存していないからOK.
Lemma 4.3の1の証明は$1_{\max_{1\leq j\leq 4}i_j\leq k^1_m, \cdots}$が少し変わって$i_j$が全部同じブロックという条件になるが問題なく，
${\rm tr}(\hat{{\bf S}}\mathcal{E}_{(j)}\hat{{\bf S}}\mathcal{E}_{(j)})$の評価も同様．よって$|E_m[(\tilde{Z}_m^\top{\bf S}\tilde{Z}_m)^2]-2{\rm tr}() -{\rm tr}()^2|\leq Q_{n,2}$はOK.
$E_m[|\tilde{Z}_m^\top{\bf S}\tilde{Z}_m|^q]$や$E_m[|\tilde{Z}_m^\top{\bf S}\tilde{Z}_m|^4]$の評価もOK. 2. 3.の評価も特に問題なくOK.
}
\end{discuss}

\begin{remark}\label{ZSZ-est-rem}
For random matrices $\mathfrak{S}_1$ and $\mathfrak{S}_2$, we have
\begin{eqnarray}
E_m[(\tilde{Z}_m^\top(\mathfrak{S}_1+\mathfrak{S}_2)\tilde{Z}_m)^2]&=&E_m[(\tilde{Z}_m^\top\mathfrak{S}_1\tilde{Z}_m)^2]+E_m[(\tilde{Z}_m^\top\mathfrak{S}_2\tilde{Z}_m)^2] \nonumber \\
&&+2E_m[\tilde{Z}_m^\top\mathfrak{S}_1\tilde{Z}_m\tilde{Z}_m^\top\mathfrak{S}_2\tilde{Z}_m]. \nonumber
\end{eqnarray}
Then, if both $\mathfrak{S}_1$ and $\mathfrak{S}_2$ satisfy the conditions of $\mathfrak{S}$ in Lemma~\ref{ZSZ-est-lemma}, we obtain
\begin{equation*}
\bigg|E_m\bigg[\prod_{j=1}^2(\tilde{Z}_m^\top\mathfrak{S}_j\tilde{Z}_m)\bigg]-2{\rm tr}\bigg(\prod_{j=1}^2(\tilde{S}_{m,\ast}\mathfrak{S}_j)\bigg)-\prod_{j=1}^2{\rm tr}(\tilde{S}_{m,\ast}\mathfrak{S}_j)\bigg|\leq \tilde{Q}_{n,2},
\end{equation*}
where $\tilde{Q}_{n,2}=\bar{R}_n(1)$.
\end{remark}
\begin{discuss}
{\colorr 
\begin{eqnarray}
&&2{\rm tr}(((\mathfrak{S}_1+\mathfrak{S}_2)\tilde{S}_{m,\dagger})^2)+{\rm tr}((\mathfrak{S}_1+\mathfrak{S}_2)\tilde{S}_{m,\dagger})^2
-\sum_{j=1}^2(2{\rm tr}((\mathfrak{S}_j\tilde{S}_{m,\dagger})^2)+{\rm tr}(\mathfrak{S}_j\tilde{S}_{m,\dagger})^2) \nonumber \\
&&\quad =4{\rm tr}(\mathfrak{S}_1\tilde{S}_{m,\dagger}\mathfrak{S}_2\tilde{S}_{m,\dagger})+2{\rm tr}(\mathfrak{S}_1\tilde{S}_{m,\dagger}){\rm tr}(\mathfrak{S}_2\tilde{S}_{m,\dagger}). \nonumber
\end{eqnarray}
}
\end{discuss}

\begin{remark}\label{ZSZ-est-rem2}
Suppose that the assumptions in Lemma~\ref{ZSZ-est-lemma} are satisfied. Then for $q>4$, Lemma~\ref{ZSZ-est-lemma}, the Burkholder--Davis--Gundy inequality
and Jensen's inequality yield
\begin{eqnarray}
&&E_\Pi\bigg[\bigg|\sum_{m=2}^{\ell_n}\bar{E}_m[\tilde{Z}_m^\top \mathfrak{S}\tilde{Z}_m]\bigg|^q\bigg]
\leq C_qE_\Pi\bigg[\bigg(\sum_{m=2}^{\ell_n}\bar{E}_m[\tilde{Z}_m^\top\mathfrak{S}\tilde{Z}_m]^2\bigg)^{q/2}\bigg] \nonumber \\
&&\quad \leq C_qE_\Pi\bigg[\bigg(\sum_{m=2}^{\ell_n}E_m[\bar{E}_m[\tilde{Z}_m^\top \mathfrak{S}\tilde{Z}_m]^2]\bigg)^{q/2}\bigg] \nonumber \\
&&\quad \quad +C_qE_\Pi\bigg[\bigg(\sum_{m=2}^{\ell_n}\bar{E}_m[\bar{E}_m[\tilde{Z}_m^\top \mathfrak{S}\tilde{Z}_m]^2]^2\bigg)^{q/4}\bigg] \nonumber \\
&&\quad \leq C_qE_\Pi\bigg[\bigg(\sum_{m=2}^{\ell_n}{\rm tr}((\mathfrak{S}\tilde{S}_{m,\dagger})^2)\bigg)^{q/2}\bigg] \nonumber \\
&&\quad \quad +\bar{R}_n(\ell_n^{\frac{q}{2}})+C_q(\ell_n-1)^{\frac{q}{4}-1}\sum_{m=2}^{\ell_n}E_\Pi[(\tilde{Z}_m^\top \mathfrak{S}\tilde{Z}_m)^q] \nonumber \\
&&\quad \leq C_qE_\Pi\bigg[\bigg(\sum_{m=2}^{\ell_n}{\rm tr}((\mathfrak{S}\tilde{S}_{m,\dagger})^2)\bigg)^{q/2}\bigg] 
+\bar{R}_n(\ell_n^{\frac{q}{2}}+b_n^{-\frac{3}{4}q}k_n^{\frac{7}{4}q}). \nonumber
\end{eqnarray}
\end{remark}

\begin{lemma}\label{HtoTildeH-lemma}
Assume [A1$'$], [V], and that $r_n=O_p(b_n^{-1+\epsilon})$ and $\underbar{r}_n=O_p(b_n^{1+\epsilon})$ for any $\epsilon>0$. Then
\begin{equation}\label{HtoTildeH-conv}
b_n^{-1/2}\sup_{\sigma_1,\sigma_2\in \bar{\Lambda}}|(H_n(\sigma_1,\hat{v}_n)-H_n(\sigma_2,\hat{v}_n))-(\tilde{H}_n(\sigma_1)-\tilde{H}_n(\sigma_2)|\overset{P}\to 0.
\end{equation}
Moreover, if [B1$'$] is also satisfied and $\ell_nb_n^{-3/7-\epsilon}\to \infty$ as $n\to \infty$ for any $\epsilon>0$, then
\begin{equation}\label{HtoTildeH-conv2}
b_n^{-1/4}\sup_{\sigma_1,\sigma_2\in \bar{\Lambda}}|(H_n(\sigma_1,\hat{v}_n)-H_n(\sigma_2,\hat{v}_n))-(\tilde{H}_n(\sigma_1)-\tilde{H}_n(\sigma_2))|\overset{P}\to 0.
\end{equation}
\end{lemma}
\begin{discuss}
{\colorr $\hat{v}_n\to v_\ast$のlog detの評価が問題になって，$H_n(\sigma_{j,n},\hat{v}_n)-\tilde{H}_n(\sigma_{j,n})$の評価ではレートが出ないか．}
\end{discuss}

\begin{proof}
We obtain 
\begin{equation*}
b_n^{-1/4}\sup_{\sigma_1,\sigma_2\in\bar{\Lambda}}|(H_n(\sigma_{1,n},\hat{v}_n)-H_n(\sigma_{2,n},\hat{v}_n)-H_n(\sigma_{1,n},v_\ast)+H_n(\sigma_{2,n},v_\ast))|\overset{P}\to 0
\end{equation*}
by an argument similar to that in the proof of Lemma 4.4 in~\cite{ogi18}. Moreover, we can make the following decomposition:
\begin{equation*}
H_n(\sigma_1,v_\ast)-H_n(\sigma_2,v_\ast)-\tilde{H}_n(\sigma_1)+\tilde{H}_n(\sigma_2)=\sum_{j=1}^3\hat{\Psi}_{j,n}(\sigma_{1,n},\sigma_{2,n}),
\end{equation*}
where $\hat{\Psi}_{j,n}$ is defined similar to that in Lemma 4.4 in~\cite{ogi18} for $1\leq j\leq 3$.

Sobolev's inequality yields
\begin{equation*}
E\Big[\sup_{\sigma_1,\sigma_2}|\hat{\Psi}_{1,n}(\sigma_1,\sigma_2)|^q\Big]\leq C_q\sup_{\sigma_1,\sigma_2}\sum_{\substack{0\leq l_1,l_2\leq 1 \\ l_1+l_2\leq 1}}E[|\partial_{\sigma_1}^{l_1}\partial_{\sigma_2}^{l_2}\hat{\Psi}_{1,n}(\sigma_1,\sigma_2)|^q]
\end{equation*}
for sufficiently large $q>0$, and then Point 2 of Lemma~\ref{ZSZ-est-lemma} yields $\sup_{\sigma_1,\sigma_2}|\hat{\Psi}_{1,n}(\sigma_1,\sigma_2)|=\bar{R}_n(b_n^{-3/4}k_n^{7/4})$. 
\begin{discuss}
{\colorr $\hat{v}_n\to v_{\ast}$は(4.7)右辺第二項の評価は$\lVert {\bf S}\rVert$の評価と$|Z_m|$の評価しか使わないのでOKで，
第一項はlog det項が$\bar{R}_n(b_n^{1/2}\ell_n^{-1})$で$j\geq 2$の時$\bar{R}_n(k_nb_n^{-1}\ell_n)=o_p(b_n^{1/4})$がわかる．
\begin{equation*}
E[|b_n^{-3/4}\sum_mZ_m^\top \partial_\sigma^l\partial_v{\bf C}(\sigma)Z_m|^q]=O_p((b_n^{-3/4}\ell_nb_n^{-1}k_n^2)^q)=O_p((b_n^{-3/4}k_n)^q)\to 0.
\end{equation*}
}
\end{discuss}
Moreover, similarly to Lemma 4.4 in~\cite{ogi18}, we have
$b_n^{-1/2}\sup_{\sigma_1,\sigma_2}|\hat{\Psi}_{j,n}(\sigma_1,\sigma_2)|\overset{P}\to 0$
for any $\epsilon>0$ and $j\in\{2,3\}$, and therefore, we obtain (\ref{HtoTildeH-conv}).
\begin{discuss}
{\colorr $b_n^{-1/2}$を掛けた評価は$\partial_x\Sigma$の代わりに$\Sigma$のりぷしっつ連続性を使えば十分}
\end{discuss}

If further [B1$'$] is satisfied and $\ell_nb_n^{-3/7-\epsilon}\to \infty$ as $n\to \infty$ for any $\epsilon>0$,
we can similarly obtain (\ref{HtoTildeH-conv2}).
\begin{discuss}
{\colorr Lemma 4.4 in~\cite{ogi18}の$E_\Pi[|\partial_\sigma^l\Psi_{j,n}(\sigma_1,\sigma_2)|^q]=O_p(b_n^{-q+\epsilon}k_n^{2q})$からわかる．}
\end{discuss}
\end{proof}

\begin{discuss}
{\colorr Lemma \ref{ZSZ-est-lemma}より先の議論の代わりにソボレフを使えば
\begin{equation*}
\sup_{\sigma_1,\sigma_2}E_{\Pi}\Big[\big|b_n^{-1/4}(H_n(\sigma_1,\hat{v}_n)-H_n(\sigma_2,\hat{v}_n)-\tilde{H}_n(\sigma_1)+\tilde{H}_n(\sigma_2)\big|^q\Big]^{1/q}=\bar{R}_n(b_n^{-1}k_n^{7/4})
\end{equation*}
より$k_n\leq b_n^{4/7-\epsilon}$なら同様の結果が示せる.
}
\end{discuss}

The following lemma gives a useful expansion of $\tilde{S}_m^{-1}(\sigma)$; the expansion will be repeatedly used. 
Let $\mathcal{K}^k_m=\sum_{1\leq l\leq k}k^l_m$.
Let 
$\mathcal{I}_p:=\{(i_0,\cdots,i_p)\in \{1,\cdots,\gamma\}^{p+1};i_{j-1}\neq i_j \ (1\leq j\leq p)\}$.
Let $\mathfrak{G}_{k,l}$ be a $k^k_m\times k^l_m$ matrix with elements $[\mathfrak{G}_{k,l}]_{ij}=|I^k_{i,m}\cap I^l_{j,m}|$,
${\bf D}_{k,m}(B,v)=[B]_{kk}D'_{k,m}+v_kM_{k,m}$,
and let ${\bf E}_{i,m}$ be a $\mathcal{K}_m^\gamma\times k^i_m$ matrix with elements
\begin{equation*}
[{\bf E}_{i,m}]_{jk}=\left\{
\begin{array}{ll}
1 & j-k=\mathcal{K}_m^{i-1} \\
0 & {\rm otherwise}
\end{array}
\right..
\end{equation*}

\begin{lemma}\label{invS-eq-lemma}
Let $B$ be a $\gamma\times \gamma$ symmetric, positive definite matrix 
with $\lVert B^{-1}\rVert \leq R$ and $\lVert B'^{-1/2}{\rm Abs}(B-B')B'^{-1/2}\rVert\leq 1-1/R$ for $B'={\rm diag}(([B]_{ii})_{i=1}^\gamma)$
with $R$ some positive constant. Then 
\begin{eqnarray}\label{invSm-exp}
&& \\
{\bf S}_m^{-1}(B,v)&=&\sum_{p=0}^\infty (-1)^p\sum_{(i_0,\cdots,i_p)\in\mathcal{I}_p}\bigg(\prod_{l=1}^p[B]_{i_{l-1},i_l}\bigg) \nonumber \\
&&\times {\bf E}_{i_0,m}\bigg(\prod_{l=1}^p{\bf D}_{i_{l-1},m}^{-1}(B,v)\mathfrak{G}_{i_{l-1},i_l}\bigg){\bf D}_{i_p,m}^{-1}(B,v){\bf E}_{i_p,m}^\top \nonumber
\end{eqnarray}
and
\begin{eqnarray}\label{log-det-Sm-exp}
&& \\
\log\det {\bf S}_m(B,v)
&=&\sum_{i=1}^\gamma \log\det {\bf D}_{i,m}(B,v) \nonumber \\
&&-\sum_{p=1}^\infty\frac{(-1)^p}{p}\sum_{\substack{(i_0,\cdots,i_p)\in\mathcal{I}_p \\ i_0=i_p}}
\prod_{l=1}^p[B]_{i_{l-1},i_l}{\rm tr}\bigg(\prod_{l=1}^p{\bf D}_{i_l,m}^{-1}\mathfrak{G}_{i_{l-1},i_l}\bigg) \nonumber
\end{eqnarray}
for $v\in (0,\infty)^\gamma$.

Moreover, 
\begin{equation*}
\sum_{p=r}^\infty \frac{p!}{(p-r)!}\sum_{(i_0,\cdots, i_p)\in\mathcal{I}_p} \prod_{l=1}^p
\bigg|\frac{[B]_{i_{l-1},i_l}}{[B]_{i_{l-1},i_{l-1}}^{1/2}[B]_{i_l,i_l}^{1/2}}\bigg|\leq \gamma^2\sum_{p=r}^\infty \frac{p!}{(p-r)!}\bigg(1-\frac{1}{R}\bigg)^p<\infty
\end{equation*}
for any nonnegative integer $r$.
\end{lemma}
\begin{discuss}
{\colorr 
このLemmaを使っているのはLemmas \ref{random-param-est},\ref{Phi-est-lemma}, \ref{Psi-est-lemma}, \ref{deterministic-st-conv-lemma} and \ref{Hnl-conv-lemma}.

Lemma~\ref{Hnl-conv-lemma}の$\partial_\sigma \tilde{S}_m^{-1}$のところでも使うから$\tilde{S}_m^{-1}$の式として書いておく

＜$\lVert B'^{-1/2}(B-B')B'^{-1/2}\rVert\leq 1-1/R$を外すことに関して＞
この条件は$\tilde{S}_m^{-1}$の級数展開の絶対収束を示すうえで必要で，絶対収束が必要なのはLemma \ref{Phi-est-lemma}の証明．
${\rm tr}(\prod \ddot{D}^{-1})\to \int dx$の残差は同符号になって絶対収束が要らないだろう．
${\rm tr}(\prod (D^{-1}\mathfrak{G})D^{-1})\to {\rm tr}(\prod \ddot{D}^{-1})$は$D\to \ddot{D}$を一つずつやればいいので
$k$毎に$S_{m,(k-1}^{-1}$の形のまま残差を評価できるから絶対収束がなくてもLemma \ref{Phi-est-lemma}の残差を評価できるだろう．
→これはリバイズ時に余裕があればやるか．
}
\end{discuss}

\begin{proof}

For $p\geq 2$ and $1\leq i_0,i_p\leq \gamma$, we obtain
\begin{eqnarray}
[(B'^{-1/2}{\rm Abs}(B-B')B'^{-1/2})^p]_{i_0,i_p}=\sum_{(i_1,\cdots,i_{p-1});(i_0,\cdots,i_p)\in \mathcal{I}_p}\prod_{1\leq l\leq p}\bigg|\frac{[B]_{i_{l-1},i_l}}{[B]^{1/2}_{i_{l-1},i_{l-1}}[B]^{1/2}_{i_l,i_l}}\bigg|. \nonumber
\end{eqnarray}
From this, we have
\begin{eqnarray}
&&\sum_{p=r}^\infty \frac{p!}{(p-r)!}\sum_{(i_0,\cdots,i_p)\in \mathcal{I}_p}\prod_{1\leq l\leq p}\bigg|\frac{[B]_{i_{l-1},i_l}}{[B]_{i_{l-1},i_{l-1}}^{1/2}[B]_{i_l,i_l}^{1/2}}\bigg| \nonumber \\
&&\quad \leq \gamma^2 \sum_{p=r}^\infty \frac{p!}{(p-r)!}\lVert B'^{-1/2}{\rm Abs}(B-B')B'^{-1/2}\rVert^p \nonumber \\
&&\quad \leq \gamma^2 \sum_{p=r}^\infty \frac{p!}{(p-r)!}\bigg(1-\frac{1}{R}\bigg)^p<\infty. \nonumber
\end{eqnarray}

Additionally, let ${\bf D}'_m={\rm diag}(({\bf D}_{i,m}(B,v))_i)$, ${\bf D}''_m={\rm diag}(([B]_{ii}D'_{i,m})_i)$, ${\bf D}'''_m={\bf D}''^{-1/2}_m{\bf D}'_m{\bf D}''^{-1/2}_m$
and ${\bf T}_m={\bf D}'''^{-1/2}_m{\bf D}''^{-1/2}_m({\bf S}_m-{\bf D}'_m){\bf D}''^{-1/2}_m{\bf D}'''^{-1/2}_m$.
\begin{discuss}
{\colorr $B$: p.d.で$\underbar{r}_n\geq 0$より$D''$のinvが存在}
\end{discuss}
Then we have
\begin{equation}\label{Sm-eq}
{\bf S}_m={\bf D}''^{1/2}_m{\bf D}'''^{1/2}_m(\mathcal{E}_{\mathcal{K}_m^\gamma}+{\bf T}_m) {\bf D}'''^{1/2}_m{\bf D}''^{1/2}_m,
\end{equation}
and
\begin{eqnarray}
\lVert {\bf T}_m\rVert^2
&\leq& \lVert {\bf D}''^{-1/2}_m({\bf S}_m-{\bf D}'_m){\bf D}''^{-1/2}_m\rVert^2 \nonumber \\
&=&\sup_{|x|=1}|{\bf D}''^{-1/2}_m({\bf S}_m-{\bf D}'_m){\bf D}''^{-1/2}_mx|^2 \nonumber \\
&=&\sup_{|x|=1}\sum_{i=1}^\gamma\bigg|\sum_{1\leq j\leq \gamma,j\neq i}\frac{[B]_{ij}}{[B]_{ii}^{1/2}[B]_{jj}^{1/2}}D'^{-1/2}_{i,m}\mathfrak{G}_{i,j}D'^{-1/2}_{j,m}x_j\bigg|^2 \nonumber \\
&\leq&\sup_{|x|=1}\sum_{i=1}^\gamma\bigg(\sum_{1\leq j\leq \gamma,j\neq i}\frac{|[B]_{ij}|}{[B]_{ii}^{1/2}[B]_{jj}^{1/2}}|D'^{-1/2}_{i,m}\mathfrak{G}_{i,j}D'^{-1/2}_{j,m}x_j|\bigg)^2 \nonumber \\
&\leq&\sup_{|x|=1}\sum_{i=1}^\gamma\bigg(\sum_{1\leq j\leq \gamma,j\neq i}\frac{|[B]_{ij}|}{[B]_{ii}^{1/2}[B]_{jj}^{1/2}}|x_j|\bigg)^2 \nonumber \\
&\leq&\lVert B'^{-1/2}{\rm Abs}(B-B')B'^{-1/2} \rVert^2\leq (1-1/R)^2, \nonumber 
\end{eqnarray}
where $x=(x_1,\cdots, x_\gamma)$ and $x_j\in \mathbb{R}^{k_m^j}$ for $1\leq j\leq \gamma$. 
To obtain this, we used the fact $\lVert D'^{-1/2}_{i,m}\mathfrak{G}_{i,j}D'^{-1/2}_{j,m}\rVert \leq 1$ for $1\leq i,j\leq \gamma$ by Lemma 2 in Ogihara and Yoshida~\cite{ogi-yos14}.

Using the above, we obtain
\begin{eqnarray}
{\bf S}_m^{-1}&=&{\bf D}''^{-1/2}_m{\bf D}'''^{-1/2}_m\sum_{p=0}^\infty (-1)^p{\bf T}_m^p{\bf D}'''^{-1/2}_m{\bf D}''^{-1/2}_m \nonumber \\
&=&\sum_{p=0}^\infty(-1)^p\{{\bf D}'^{-1}_m({\bf S}_m-{\bf D}'_m)\}^p{\bf D}'^{-1}_m \nonumber \\
&=&\sum_{p=0}^\infty(-1)^p \sum_{(i_0,\cdots,i_p)\in\mathcal{I}_p}\bigg(\prod_{l=1}^p[B]_{i_{l-1},i_l}\bigg) \nonumber \\
&&\times {\bf E}_{i_0,m}\bigg(\prod_{l=1}^p{\bf D}_{i_{l-1},m}^{-1}(B,v)\mathfrak{G}_{i_{l-1},i_l}\bigg){\bf D}_{i_p,m}^{-1}(B,v){\bf E}_{i_p,m}^\top. \nonumber
\end{eqnarray}

Moreover, applying (\ref{Sm-eq}) and Lemma A.3 of~\cite{ogi18} yields
\begin{eqnarray}
&&\log\det {\bf S}_m-\log\det {\bf D}'_m \nonumber \\
&&\quad = \log\det(\mathcal{E}_{\mathcal{K}_m^\gamma}+{\bf T}_m)=-\sum_{p=1}^\infty \frac{(-1)^p}{p}{\rm tr}({\bf T}_m^p) \nonumber \\
&&\quad = -\sum_{p=1}^\infty \frac{(-1)^p}{p}{\rm tr}(({\bf D}'^{-1}_m({\bf S}_m-{\bf D}'_m))^p), \nonumber
\end{eqnarray}
so (\ref{log-det-Sm-exp}) holds.

\end{proof}

Let $\tilde{D}_{k,m}=\tilde{D}_{k,m}(\sigma)=[\tilde{\Sigma}_m(\sigma)]_{kk}D'_{k,m}+v_{k,\ast}M_{k,m}$ and
$D'_{i,j,m}=\{|I^i_{k-\mathcal{K}^{i-1}_m}\cap I^j_{l-\mathcal{K}_m^{j-1}}|1_{\{\mathcal{K}_m^{i-1}<k\leq \mathcal{K}_m^i,\mathcal{K}_m^{j-1}<l\leq \mathcal{K}_m^j\}}\}_{1\leq k,l\leq \mathcal{K}^\gamma_m}$.
To find the limit of $\tilde{H}_n$, we first consider the limit of ${\rm tr}(\tilde{S}_m^{-1}(\sigma)D'_{i,j,m})$.

Let
\begin{equation*}
\Phi_{i,j,m}(\sigma_1)={\rm tr}(\tilde{S}_m^{-1}(\sigma_1)D'_{i,j,m})-\frac{b_n^{1/2}}{2}\int^{s_m}_{s_{m-1}}\sqrt{\tilde{a}_t^i\tilde{a}_t^j}[\mathcal{D}(t,\Sigma(\sigma_1))^{-1/2}]_{ij}dt,
\end{equation*}
and
\begin{equation*}
\Psi_m(\sigma_1,\sigma_2)=\log\frac{\det\tilde{S}_m(\sigma_1)}{\det\tilde{S}_m(\sigma_2)}-b_n^{1/2}\int^{s_m}_{s_{m-1}}({\rm tr}(\mathcal{D}(t,\Sigma(\sigma_1))^{1/2}-\mathcal{D}(t,\Sigma(\sigma_2))^{1/2})dt.
\end{equation*}
\begin{lemma}\label{Phi-est-lemma}
Assume [A1$'$] and [A2]. Let $\sigma_{1,n}$ be a $\bar{\Lambda}$-valued random variable. 
Then $b_n^{-1/2}\ell_n\max_{i,j,m}|\Phi_{i,j,m}(\sigma_{1,n})|\overset{P}\to 0$.
If [B1$'$] and [B2] are also satisfied, then $\max_{i,j,m}|\partial_\sigma^l\Phi_{i,j,m}(\sigma_{1,n})|=\bar{R}_n(b_n^{-5/8}k_n+1)$ for $0\leq l\leq 1$.
\end{lemma}
\begin{discuss}
{\colorr $\partial_\sigma \Phi_{j,m}$の評価はLemma \ref{log-likelihood-eq-lemma}の証明の最後でソボレフを使う時に必要．}
\end{discuss}

\begin{discuss}
{\colorr 
\begin{eqnarray}
\sup_{\sigma}\bigg|b_n^{-1/2}\partial_{\sigma}^k\sum_m\log\det(\tilde{S}_m\tilde{D}_m^{-1})
-\partial_{\sigma}^k\int^T_0\bigg\{\mathcal{U}_t(b_t)-\sqrt{\tilde{a}^1_t}|b_t^1|-\sqrt{\tilde{a}^2_t}|b^2_t|\bigg\}dt\bigg|\overset{P}\to 0, \nonumber \\
\sup_{\sigma}\bigg|b_n^{-1/2}\partial_{\sigma}^k\sum_m\log\frac{\det\dot{D}_{j,m}}{\det\dot{D}_{j,m,\ast}}
-\partial_{\sigma}^k\int^T_0\sqrt{\tilde{a}^j_t}(|b^j_t|-|b^j_{t,\ast}|)dt\bigg|\overset{P}\to 0. \nonumber
\end{eqnarray}
}
\end{discuss}

\begin{proof}
Lemma~\ref{invS-eq-lemma} yields
\begin{eqnarray}
{\rm tr}(\tilde{S}_m^{-1}D'_{i,j,m})=\sum_{p=0}^\infty(-1)^p\sum_{\substack{(i_0,\cdots, i_p)\in\mathcal{I}_p \\ i_0=i,i_p=j}}\prod_{l=1}^p[\tilde{\Sigma}_m]_{i_{l-1},i_l}
{\rm tr}\bigg(\prod_{l=1}^{p+1}\tilde{D}_{i_{l-1},m}^{-1}\mathfrak{G}_{i_{l-1},i_l}\bigg). \nonumber
\end{eqnarray}
Here, we let $i_{p+1}=i_0$.

Let $c'_j=\check{a}_m^j/(\check{a}_m^1)^2[\tilde{\Sigma}_m]_{jj}b_n^{-1}/v_{j,\ast}$ with $\check{a}_m^i=a_{s_{m-1}}^i$. 
Then, thanks to Lemma 5.2 of~\cite{ogi18}, we obtain
\begin{eqnarray}\label{tr-invD-est}
&& \\
&&{\rm tr}\bigg(\prod_{l=1}^{p+1}\tilde{D}_{i_{l-1},m}^{-1}\mathfrak{G}_{i_{l-1},i_l}\bigg) \nonumber \\
&&\quad =\bigg(\prod_{l=0}^p\frac{\check{a}_m^{i_l}}{v_{i_l,\ast}}\bigg)\frac{b_n^{-p-1}}{(\check{a}_m^1)^{2p+2}}{\rm tr}\bigg(\prod_{l=0}^p(c'_{i_l}\mathcal{E}_{k^1_m}+M_{1,m})^{-1}\bigg)+o_p(b_n^{1/2}\ell_n^{-1}) \nonumber \\
&&\quad =\bigg(\prod_{l=0}^p\frac{\check{a}_m^{i_l}}{v_{i_l,\ast}}\bigg)\frac{b_n^{-p-1}}{(\check{a}_m^1)^{2p+2}}
\frac{T\check{a}_m^1k_n}{\pi}\int^\pi_0\prod_{l=0}^p(c'_{i_l}+2(1-\cos x))^{-1}dx \nonumber \\
&&\quad \quad  + o_p(b_n^{1/2}\ell_n^{-1}). \nonumber 
\end{eqnarray}
\begin{discuss}
{\colorr 
\begin{equation*}
=\frac{T\check{a}_m^1k_n}{\pi b_n^{p+1}}\partial_\sigma^l\int^\pi_0\prod_{l=0}^p[\mathcal{F}_m(x)]_{i_l}^{-1}dx + o_p(b_n^{1/2}\ell_n^{-1})
\end{equation*}
$\mathfrak{G}_{i,j}$ひとつは$b_n^{-1}(\check{a}_m^1)^{-1}$, $\tilde{D}_{j,m}^{-1}$は$\check{a}^j_m(\check{a}_m^1v_{j,\ast})^{-1}$に変換される．}
\end{discuss}

Therefore, we obtain
\begin{eqnarray}\label{tr-invD-est2}
&& \\
{\rm tr}(\tilde{S}_m^{-1}D'_{i,j,m})&=&\sum_{p=0}^\infty(-1)^p\frac{T\check{a}_m^1k_n}{\pi b_n^{p+1}}\sum_{\substack{(i_0,\cdots, i_p)\in\mathcal{I}_p \\ i_0=i,i_p=j}}\prod_{l=1}^p[\tilde{\Sigma}_m]_{i_{l-1},i_l} \nonumber \\
&&\times \int^\pi_0\prod_{l=0}^p\bigg(\frac{[\Sigma_m]_{i_l,i_l}}{b_n}+\frac{2(\check{a}_m)^2v_{i_l,\ast}}{\check{a}_m^{i_l}}(1-\cos x)\bigg)^{-1}dx \nonumber \\
&&+ o_p(b_n^{1/2}\ell_n^{-1}) \nonumber
\end{eqnarray}
by Lemma~\ref{invS-eq-lemma}.

\begin{discuss}
{\colorr
Lemma 5.2の残差項が各$i_0,\cdots, i_p$に対して一様に評価されていればいいが，Lemma 5.1の残差は$\bar{k}_n$, $\underbar{k}_n$と[A2]の収束を抑えられればいいから
インディケーターは一様にとれる．Lemma 5.2の$p/p-1$と$\Lambda_1$も$\bar{k}_n$, $\underbar{k}_n$のコントロール，
$p\cdots p$の数変換と$\#\{\}$の変換が[A2]の収束でコントロールされるのでOK. ${\rm }tr() -I$項もOK.
}
\end{discuss}

Moreover, if [B1$'$] and [B2] are also satisfied, then by Remark 5.2 in~\cite{ogi18} and following the method of the proof of Lemma 5.2 in~\cite{ogi18}, we obtain
\begin{eqnarray}
\partial_\sigma^l{\rm tr}(\tilde{S}_m^{-1}D'_{i,j,m})&=&\sum_{p=0}^\infty(-1)^p\frac{T\check{a}_m^1k_n}{\pi b_n^{p+1}}\partial_\sigma^l\sum_{\substack{(i_0,\cdots, i_p)\in\mathcal{I}_p \\ i_0=i,i_p=j}}\prod_{l=1}^p[\tilde{\Sigma}_m]_{i_{l-1},i_l} \nonumber \\
&&\times \int^\pi_0\prod_{l=0}^p\bigg(\frac{[\Sigma_m]_{i_l,i_l}}{b_n}+\frac{2(\check{a}_m)^2v_{i_l,\ast}}{\check{a}_m^{i_l}}(1-\cos x)\bigg)^{-1}dx \nonumber \\
&&+ \bar{R}_n(b_n^{-5/8}k_n+1). \nonumber
\end{eqnarray}
\begin{discuss}
{\colorr
Lemma 5.2 in~\cite{ogi18}の微分バージョンが示されていることと，Lemma~\ref{invS-eq-lemma}の微分版も絶対収束することを使う. 
微分バージョンは$\sigma_{1,n}$に対して示す必要があるのか？→ないかもしれないが同様に示せるからいい

Ogihara 2017のLemma 5.1のresidual$\times b_n^{-1}$は$\bar{R}_n(b_n^{1/2}\ell_n^{-1})\times \bar{R}_n(k_n^{-1}b_n^{\eta'})=\underbar{R}_n(1)$で, 
残りはLemma 5.1の$\sum_p|\mathcal{T}_{m,1}^{n,p}(\mathcal{T}_{m,2}^{n,p})^{-1}-1|$の評価と$\sup_p|\mathcal{T}^{n,p}_{m,3}|=\bar{R}_n(b_n^{-\eta/2-1})$を使えば
$\bar{R}_n(b_n^{1/2}\ell_n^{-1})\times \bar{R}_n(b_n^{\eta-\eta'}+b_n^{-1+\eta+\eta'}+b_n^{-\eta/2})$より$\eta=1/4$, $\eta'=3/8$ととれば$\bar{R}_n(b_n^{-5/8}k_n)$となる.
Lemma 5.2はresidualがone by oneに$\tilde{D}_{j,m}^{-1}\to \dot{D}_{j,m}^{-1}$としていけば$(r_n/\underbar{r}_n)^p\to r_n/\underbar{r}_n$とでき，
residualにかかる因子は$\bar{R}_n(b_n^{\delta}k_n^{-1})$, $p/p-1$の因子が$\bar{R}_n(b_n^{\eta-1/2})$, $p\cdots p$の変換の因子が$\bar{R}_n(b_n^{-1/2}\sqrt{k_n})$,
$\#\{\alpha; \quad \}\#\{\alpha;\quad \}$への変換の因子が$\bar{R}_n(b_n^{-\eta/2})$, ${\rm tr}()\to I_p$が$\bar{R}_n(b_n^{1/2}k_n^{-1})$.
よって，任意の$\epsilon>0$に対して，$\delta <1/2+\epsilon$かつ$1/4<\eta<3/8$となるようにすれば残差項は
\begin{equation*}
\bar{R}_n(b_n^{1/2}\ell_n^{-1}(b_n^{1/2+\epsilon}k_n^{-1}+b_n^{-1/8}+b_n^{-1/2}k_n^{1/2}+b_n^{-1/8}))=\bar{R}_n(b_n^{-5/8}k_n+b_n^{\epsilon})
\end{equation*}
なのでOK. Ogihara \cite{ogi18} Remark 5.2より$p$の和としてもちゃんと評価できている.

また，Ogihara \cite{ogi18}のLemma A.9を$A_n=c_1,B_n=c'_2,C_n=b_n^{-2}a_{s_{m-1}}^2(a_{s_{m-1}}^1)^{-3}v_{1,\ast}^{-1}v_{2,\ast}^{-1}\det(\tilde{b}_m\tilde{b}_m^{\top})$として適用すれば
(A.20)の左辺と右辺のレート比が$\ell_n^{-1}$で（右辺は$\ell_n^{-1}$か？$b_n^{-1/2}\ell_n^{-1}$のような気もするが？→どうせ使わないので考えないでいい．）
\begin{equation*}
\frac{Ta_{s_{m-1}}k_n}{\pi}\sum_{p=0}^{\infty}\frac{(a_{s_{m-1}}^2)^pI_{p+1,p}(c_1,c'_2)}{b_n^{2p+1}(a_{s_{m-1}}^1)^{2p+1}v_{1,\ast}^{p+1}v_{2,\ast}^p}
-Tb_n^{1/2}\ell_n^{-1}\frac{\sqrt{\dot{a}_m^1\dot{a}_m^2}|\tilde{b}_m^2|^2+\dot{a}_m^1\mathfrak{C}(\tilde{b}_m)}{2\mathfrak{C}(\tilde{b}_m)\mathcal{U}_{s_{m-1}}(\tilde{b}_m)}=\bar{R}_n(\ell_n^{-1}).
\end{equation*}
よって$\Phi_{1,m}=\bar{R}_n(b_n^{-5/8}k_n+1)$. 上からの評価関数は$m$に依存せずとれるので$\sup_m|\Phi_{1,m}|=\bar{R}_n(b_n^{-5/8}k_n+1)$.
同様にして$\sup_m|\Phi_{j,m}|=\bar{R}_n(b_n^{-5/8}k_n+1)$ for $2\leq j\leq 3$.
}
\end{discuss}

Let $\mathfrak{F}_m(x)={\rm diag}(([\tilde{\Sigma}_m]_{ll}b_n^{-1}+2v_{l,\ast}(\check{a}_m^1)^2/\check{a}_m^l(1-\cos x))_l)$
and $\tilde{\Sigma}'_m={\rm diag}(([\tilde{\Sigma}_m]_{ll})_l)$.
By an argument similar to the above, the right-hand side of (\ref{tr-invD-est2}) is equal to
\begin{equation*}
\frac{T\check{a}_m^1\ell_n^{-1}}{\pi}\int^\pi_0[(\mathfrak{F}_m(x)+b_n^{-1}(\tilde{\Sigma}_m-\tilde{\Sigma}'_m))^{-1}]_{ij}dx+o_p(b_n^{1/2}\ell_n^{-1}).
\end{equation*}
\begin{discuss}
{\colorr 上の右辺は
\begin{eqnarray}
&&\sum_{p=0}^\infty(-1)^p\sum_{(i_0,\cdots,i_p)\in\mathcal{I}_p}\prod_{l=1}^p[\tilde{\Sigma}_m]_{i_{l-1},i_l}
\frac{T\check{a}_m^1k_n}{\pi b_n^{p+1}}\int^\pi_0\sum_{l=0}^p[\mathfrak{F}_m(x)]_{ll}^{-1}dx+o_p(b_n^{1/2}\ell_n^{-1}) \nonumber \\
&&\quad =\sum_{p=0}^\infty (-1)^p \frac{T\check{a}_m^1k_n}{\pi b_n^{p+1}}\int^\pi_0 \left[\mathfrak{F}_m^{-1}(x)(\tilde{\Sigma}_m-\tilde{\Sigma}'_m))^p\mathcal{F}(x)^{-1}\right]_{ij}dx+o_p(b_n^{1/2}\ell_n^{-1}). \nonumber
\end{eqnarray}
}
\end{discuss}
Hence, by setting $\mathfrak{A}_m={\rm diag}((4v_{j,\ast}(\check{a}_m^1)^2/\check{a}_m^j)_j)+b_n^{-1}\tilde{\Sigma}_m$,
Lemma \ref{inv-residue-lemma} yields
\begin{eqnarray}
&&{\rm tr}(\tilde{S}_m^{-1}D'_{i,j,m}) \nonumber \\
&&\quad =\frac{T\check{a}_m^1\ell_n^{-1}}{\pi}\int^\infty_{-\infty}[(\mathfrak{A}_mt^2+b_n^{-1}\tilde{\Sigma}_m)^{-1}]_{ij}dt+o_p(b_n^{1/2}\ell_n^{-1}) \nonumber \\
&&\quad =Tb_n^{1/2}\ell_n^{-1}\check{a}_m^1[\mathfrak{A}_m^{-1/2}(\mathfrak{A}_m^{-1/2}\tilde{\Sigma}_m\mathfrak{A}_m^{-1/2})^{-1/2}\mathfrak{A}_m^{-1/2}]_{ij}+o_p(b_n^{1/2}\ell_n^{-1}) \nonumber \\
&&\quad =\frac{b_n^{1/2}}{2}\int^{s_m}_{s_{m-1}}\sqrt{\tilde{a}_t^i\tilde{a}_t^j}[\mathcal{D}(t,\Sigma(\sigma_{1,n}))^{-1/2}]_{ij}dt+o_p(b_n^{1/2}\ell_n^{-1}). \nonumber
\end{eqnarray}
Additionally, we can extend this to $\max_{i,j,m}|\partial_\sigma^l\Phi_{i,j,m}(\sigma_{1,n})|=\bar{R}_n(b_n^{-5/8}k_n+1)$ for $0\leq l\leq 1$ in cases where [B1$'$] and [B2] are satisfied.
\begin{discuss}
{\colorr $\lVert \mathfrak{F}_m^{-1/2}(\mathfrak{D}_m-\mathfrak{D}'_m)\mathfrak{F}_m^{-1/2}\rVert\leq b_n$.
$x\to t$の変換の正当性は要素ごとに$x$の関数になっていてそれを$t$に変えているだけだからOK.

$\mathcal{A}_m^{-1/2}=(1/2){\rm diag}((\sqrt{\check{a}_m^l}/\check{a}_m^1)_l)+O_p(b_n^{-1})$は微分を計算すればわかるか．
$\Sigma(t,X_t,\sigma_{1,n})-\Sigma(s_{m-1},X_{s_{m-1}},\sigma_{1,n})$の評価は$O_p(\ell_n^{-1})$をいえればいいからりぷしっつ条件からOK.
}
\end{discuss}

\end{proof}

\begin{lemma}\label{Psi-est-lemma}
Assume [A1$'$] and [A2]. Let $\sigma_{1,n}$ and $\sigma_{2,n}$ be $\bar{\Lambda}$-valued random variables. 
Then, 
\begin{equation*}
b_n^{-1/2}\ell_n\max_m|\Psi_m(\sigma_{1,n},\sigma_{2,n})|\overset{P}\to 0.
\end{equation*}
If [B1$'$] and [B2] are also satisfied, then $\max_m|\partial_{\sigma_1}^{l_1}\partial_{\sigma_2}^{l_2}\Psi_m(\sigma_{1,n},\sigma_{2,n})|=\bar{R}_n(b_n^{-5/8}k_n+1)$ 
for $0\leq l_1,l_2\leq 1$.
\end{lemma}

\begin{proof}
Letting $\tilde{D}'_m={\rm diag}((\tilde{D}_{i,m})_{i=1}^\gamma)$ Lemma~\ref{invS-eq-lemma} yields
\begin{eqnarray}
\log\frac{\det \tilde{S}_m}{\det \tilde{D}'_m}&=&\sum_{p=1}^\infty\frac{(-1)^{p-1}}{p}\sum_{\substack{(i_0,\cdots,i_p)\in\mathcal{I}_p \\ i_0=i_p}}
\prod_{l=1}^p[\tilde{\Sigma}_m]_{i_{l-1},i_l}{\rm tr}\bigg(\prod_{l=1}^p\tilde{D}_{i_l,m}^{-1}\mathfrak{G}_{i_{l-1},i_l}\bigg) \nonumber \\
&=&\sum_{p=1}^\infty\frac{(-1)^{p-1}}{p}\sum_{\substack{(i_0,\cdots,i_p)\in\mathcal{I}_p \\ i_0=i_p}}
\prod_{l=1}^p[\tilde{\Sigma}_m]_{i_{l-1},i_l}\frac{T\check{a}_m^1k_n}{\pi b_n^p} \nonumber \\
&&\times \int^\pi_0\prod_{l=0}^{p-1}[\mathcal{F}_m(x)]^{-1}_{i_l,i_l}dx+o_p(b_n^{1/2}\ell_n^{-1}) \nonumber \\
&=&\frac{T\check{a}_m^1k_n}{\pi}\int^\pi_0\log\det(\mathcal{E}_\gamma+b_n^{-1}\mathcal{F}_m^{-1}(x)(\tilde{\Sigma}_m-\tilde{\Sigma}'_m))dx+o_p(b_n^{1/2}\ell_n^{-1}). \nonumber
\end{eqnarray}

We write $\mathfrak{A}'_m=\mathfrak{A}_m-b_n^{-1}\tilde{\Sigma}_m+b_n^{-1}\tilde{\Sigma}'_m$.
Lemma A.3 of~\cite{ogi18} and Lemmas~\ref{log-residue-lemma} and~\ref{rootA-diff-lemma} together yield
\begin{eqnarray}
&&\int^\pi_0\log\det (\mathcal{E}_\gamma+b_n^{-1}\mathfrak{F}_m^{-1}(\tilde{\Sigma}_m-\tilde{\Sigma}'_m))dx \nonumber \\
&&\quad =\int^\pi_0\log\frac{\det (\mathfrak{F}_m(x)+b_n^{-1}(\tilde{\Sigma}_m-\tilde{\Sigma}'_m))}{\det \mathfrak{F}_m(x)}dx \nonumber \\
&&\quad =\int^\infty_{-\infty}\frac{1}{1+t^2}\Big\{\log\det (\mathfrak{A}_mt^2+b_n^{-1}\tilde{\Sigma}_m)-\log\det (\mathfrak{A}'_mt^2+b_n^{-1}\tilde{\Sigma}'_m)\Big\}dt \nonumber \\
&&\quad =\pi\bigg\{\log \frac{\det \mathfrak{A}_m}{\det \mathfrak{A}'_m} +2\log\frac{\det(\mathcal{E}_\gamma+b_n^{-1/2}\sqrt{\mathfrak{A}_m^{-1/2}\tilde{\Sigma}_m\mathfrak{A}_m^{-1/2}}) }{\det (\mathcal{E}_\gamma+b_n^{-1/2}\sqrt{\mathfrak{A}'^{-1/2}_m\tilde{\Sigma}'_m\mathfrak{A}'^{-1/2}_m})}\bigg\} \nonumber \\
&&\quad =2\pi b_n^{-1/2}\Big\{{\rm tr}(\sqrt{\mathfrak{A}_m^{-1/2}\tilde{\Sigma}_m\mathfrak{A}_m^{-1/2}})-{\rm tr}(\sqrt{\mathfrak{A}'^{-1/2}_m\tilde{\Sigma}'_m\mathfrak{A}'^{-1/2}_m})\Big\}+o_p(b_n^{-1/2}) \nonumber \\
&&\quad =2\pi b_n^{-1/2}\bigg\{\frac{1}{2\check{a}_m^1}{\rm tr}(\mathcal{D}(s_{m-1},\Sigma(\sigma_{1,n}))^{1/2})-\sum_j\frac{(\check{a}_m^j/v_{j,\ast})^{1/2}}{2\check{a}_m^1}\sqrt{[\tilde{\Sigma}_m]_{jj}}\bigg\} \nonumber \\
&&\quad \quad +o_p(b_n^{-1/2}). \nonumber
\end{eqnarray}
\begin{discuss}
{\colorr 最後の等式はLemma~\ref{rootA-diff-lemma}を二回使う．
$\lVert (A^{1/2}+B^{1/2})^{-1}\rVert\leq \lVert A^{-1/2}\rVert =\lVert A\rVert^{-1/2}$の評価は難しくないか．}
\end{discuss}

By an argument similar to the above, (5.20) in~\cite{ogi18}, and the arguments after that yield
\begin{equation*}
\log\frac{\det \tilde{D}_{l,m}(\sigma_{1,n})}{\det \tilde{D}_{l,m}(\sigma_{2,n})}=b_n^{1/2}T\ell_n^{-1}\sqrt{\frac{\check{a}_m^l}{v_{l,\ast}}}[\tilde{\Sigma}_m]_{ll}^{1/2}(\sigma_{1,n})-[\tilde{\Sigma}_m]_{ll}^{1/2}(\sigma_{2,n}))+o_p(b_n^{1/2}\ell_n^{-1})
\end{equation*}
for any $1\leq l\leq \gamma$. Therefore, we obtain
\begin{eqnarray}
&&\log\frac{\det \tilde{S}_m(\sigma_{1,n})}{\det \tilde{S}_m(\sigma_{2,n})} \nonumber \\
&&\quad =b_n^{1/2}\int^{s_m}_{s_{m-1}}\Big\{{\rm tr}(\mathcal{D}(t,\Sigma(\sigma_{1,n}))^{1/2}-{\rm tr}(\mathcal{D}(t,\Sigma(\sigma_{2,n}))^{1/2}\Big\}dt+o_p(b_n^{1/2}\ell_n^{-1}). \nonumber
\end{eqnarray}

We then have $\max_m|\partial_{\sigma_1}^{l_1}\partial_{\sigma_2}^{l_2}\Psi_m(\sigma_{1,n},\sigma_{2,n})|=\bar{R}_n(b_n^{-5/8}k_n+1)$ for $0\leq l_1,l_2\leq 1$ 
if further [B1$'$] and [B2] are satisfied.
\end{proof}

\subsection{Proof of Theorem \ref{consistency-theorem}}\label{consistency-proof-section}

\begin{discuss}
{\colorr $\tilde{\Sigma}_{m,\dagger}$と$\tilde{S}_{m,\dagger}$ Prop~\ref{Hn-diff-prop}でdefされている．
${\bf B}_t=(\sqrt{\tilde{a}_t^i\tilde{a}_t^j}[\Sigma(t,X_t,\sigma)]_{ij})_{1\leq i,j\leq d'}$に対して
\begin{equation*}
\mathcal{Y}(\sigma)=\int^T_0\bigg\{-\frac{1}{4}{\rm tr}(({\bf B}_{t,\ast}-{\bf B}_t){\bf B}_t^{-1/2})+\frac{1}{2}{\rm tr}({\bf B}_{t,\ast}^{1/2})-\frac{1}{2}{\rm tr}({\bf B}_t^{1/2})\bigg\}dt
=-\frac{1}{4}\int^T_0{\rm tr}(({\bf B}_{t,\ast}^{1/2}-{\bf B}_t^{1/2})^2{\bf B}_t^{-1/2}).
\end{equation*}
($\tilde{A}^{-1/2}$の中から係数$1/2$が出てくる)
一次元の$\mathcal{Y}(\sigma)$の形と一致することを確認．二次元でも数値的に合っていることを確認．
}
\end{discuss}

Using localization techniques similar to those used in Lemma 4.1 of Gobet~\cite{gob01}, we may assume [A1$'$] instead of [A1] without weakening the outcome.

Then, Lemma \ref{HtoTildeH-lemma} yields
\begin{equation}\label{consistency-eq1}
b_n^{-1/2}(H_n(\sigma_{1,n},\hat{v}_n)-H_n(\sigma_{2,n},\hat{v}_n))-b_n^{-1/2}(\tilde{H}_n(\sigma_{1,n})-\tilde{H}_n(\sigma_{2,n}))\overset{P}\to  0
\end{equation}
for $\bar{\Lambda}$-valued random variables $\sigma_{1,n}$ and $\sigma_{2,n}$.

Moreover, Sobolev's inequality and Remark~\ref{ZSZ-est-rem2} yield
\begin{eqnarray}\label{consistency-eq2}
&& \\
&&E_\Pi\bigg[\sup_{\sigma_1,\sigma_2\in \bar{\Lambda}}\bigg|
\sum_m{\rm tr}((\tilde{S}_m^{-1}(\sigma_1)-\tilde{S}_m^{-1}(\sigma_2))(\tilde{Z}_m\tilde{Z}_m^\top-\tilde{S}_{m,\dagger}))\bigg|^q\bigg] \nonumber \\
&&\quad \leq C_q\sup_{\sigma_1,\sigma_2}\sum_{j_1,j_2=0}^1E_\Pi\bigg[\bigg|\sum_m\partial_{\sigma_1}^{j_1}\partial_{\sigma_2}^{j_2}{\rm tr}((\tilde{S}_m^{-1}(\sigma_1)-\tilde{S}_m^{-1}(\sigma_2))\bar{E}_m[\tilde{Z}_m\tilde{Z}_m^\top])\bigg|^q\bigg] \nonumber \\
&&\quad \leq C_q\sup_{\sigma_1,\sigma_2}\sum_{j_1,j_2=0}^1E_\Pi\bigg[\bigg(\sum_m{\rm tr}((\partial_{\sigma_1}^{j_1}\partial_{\sigma_2}^{j_2}(\tilde{S}_m^{-1}(\sigma_1)-\tilde{S}_m^{-1}(\sigma_2))\tilde{S}_{m,\dagger})^2)\bigg)^{\frac{q}{2}}\bigg] \nonumber \\
&&\quad \quad + \bar{R}_n(\ell_n^{q/2}+b_n^{-3q/4}k_n^{7q/4}) \nonumber
\end{eqnarray}
for sufficiently large $q$. 

Lemmas A.1 and 5.2 of~\cite{ogi18} together yield
\begin{eqnarray}\label{tr2-est}
&&{\rm tr}(((\tilde{S}_m^{-1}(\sigma_1)-\tilde{S}_m^{-1}(\sigma_2))\tilde{S}_{m,\dagger})^2) \\
&&\quad ={\rm tr}(\{\tilde{S}_m^{-1}(\sigma_1)(\tilde{S}_{m,\dagger}-\tilde{S}_m(\sigma_1))-\tilde{S}_m^{-1}(\sigma_2)(\tilde{S}_{m,\dagger}-\tilde{S}_m(\sigma_2))\}^2) \nonumber \\
&&\quad \leq 2\sum_{j=1}^2{\rm tr}(\{\tilde{S}_m^{-1}(\sigma_j)(\tilde{S}_{m,\dagger}-\tilde{S}_m(\sigma_j))\}^2) \nonumber \\
&&\quad \leq 2\sum_{j=1}^2\lVert D'^{1/2}_m\tilde{S}_m(\sigma_j)D'^{1/2}_m\rVert \nonumber \\
&&\quad \quad \times {\rm tr}(D'^{-1/2}_m(\tilde{S}_{m,\dagger}-\tilde{S}_m(\sigma_j))\tilde{S}_m^{-1}(\sigma_j)(\tilde{S}_{m,\dagger}-\tilde{S}_m(\sigma_j))D'^{-1/2}_m) \nonumber \\
&&\quad \leq C\sum_{j=1}^2\lVert D'^{-1/2}_m(\tilde{S}_{m,\dagger}-\tilde{S}_m(\sigma_j))D'^{-1/2}_m\rVert^2 {\rm tr}(\tilde{S}_m^{-1}(\sigma_j)D'_m) \nonumber \\
&&\quad \leq C\sum_{j=1}^2{\rm tr}(\tilde{S}_m^{-1}(\sigma_j)D'_m). \nonumber
\end{eqnarray}

Similarly, we obtain
\begin{equation}\label{tr2-est2}
{\rm tr}((\partial_{\sigma_1}^{j_1}\partial_{\sigma_2}^{j_2}(\tilde{S}_m^{-1}(\sigma_1)-\tilde{S}_m^{-1}(\sigma_2))\tilde{S}_{m,\dagger})^2)\leq C\sum_{j=1}^2{\rm tr}(\tilde{S}_m^{-1}(\sigma_j)D'_m)
\end{equation}
for $0\leq j_1,j_2\leq 1$. 
\begin{discuss}
{\colorr
\begin{eqnarray}
&&{\rm tr}(\partial_\sigma \tilde{S}_m^{-1}\tilde{S}_{m,\dagger})^2)
={\rm tr}(\tilde{S}_\dagger^{1/2}\tilde{S}^{-1}\partial_\sigma \tilde{S}\tilde{S}^{-1}\tilde{S}_\dagger\tilde{S}^{-1}\partial_\sigma \tilde{S}\tilde{S}^{-1}\tilde{S}_\dagger^{1/2}) \nonumber \\
&&\quad \leq \lVert \tilde{S}^{-1/2}\tilde{S}_\dagger\tilde{S}^{-1/2}\rVert {\rm tr}(\tilde{S}_\dagger^{1/2}\tilde{S}^{-1}\partial_\sigma \tilde{S}\tilde{S}^{-1}\partial_\sigma \tilde{S}\tilde{S}^{-1}\tilde{S}_\dagger^{1/2}) \nonumber \\
&&\quad \leq C\lVert D'^{-1/2}\partial_\sigma \tilde{S}D'^{-1/2}\rVert \lVert D'^{1/2}\tilde{S}^{-1}D'^{1/2}\rVert
{\rm tr}(\tilde{S}_\dagger\tilde{S}^{-1}D'\tilde{S}^{-1}) \nonumber \\
&&\quad \leq C\lVert \tilde{S}^{-1/2}\tilde{S}_\dagger\tilde{S}^{-1/2}\rVert
{\rm tr}(\tilde{S}^{-1}D'). \nonumber
\end{eqnarray}
}
\end{discuss}

Moreover, Lemma 2 of~\cite{ogi-yos14}, Lemma A.4 of~\cite{ogi18} and the inequality
\begin{equation*}
D'^{-1/2}_{i,m}\tilde{D}_{i,m}D'^{-1/2}_{i,m}\geq [\tilde{\Sigma}_m]_{ii}\mathcal{E}_{\mathcal{K}^\gamma_m}
\end{equation*}
together give
\begin{equation*}
\lVert D'^{-1/2}_{i,m}\mathfrak{S}_{i,j}D'^{-1/2}_{j,m}\rVert \leq 1 \quad {\rm and} \quad
\lVert D'^{1/2}_{i,m}\tilde{D}_{i,m}^{-1}D'^{1/2}_{i,m}\rVert \leq [\tilde{\Sigma}_m]_{ii}^{-1},
\end{equation*}
Taking this with Lemma~\ref{invS-eq-lemma} and Lemma 5.2 of~\cite{ogi18} and setting $p=0$, we obtain
\begin{eqnarray}\label{tr-invSD-est}
&& \\
&&{\rm tr}(\tilde{S}_m^{-1}D'_m) \nonumber \\
&&\quad \leq \sum_{p=0}^\infty \sum_{\substack{(i_0,\cdots,i_p)\in\mathcal{I}_p \\ i_p=i_0}}\bigg|\prod_{l=1}^p[\tilde{\Sigma}_m]_{i_{l-1},i_l}\bigg|
{\rm tr}\bigg(\bigg(\prod_{l=1}^p(\tilde{D}_{i_{l-1},m}^{-1}\mathfrak{S}_{i_{l-1},i_l})\bigg)\tilde{D}_{i_p,m}^{-1}D'_{i_0,m}\bigg) \nonumber \\
&&\quad \leq \sum_{p=0}^\infty \sum_{(i_0,\cdots,i_p)\in\mathcal{I}_p}\bigg|\prod_{l=1}^p
\frac{[\tilde{\Sigma}_m]_{i_{l-1},i_l}}{[\tilde{\Sigma}_m]_{i_{l-1},i_{l-1}}^{1/2}[\tilde{\Sigma}_m]_{i_l,i_l}^{1/2}}\bigg|
{\rm tr}(\tilde{D}_{i_0,m}^{-1}D'_{i_0,m}) \nonumber \\
&&\quad =O_p(b_n^{1/2}\ell_n^{-1}). \nonumber
\end{eqnarray}

Then, (\ref{consistency-eq1})--(\ref{tr-invSD-est}) yield
\begin{eqnarray}
&&H_n(\sigma_{1,n},\hat{v}_n)-H_n(\sigma_{2,n},\hat{v}_n) \nonumber \\
&&\quad =-\frac{1}{2}\sum_m\bigg({\rm tr}((\tilde{S}_m^{-1}(\sigma_{1,n})-\tilde{S}_m^{-1}(\sigma_{2,n}))\tilde{S}_{m,\dagger})
+\log\frac{\det \tilde{S}_m(\sigma_{1,n})}{\det \tilde{S}_m(\sigma_{2,n})}\bigg) \nonumber \\
&&\quad \quad +o_p(b_n^{1/2}). \nonumber
\end{eqnarray}
Moreover, [A1$'$] and [A2] imply
\begin{eqnarray}
&&-\frac{1}{4}\sum_m\sum_{i,j}\int^{s_m}_{s_{m-1}}\sqrt{\tilde{a}_t^i\tilde{a}_t^j}[\mathcal{D}(t,\Sigma(\sigma_{1,n}))^{-1/2}]_{ij}[\Sigma_{m,\dagger}-\Sigma_m(\sigma_{1,n})]_{ij}dt \nonumber \\
&& \quad =-\frac{1}{4}\int^T_0{\rm tr}(\mathcal{D}(t,\Sigma(\sigma_{1,n}))^{-1/2}(\mathcal{D}(t,\Sigma_\dagger)-\mathcal{D}(t,\Sigma(\sigma_{1,n})))dt+o_p(1). \nonumber
\end{eqnarray}
Together with Lemmas \ref{Phi-est-lemma} and \ref{Psi-est-lemma} and the observation that 
\begin{equation*}
{\rm tr}(\tilde{S}_m^{-1}(\sigma_{j,n})\tilde{S}_{m,\dagger})={\rm tr}(\tilde{S}_m^{-1}(\sigma_{j,n})(\tilde{S}_{m,\dagger}-\tilde{S}_m(\sigma_{j,n})))+\sum_{i=1}^\gamma k^i_m
\end{equation*}
for $j\in \{1,2\}$, we have
\begin{equation}\label{Hn-limit-eq}
b_n^{-1/2}(H_n(\sigma_{1,n},\hat{v}_n)-H_n(\sigma_{2,n},\hat{v}_n))=D(\Sigma(\sigma_{2,n}),\Sigma_\dagger)-D(\Sigma(\sigma_{1,n}),\Sigma_\dagger)+o_p(1).
\end{equation}

Since $\sigma\mapsto D(\Sigma(\sigma),\Sigma_{\dagger})$ is continuous,  there exists a $\bar{\Lambda}$-valued random variable $\sigma_{\ast}$ such that
\begin{equation}\label{sigma-ast-eq}
D(\Sigma(\sigma_{\ast}),\Sigma_{\dagger})=\min_{\sigma} D(\Sigma(\sigma),\Sigma_{\dagger}).
\end{equation}

The definition of $\hat{\sigma}_n$ implies $b_n^{-1/2}(H_n(\hat{\sigma}_n,\hat{v}_n)-H_n(\sigma_{\ast},\hat{v}_n))\geq 0$.
Together with the inequality $D(\Sigma(\hat{\sigma}_n),\Sigma_{\dagger})-D(\Sigma(\sigma_\ast),\Sigma_{\dagger})\geq \delta$,
we have
\begin{equation*}
|b_n^{-1/2}(H_n(\hat{\sigma}_n,\hat{v}_n)-H_n(\sigma_{\ast},\hat{v}_n))+D(\Sigma(\hat{\sigma}_n),\Sigma_{\dagger})-D(\Sigma(\sigma_\ast),\Sigma_{\dagger})|\geq \delta.
\end{equation*}
Therefore, for any $\epsilon,\delta>0$ there exists some $N\in\mathbb{N}$ such that
\begin{equation*}
P[|\mathcal{D}(\Sigma(\hat{\sigma}_n),\Sigma_\dagger)-D(\Sigma(\sigma_\ast),\Sigma_\dagger)|\geq \delta]<\epsilon.
\end{equation*}
for all $n\geq N$, by (\ref{Hn-limit-eq}).

\qed

\subsection{Proof of optimal rate of convergence}\label{opt-conv-proof-section}

In this subsection, we prove the results of Section~\ref{opt-conv-section}.
We first prove Proposition~\ref{Hn-diff-prop}.

\noindent {\bf Proof of Proposition \ref{Hn-diff-prop}.}

By localization techniques, we can assume [B1$'$] instead of [B1].

By the definitions of $\tilde{H}_n$ and $\Delta_n$, we have
\begin{eqnarray}
&&\tilde{H}_n(\sigma_{1,n})-\tilde{H}_n(\sigma_{2,n}) \nonumber \\
&&\quad=-\frac{1}{2}\sum_m\bigg(\tilde{Z}_m^{\top}(\tilde{S}_m^{-1}(\sigma_{1,n})-\tilde{S}_m^{-1}(\sigma_{2,n}))\tilde{Z}_m+\log\frac{\det \tilde{S}_m(\sigma_{1,n})}{\det \tilde{S}_m(\sigma_{2,n})}\bigg)  \nonumber \\
&&\quad = \Delta_n(\sigma_{1,n},\sigma_{2,n}) \nonumber \\
&&\quad \quad -\frac{1}{2}\sum_m\bigg({\rm tr}(\tilde{S}_{m,\dagger}(\tilde{S}_m^{-1}(\sigma_{1,n})-\tilde{S}_m^{-1}(\sigma_{2,n})))+\log\frac{\det \tilde{S}_m(\sigma_{1,n})}{\det \tilde{S}_m(\sigma_{2,n})}\bigg). \nonumber
\end{eqnarray}

Then, Lemma \ref{HtoTildeH-lemma}, the H\"older continuity of $a_t$ and $[B1']$ yield the desired result.
\begin{discuss}
{\colorr 
\begin{eqnarray}
&&b_n^{-1/4}(H_n(\sigma_{1,n},\hat{v}_n)-H_n(\sigma_{2,n},\hat{v}_n))-b_n^{-1/4}\Delta_n (\sigma_{1,n},\sigma_{2,n}) \nonumber \\
&&+b_n^{1/4}(D(\Sigma(\sigma_{1,n}),\Sigma_\dagger)-D(\Sigma(\sigma_{2,n}),\Sigma_\dagger)) \nonumber \\
&=&-\frac{1}{2}b_n^{-1/4}\bigg\{\sum_m{\rm tr}((\tilde{S}_m^{-1}(\sigma_{1,n})-\tilde{S}_m^{-1}(\sigma_{2,n}))\tilde{S}_{m,\dagger})+\log\frac{\det \tilde{S}_m(\sigma_{1,n})}{\det \tilde{S}_m(\sigma_{2,n})}\bigg\} \nonumber \\
&&+b_n^{1/4}(D(\Sigma(\sigma_{1,n}),\Sigma_\dagger)-D(\Sigma(\sigma_{2,n}),\Sigma_\dagger))+o_p(1) \nonumber
\end{eqnarray}
で
\begin{eqnarray}
&&E_m(\tilde{a}_m,\tilde{\Sigma}_m(\sigma_{j,n}),\tilde{\Sigma}_{m,\dagger},v_\ast) \nonumber \\
&&\quad ={\rm tr}(\tilde{S}_m^{-1}(\sigma_{j,n})\tilde{S}_{m,\dagger}) \nonumber \\
&&\quad \quad +\frac{b_n^{1/2}}{2}T\ell_n^{-1}{\rm tr}(\mathcal{D}(s_{m-1},\Sigma(\sigma_{j,n}))-\mathcal{D}(s_{m-1},\Sigma_\dagger))\mathcal{D}(s_{m-1},\Sigma(\sigma_{j,n}))^{1/2}), \nonumber
\end{eqnarray}
\begin{eqnarray}
F_m(\tilde{a}_m,\tilde{\Sigma}_m(\sigma_{j,m}),v_\ast)
=\log\det\tilde{S}_m(\sigma_{1,n})-b_n^{1/2}T\ell_n^{-1}{\rm tr}(\mathcal{D}(s_{m-1},\Sigma(\sigma_{j,n}))^{1/2}) \nonumber
\end{eqnarray}
よりOK.}
\end{discuss}
\begin{discuss}
{\colorr $E[b_{t,\dagger}-b_{s,\dagger}|\mathcal{F}_s]=O_p(t-s), E[|b_{t,\dagger}-b_{s,\dagger}|^2|\mathcal{F}_s]=O_p(t-s)$が必要．}
\end{discuss}

\qed

In the following, we check that $\check{H}_n-H_n$ cancels the bias term of $H_n$.
First, we show hat $B_{m,n}$ is a good approximation of $\tilde{\Sigma}_{m,\dagger}$.

Let $\tilde{B}_{m,n}$ be a $\gamma\times \gamma$ matrix defined by
\begin{equation*}
[\tilde{B}_{m,n}]_{ij}=\frac{\ell_n}{T\Psi_1}\bigg\{\bigg(\sum_{l=1}^{k^i_m}g_l^iZ_{m,l}^i\bigg)\bigg(\sum_{l=1}^{k^j_m}g_l^jZ_{m,l}^j\bigg)-\frac{v_{i,\ast}}{k^i_m}\Psi_21_{\{i=j\}}\bigg\}.
\end{equation*}

\begin{lemma}\label{Bmn-error-est-lemma}
Assume [B1$'$], [B2], and [V]. Then, $\tilde{B}_{m,n}-B_{m,n}=O_p(b_n^{1/2}k_n^{-2})$,
$\bar{E}_m[\tilde{B}_{m,n}-\tilde{\Sigma}_{m,\dagger}]=\bar{R}_n(1)$, and $E_m[\tilde{B}_{m,n}-\tilde{\Sigma}_{m,\dagger}]=\bar{R}_n(k_n^{-1/2}+\ell_n^{-1})$ all hold.
\end{lemma}
\begin{discuss}
{\colorr $\det B_{m,n}>0$かどうかはわからないが，$\det$は$C$にしか入れないから問題ない．}
\end{discuss}

\begin{proof}
It is straightforward to see that $\tilde{B}_{m,n}-B_{m,n}=O_p(b_n^{1/2}k_n^{-2})$.

We first obtain
\begin{eqnarray}
\sum_{l_1=1}^{k^i_m}\sum_{l_2=1}^{k^j_m}g_{l_1}^ig_{l_2}^jZ_{1,m,l_1}^iZ_{2,m,l_2}^j&=&\sum_{l_2=0}^{k^j_m}\sum_{l_1=1}^{k^i_m}g_{l_1}^i(g_{l_2}^j-g_{l_2+1}^j)Z_{1,m,l_1}^i\epsilon_{l_2}^j \nonumber \\
&=&\tilde{R}_n(k_nk_n^{-1}b_n^{-1/2})=\tilde{R}_n(b_n^{-1/2}). \nonumber
\end{eqnarray}
Here, $X_{m,n}=\tilde{R}_n(c_n)$ means that $E_m[X_{m,n}]=0$ and $X_{m,n}=\bar{R}_n(c_n)$. 
Moreover, It\^o's formula yields
\begin{eqnarray}
&&\bigg(\sum_{l_1=1}^{k^i_m}g_{l_1}^iZ_{1,m,l_1}^i\bigg)\bigg(\sum_{l_2=1}^{k^j_m}g_{l_2}^jZ_{1,m,l_2}^j\bigg) \nonumber \\
&=&\sum_{l_1=1}^{k^i_m}\sum_{l_2=1}^{k^j_m}g_{l_1}^ig_{l_2}^j\bigg(\int_{I^i_{l_1,m}\cap I^j_{l_2,m}}(\Sigma_{t,\dagger})_{ij} dt
+\int_{I^i_{l_1,m}}\int_{I_{l_2,m}^j\cap [0,t)}b_{s,\dagger}^jdW_sb_{t,\dagger}^idW_t \nonumber \\
&&\quad \quad \quad \quad \quad +\int_{I^j_{l_2,m}}\int_{I_{l_1,m}^i\cap [0,t)}b_{s,\dagger}^idW_sb_{t,\dagger}^jdW_t\bigg) \nonumber \\
&=&[\tilde{\Sigma}_{m,\dagger}]_{ij} \sum_{l_1=1}^{k^i_m}\sum_{l_2=1}^{k^j_m}g_{l_1}^ig_{l_2}^j|I^i_{l_1,m}\cap I^j_{l_2,m}|+\tilde{R}_n(\ell_n^{-3/2})+\bar{R}_n(\ell_n^{-2})+\tilde{R}_n(\ell_n^{-1}). \nonumber
\end{eqnarray}
\begin{discuss}
{\colorr $\int\int \ dW_sdW_t$項の評価は$b_{t,\dagger}\to b_{s_{m-1},\dagger}$としたものが$\tilde{R}_n(\ell_n^{-1})$となるのはOK. $b_{t,\dagger}\to b_{t,\dagger}-b_{s_{m-1},\dagger}$としたものは，
\begin{equation*}
\sum_{l_1}g_{l_1}E\bigg[\bigg(\sum_{l_2}g_{l_2}\int_{I_{l_2,t}}b_{s,\dagger}dW_s(b_{t,\dagger}-b_{s_{m-1},\dagger})\bigg)^2\bigg]
\end{equation*}
において，$b_{t,\dagger}-b_{s_{m-1},\dagger}$のマルチンゲールパートは$I_{l_2}$に対応する部分とそうでない部分に分ければ$\bar{R}_n(k_n^2b_n^{-1}b_n^{-1}\ell_n^{-1})+\bar{R}_n(k_nb_n^{-1}k_n^2b_n^{-2})=\bar{R}_n(\ell_n^{-3})$
なのでそのルートは$\tilde{R}_n(\ell_n^{-3/2})$. $dt$パートは$\bar{R}_n(k_nb_n^{-1}k_n^2b_n^{-1}\ell_n^{-2})=\bar{R}_n(\ell_n^{-2}k_n^3b_n^{-2})$よりそのルートは$\tilde{R}_n(\ell_n^{-1})$.
}
\end{discuss}
Furthermore, we have
\begin{eqnarray}
\bigg(\sum_{l_1=1}^{k^i_m}g_{l_1}^iZ_{2,m,l_1}^i\bigg)\bigg(\sum_{l_2=1}^{k^j_m}g_{l_2}^jZ_{2,m,l_2}^j\bigg)
&=&\sum_{l_1=0}^{k^i_m}(g_{l_1}^i-g_{l_1+1}^i)\epsilon_{l_1}^i\sum_{l_2=0}^{k^j_m}(g_{l_2}^j-g_{l_2+1}^j)\epsilon^j_{l_2}. \nonumber
\end{eqnarray}
If $i\neq j$, then the independence of $\epsilon_{l_1}^i$ and $\epsilon_{l_2}^j$ implies that the right-hand side is equal to $\tilde{R}_n(k_n^{-1})$.
If $i=j$, then that right-hand side is equal to
\begin{eqnarray}
&&\sum_{l=0}^{k^i_m}(g_l^i-g_{l+1}^i)^2(\epsilon_l^i)^2+2\sum_{l_1=0}^{k_m^i}\sum_{l_2<l_1}(g_{l_2}^i-g_{l_2+1}^i)(g_{l_1}^i-g_{l_1+1}^i)\epsilon_{l_1}^i\epsilon_{l_2}^i \nonumber \\
&&\quad =\sum_{l=0}^{k^i_m}(g_l^i-g_{l+1}^i)^2v_{i,\ast}+\tilde{R}_n(k_n^{-1})
=\frac{\hat{v}_{i,n}}{k^i_m}\Psi_2+\tilde{R}_n(k_n^{-1})+\bar{R}_n(k_n^{-1}b_n^{-1/2}). \nonumber
\end{eqnarray}
\begin{discuss}
{\colorr 
\begin{equation*}
\sum_{l=0}^{k_m^i}\bigg(g\bigg(\frac{l}{k_m^i+1}\bigg)-g\bigg(\frac{l+1}{k_m^i+1}\bigg)\bigg)^2
=\frac{1}{(k_m^i)^2}\sum_{l=0}^{k_m^i}g'\bigg(\frac{l}{k_m^i+1}\bigg)^2+\bar{R}_n(k_n^{-2})=\frac{1}{k_m^i}\int^1_0g'(x)^2dx+\bar{R}_n(k_n^{-2}).
\end{equation*}
}
\end{discuss}
Therefore, we obtain
\begin{eqnarray}
&&\bigg(\sum_{l_1=1}^{k^i_m}g_{l_1}^iZ_{m,l_1}^i\bigg)\bigg(\sum_{l_2=1}^{k^j_m}g_{l_2}^jZ_{m,l_2}^j\bigg) \nonumber \\
&&\quad =[\tilde{\Sigma}_{m,\dagger}]_{ij}\sum_{l_1=1}^{k^i_m}\sum_{l_2=1}^{k^j_m}g_{l_1}^ig_{l_2}^j|I^i_{l_1}\cap I^j_{l_2}|
+\frac{\hat{v}_{i,n}}{k^i_m}\Psi_21_{\{i=j\}}+\tilde{R}_n(\ell_n^{-1})+\bar{R}_n(\ell_n^{-2}). \nonumber
\end{eqnarray}
We divide $[s_{m-1},s_m)$ into $[Tk_n]$ equidistant intervals. If $I^i_{l,m}\cap [(j-1)[Tk_n]^{-1},j[Tk_n]^{-1})\neq \emptyset$,
$g^i_l-g(k/[Tk_n])=O_p(k_n^{-1/2})$ since $l/(k^i_m+1)=k/[Tk_n]+\bar{R}_n(k_n^{-1/2})$ by [B2].
Therefore, we have
\begin{discuss}
{\colorr $k_m^i=\check{a}_m^i[Tk_n]+\bar{R}_n(\sqrt{k_n})$, $l=\check{a}_m^ij+\bar{R}_n(\sqrt{j})$より$l/k_m^i=j/[Tk_n]+\bar{R}_n(k_n^{-1/2})$.}
\end{discuss}
\begin{eqnarray}
&&\sum_{l_1=1}^{k^i_m}\sum_{l_2=1}^{k^j_m}g_{l_1}^ig_{l_2}^j|I^i_{l_1,m}\cap I^j_{l_2,m}| \nonumber \\
&&\quad = \sum_{k=1}^{[Tk_n]}\sum_{l_1=1}^{k^i_m}\sum_{l_2=1}^{k^j_m}g_{l_1}^ig_{l_2}^j\bigg|I^i_{l_1,m}\cap I^j_{l_2,m}\cap \bigg[\frac{j-1}{[Tk_n]},\frac{j}{[Tk_n]}\bigg)\bigg| \nonumber \\
&&\quad = \sum_{k=1}^{[Tk_n]}\sum_{l_1=1}^{k^i_m}\sum_{l_2=1}^{k^j_m}g\bigg(\frac{k}{[Tk_n]}\bigg)^2\bigg|I^i_{l_1,m}\cap I^j_{l_2,m}\cap \bigg[\frac{j-1}{[Tk_n]},\frac{j}{[Tk_n]}\bigg)\bigg|+\bar{R}_n(\ell_n^{-1}k_n^{-\frac{1}{2}}) \nonumber \\
&&\quad =\sum_{k=1}^{[Tk_n]}\frac{T\ell_n^{-1}}{[Tk_n]}g\bigg(\frac{k}{[Tk_n]}\bigg)^2+\bar{R}_n(\ell_n^{-1}k_n^{-\frac{1}{2}}) \nonumber \\
&&\quad =T\ell_n^{-1}\Psi_1+\bar{R}_n(b_n^{-1}k_n^{1/2}). \nonumber
\end{eqnarray}
For the last equation, we used the fact $\sum_{j=1}^mf(j/m)/m=\int f(x)dx+O(m^{-1})$ for $f\in C^1((0,1))$ when $f'$ is bounded. 
\begin{discuss}
{\colorr 
\begin{equation*}
\bigg|\sum_{j=1}^mf(j/m)m^{-1}-\sum_{j=1}^m\int_{\frac{j-1}{m}}^{\frac{j}{m}}f(x)dx\bigg|\leq \bigg|\sum_{j=1}^m\int_{\frac{j-1}{m}}^{\frac{j}{m}}\int_x^{\frac{j}{m}}f'(s)dsdx\bigg|\leq \frac{1}{m}\sup_x|f'(x)|.
\end{equation*}
}
\end{discuss}
Hence, we have
\begin{equation*}
\tilde{B}_{m,n}-\tilde{\Sigma}_{m,\dagger}=\tilde{R}_n(1)+\bar{R}_n(\ell_n^{-1})+\bar{R}_n(k_n^{-1/2}).
\end{equation*}
\end{proof}

\begin{discuss}
{\colorr バイアス補正をしていて${\rm tr}(S_\ast S^{-1})$の極限での近似を行っていないので[B2]が
Ogihara~\cite{ogi18}の[A2]の十分条件かどうかを考える必要はない．}
\end{discuss}

\begin{lemma}\label{tildeSigma-est-lemma}
Assume [B1$'$], [B2] and [V]. Then
\begin{equation*}
\sum_{j=1}^2(-1)^{j-1}\sum_m\Big\{G_m(\tilde{a}_{s_{m-1}},\Sigma_m(\sigma_{j,n}),\tilde{B}_{m,n},v_\ast)-G_m(\tilde{a}_{s_{m-1}},\tilde{\Sigma}_m(\sigma_{j,n}),\tilde{B}_{m,n},v_\ast)\Big\}
\end{equation*}
is $o_p(b_n^{1/4})$.
\end{lemma}
\begin{discuss}
{\colorr ここは$\sigma_{2,n}$との差をとらなくても示せるが${\bf C}_m^{(1)}$を使いまわすために差をとる．}
\end{discuss}
\begin{proof}
\begin{discuss}
{\colorr $B$の部分をいじるので４つの項の処理が必要．}
\end{discuss}
We first show that
\begin{equation}\label{tildeSigma-est-eq1}
\sum_{j=1}^2(-1)^{j-1}\sum_m{\rm tr}(({\bf S}_m^{-1}(\Sigma_m(\sigma_{j,n}),v_\ast)-{\bf S}_m^{-1}(\tilde{\Sigma}_m(\sigma_{j,n}),v_\ast)){\bf S}_m(\tilde{B}_{m,n},v_\ast))=o_p(b_n^{1/4}).
\end{equation}
\begin{discuss}
{\colorr Ogihara 2017 Lemma 4.4の$\hat{\Psi}_{2,m}$の評価と同様にやる．まず$B_{m,n}$の$\hat{v}_n\to v_\ast$としてよい．}
\end{discuss}
The left-hand side of (\ref{tildeSigma-est-eq1}) can be rewritten as 
\begin{equation*}
\sum_m{\rm tr}({\bf C}_m^{(1)}{\bf S}_m(\tilde{B}_{m,n},v_\ast))\cdot (\hat{X}_{m-1}-X_{s_{m-1}})+o_p(b_n^{1/4}),
\end{equation*}
where ${\bf C}_m^{(1)}=\int_0^1\sum_{j=1}^2(-1)^{j-1}\partial_x {\bf S}_m^{-1}(\Sigma(s_{m-1},{\bf X}_{\alpha,m},\sigma_{j,n}),v_\ast)d\alpha$.
\begin{discuss}
{\colorr 
\begin{equation*}
{\bf S}_m(B_{m,n},v_\ast)-{\bf S}_m(\tilde{B}_{m,n},v_\ast)={\bf S}_m\bigg(\frac{\ell_n(\hat{v}_{i,n}-v_{i,\ast})}{T\Phi_1k^i_m}1_{i=j},0\bigg)
\end{equation*}
より
\begin{eqnarray}
&&\sum_m{\rm tr}(({\bf S}_m(\Sigma_m^{-1}(\sigma_{1,n}),v_\ast)-{\bf S}_m^{-1}(\tilde{\Sigma}_m(\sigma_{1,n}),v_\ast))({\bf S}_m(B_{m,n},v_\ast)-{\bf S}_m(\tilde{B}_{m,n},v_\ast)) \nonumber \\
&&\quad =O_p(\ell_n/k_nb_n^{1/2}\ell_n^{-1}\ell_nb_n^{-1/2})=o_p(b_n^{1/4}). \nonumber
\end{eqnarray}
${\bf S}_m(\Sigma(s_{m-1},t\hat{X}_{m-1}+(1-t)X_{s_{m-1}},\sigma),v_\ast)$の非退化性に注意．
}
\end{discuss}

Then, using Lemma~\ref{Bmn-error-est-lemma}, Sobolev's inequality, and an estimate similar to (\ref{tr2-est}) yields
\begin{eqnarray}
&&\sum_{j=1}^2(-1)^{j-1}\sum_m{\rm tr}(({\bf S}_m^{-1}(\Sigma_m(\sigma_{j,n}),v_\ast)-{\bf S}_m^{-1}(\tilde{\Sigma}_m(\sigma_{j,n}),v_\ast)){\bf S}_m(\tilde{B}_{m,n},v_\ast)) \nonumber \\
&&\quad =\sum_m{\rm tr}({\bf C}^{(1)}_m{\bf S}_m(\tilde{\Sigma}_{m-1,\dagger},v_\ast))\cdot (\hat{X}_{m-1}-X_{s_{m-1}})+o_p(b_n^{1/4}) \nonumber \\
&&\quad =\sum_m{\rm tr}({\bf C}^{(2)}_m{\bf S}_m(\tilde{\Sigma}_{m-1,\dagger},v_\ast))\cdot (\hat{X}_{m-1}-X_{s_{m-1}})+o_p(b_n^{1/4}) \nonumber \\
&&\quad =O_p(\ell_n^{1/2}b_n^{1/2}\ell_n^{-1}\ell_n^{-1/2})+o_p(b_n^{1/4})=o_p(b_n^{1/4}), \nonumber
\end{eqnarray}
where ${\bf C}_m^{(2)}=\sum_{j=1}^2(-1)^{j-1}\partial_x {\bf S}_m(\Sigma(s_{m-1},X_{s_{m-2}},\sigma_{j,n}),v_\ast)$.
\begin{discuss}
{\colorr ここは$q$乗で評価できない項がないからソボレフで行ける．${\bf C}_m^{(1)}$の${\bf X}_{\alpha,m}\to X_{s_{m-2}}$とするときは${\bf X}_{\alpha,m}-X_{s_{m-2}}$が前に出てレートが落ちるのでよい．
$X_{s_{m-2}}$にした後も(\ref{tr2-est})とsimilar}
\end{discuss}

Similarly, we obtain
\begin{eqnarray}
\sum_{j=1}^2(-1)^{j-1}\sum_m\log\frac{\det {\bf S}_m(\Sigma_m(\sigma_{j,n}),v_\ast)}{\det {\bf S}_m(\tilde{\Sigma}_m(\sigma_{j,n}),v_\ast)}=o_p(b_n^{1/4})  \nonumber
\end{eqnarray}
by using the equation: 
\begin{eqnarray}
&&\partial_\alpha\log \det{\bf S}_m(\Sigma(s_{m-1},{\bf X}_{\alpha,m},\sigma),v_\ast) \nonumber \\
&&\quad ={\rm tr}({\bf S}_m^{-1}\partial_x{\bf S}_m)(\Sigma(s_{m-1},{\bf X}_{\alpha,m},\sigma),v_\ast)(\hat{X}_{m-1}-X_{s_{m-1}}). \nonumber
\end{eqnarray}

We also obtain estimates for the limit functions by Sobolev's inequality and similar arguments.
\begin{discuss}
{\colorr 
\begin{equation*}
b_n^{1/2}\ell_n^{-1}\sum_m{\rm tr}\Big(A(\tilde{a})B_{m,n}A(\tilde{a})\big\{(A(\tilde{a})\Sigma_m(\sigma_{j,n})A(\tilde{a}))^{-1/2}
-(A(\tilde{a})\tilde{\Sigma}_m(\sigma_{j,n})A(\tilde{a}))^{-1/2}\big\}\Big)=o_p(b_n^{1/4}),
\end{equation*}
\begin{equation*}
b_n^{1/2}\ell_n^{-1}\sum_m{\rm tr}\Big((A(\tilde{a})\Sigma_m(\sigma_{j,n})A(\tilde{a}))^{1/2}
-(A(\tilde{a})\tilde{\Sigma}_m(\sigma_{j,n})A(\tilde{a}))^{1/2}\Big)=o_p(b_n^{1/4})
\end{equation*}
を示せばよい．

\begin{eqnarray}
&&b_n^{1/2}\ell_n^{-1}\sum_m{\rm tr}\Big((A(\tilde{a})\Sigma_m(\sigma_{j,n})A(\tilde{a}))^{1/2}-(A(\tilde{a})\tilde{\Sigma}_m(\sigma_{j,n})A(\tilde{a}))^{1/2}\Big) \nonumber \\
&&\quad =b_n^{1/2}\ell_n^{-1}\sum_m\int^1_0{\rm tr}(\partial_x(A(\tilde{a})\Sigma(s_{m-1},{\bf X}_{t,m},\sigma_{j,n})A(\tilde{a}))^{1/2})dt(\hat{X}_{m-1}-X_{s_{m-1}}) \nonumber \\
&&\quad =b_n^{1/2}\ell_n^{-1}\sum_m{\rm tr}(\partial_x(A(\tilde{a})\tilde{\Sigma}_m(\sigma_{j,n})A(\tilde{a}))^{1/2})(\hat{X}_{m-1}-X_{s_{m-1}})+o_p(b_n^{1/4}), \nonumber
\end{eqnarray}
where ${\bf X}_{t,m}$ is defined in (\ref{bfX-def}).

\begin{equation*}
\sup_\sigma E\bigg[\bigg|b_n^{1/2}\ell_n^{-1}\sum_m{\rm tr}(\partial_\sigma^l\partial_x(A(\tilde{a})\tilde{\Sigma}_m(\sigma)A(\tilde{a}))^{1/2})(\hat{X}_{m-1}-X_{s_{m-1}})\bigg|^q\bigg]<\infty
\end{equation*}
for $l=0,1$より，ソボレフの不等式から上の式は$o_p(b_n^{1/4})$.

\begin{equation*}
\Phi_m(\sigma)=(A(\tilde{a})\Sigma_m(\sigma)A(\tilde{a}))^{-1/2}-(A(\tilde{a})\tilde{\Sigma}_m(\sigma)A(\tilde{a}))^{-1/2}
\end{equation*}
とおくと，
\begin{eqnarray}
&&E\bigg[\bigg|b_n^{1/2}\ell_n^{-1}\sum_m{\rm tr}(A(\tilde{a})(B_{m,n}-\Sigma_{m-1,\dagger})A(\tilde{a})\partial_\sigma^l\Phi_m(\sigma))\bigg|^q\bigg]^{1/q} \nonumber \\
&&\quad =O(b_n^{1/2}\ell_n^{-1}\ell_n\ell_n^{-1/2}(k_n^{-1/2}+\ell_n^{-1}))+O(b_n^{-1/2}\ell_n^{-1}\ell_n^{1/2}\ell_n^{-1/2})=o(b_n^{1/4}) \nonumber
\end{eqnarray}
for $l=0,1$. ソボレフの不等式より
\begin{equation*}
b_n^{1/2}\ell_n^{-1}\sum_m{\rm tr}(A(\tilde{a})(B_{m,n}-\Sigma_{m-1,\dagger})A(\tilde{a})\partial_\sigma^l\Phi_m(\sigma_{j,n}))=o_p(b_n^{1/4}).
\end{equation*}

よってあとは
\begin{equation*}
b_n^{1/2}\ell_n^{-1}\sum_m{\rm tr}(A(\tilde{a})\Sigma_{m-1,\dagger}A(\tilde{a})\Phi_m(\sigma_{j,n}))=o_p(b_n^{1/4})
\end{equation*}
を示せばよいが同様にソボレフから示せる．

\begin{equation*}
E\bigg[\bigg|b_n^{1/2}\ell_n^{-1}\sum_m{\rm tr}(A(\tilde{a})\Sigma_{m-1,\dagger}A(\tilde{a})\partial_\sigma^l\Phi_m(\sigma))\bigg|^q\bigg]=O(b_n^{1/2}\ell_n^{-1}\ell_n^{1/2}\ell_n^{-1/2})=o(b_n^{1/4}).
\end{equation*}
$A^{1/2}$の微分可能性について：$A$を対角化する直行行列$V$を考えると，可逆行列$V'$による対角化は非対角成分を順に$0$にしていけばよいので行列要素の有理関数で書ける．
そこから$V'$の各行ベクトルの内積を$0$にしたり，ノルムを$1$にする操作も有利関数で書けるので$V$も有理関数で書け，ある近傍で同じ有理関数で書けることもわかる．
$A^{1/2}=V\Lambda^{1/2}V^\top$で$\partial A^{1/2}=\partial V\Lambda^{1/2}V^\top+V\partial\Lambda^{1/2}V^\top+V\Lambda^{1/2}\partial V^\top$より微分は可能で，
$V$も有理関数で書かれているので，$\partial A^{1/2}-\partial B^{1/2}=O(\lVert A-B\rVert \vee \lVert \partial A-\partial B\rVert )$.
}
\end{discuss}
\end{proof}

\begin{lemma}\label{log-likelihood-eq-lemma}
Assume [B1$'$], [B2], and [V]. Then,
\begin{eqnarray}
&&\check{H}_n(\sigma_{1,n}, \hat{v}_n)-\check{H}_n(\sigma_{2,n}, \hat{v}_n) \nonumber \\
&&\quad= \Delta_n(\sigma_{1,n},\sigma_{2,n})-b_n^{1/2}(D(\Sigma(\sigma_{1,n}),\Sigma_\dagger)-D(\Sigma(\sigma_{2,n}),\Sigma_\dagger))+o_p(b_n^{1/4}). \nonumber
\end{eqnarray}
\end{lemma}

\begin{proof}
Let $\Sigma_{m,j}=\Sigma_m(\sigma_{j,n})$ and
\begin{equation*}
\mathcal{G}(v,\sigma)=\frac{1}{2}\sum_mG_m(\hat{a}_m,\Sigma_m(\sigma),B_{m,n},v).
\end{equation*}
From this, we obtain
\begin{eqnarray}\label{G-diff-eq}
&&\mathcal{G}(\hat{v}_n,\sigma_{1,n})-\mathcal{G}(\hat{v}_n,\sigma_{2,n})-\mathcal{G}(v_{\ast},\sigma_{1,n})+\mathcal{G}(v_{\ast},\sigma_{2,n}) \\
&&\quad =\frac{1}{2}\sum_m{\rm tr}(({\bf S}_m^{-1}(\Sigma_{m,1},\hat{v}_n)-{\bf S}_m^{-1}(\Sigma_{m,2},\hat{v}_n)){\bf S}_m(B_{m,n},\hat{v}_n)) \nonumber \\
&&\quad \quad -\frac{1}{2}\sum_m{\rm tr}(({\bf S}_m^{-1}(\Sigma_{m,1},v_{\ast})-{\bf S}_m^{-1}(\Sigma_{m,2},v_{\ast})){\bf S}_m(B_{m,n},v_{\ast})) \nonumber \\
&&\quad \quad +\frac{1}{2}\sum_m\log\frac{\det {\bf S}_m(\Sigma_{m,1},\hat{v}_n)}{\det {\bf S}_m(\Sigma_{m,2},\hat{v}_n)}
-\frac{1}{2}\sum_m\log\frac{\det {\bf S}_m(\Sigma_{m,1},v_{\ast})}{\det {\bf S}_m(\Sigma_{m,2},v_{\ast})}. \nonumber
\end{eqnarray}
Let $v_s=(1-s)v_\ast+s\hat{v}_n$ for $s\in [0,1]$, then the first and the second terms of the right-hand side of (\ref{G-diff-eq}) become
\begin{eqnarray}
&&\frac{1}{2}\sum_m\int^1_0\partial_v{\rm tr}(({\bf S}_m^{-1}(\Sigma_{m,1},v_s)-{\bf S}_m^{-1}(\Sigma_{m,2},v_s)){\bf S}_m(B_{m,n},v_s))ds(\hat{v}_n-v_{\ast}) \nonumber \\
&&\quad =O_p(b_n^{1/2}\ell_n^{-1}b_n^{-1/2}\ell_n)=o_p(b_n^{1/4}). \nonumber
\end{eqnarray}
\begin{discuss}
{\colorr ${\rm tr}(S^{-1}S)=O_p(b_n^{1/2}\ell_n^{-1})$をgeneralに書いたLemma がほしいか}
\end{discuss}
The third term of the right-hand side of (\ref{G-diff-eq}) is similarly estimated, giving
\begin{equation*}
\mathcal{G}(\hat{v}_n,\sigma_{1,n})-\mathcal{G}(\hat{v}_n,\sigma_{2,n})-\mathcal{G}(v_{\ast},\sigma_{1,n})+\mathcal{G}(v_{\ast},\sigma_{2,n})=o_p(b_n^{1/4}).
\end{equation*}

Therefore, it is sufficient to show that
\begin{eqnarray}\label{log-likelihood-eq-lemma-eq1}
&&\sum_{j=1}^2(-1)^{j-1}\sum_mG_m(\hat{a}_m,\Sigma_{m,j},B_{m,n},v_{\ast}) \\
&&\quad -\sum_{j=1}^2(-1)^{j-1}\sum_mG_m(\hat{a}_m,\tilde{\Sigma}_m(\sigma_{j,n}),\tilde{\Sigma}_{m,\dagger},v_{\ast})=o_p(b_n^{1/4}), \nonumber 
\end{eqnarray}
by Proposition \ref{Hn-diff-prop}.
\begin{discuss}
{\colorr
\begin{eqnarray}
&&\check{H}_n(\sigma_{1,n}, \hat{v}_n)-\check{H}_n(\sigma_{2,n}, \hat{v}_n)
=\tilde{H}_n(\sigma_{1,n})-\tilde{H}_n(\sigma_{2,n}) +o_p(b_n^{1/4}) \nonumber \\
&& \quad +\frac{1}{2}\sum_m\bigg\{G_m\bigg(\frac{k^1_m}{Tk_n\hat{v}_{1,n}},\frac{k^2_m}{Tk_n\hat{v}_{2,n}},\tilde{\Sigma}_{m,1},B_{m,n},v_{\ast}\bigg)
-G_m\bigg(\frac{k^1_m}{Tk_n\hat{v}_{1,n}},\frac{k^2_m}{Tk_n\hat{v}_{2,n}},\tilde{\Sigma}_{m,2},B_{m,n},v_{\ast}\bigg)\bigg\}. \nonumber
\end{eqnarray}
}
\end{discuss}

Condition $[B2]$ yields
\begin{equation*}
\hat{a}_m^j-\tilde{a}_{s_{m-1}}^j=O_p(k_n^{-1/2}+\ell_n^{-1})=o_p(b_n^{1/4}).
\end{equation*}
\begin{discuss}
{\colorr [B2]で$a^j_{s_{m-1}}$を使うから$a_t$のHolder連続性はここでは必要ない}
\end{discuss}
Since there exists some $\epsilon>0$ such that $P[\min_j \hat{v}_{j,n}\geq \epsilon]\to 1$ as $n\to\infty$ by [V], 
Lemmas~\ref{Bmn-error-est-lemma} and~\ref{rootA-diff-lemma} yield
\begin{equation}\label{log-likelihood-eq-lemma-eq2}
\sum_mG_m\bigg(\hat{a}_m,\Sigma_{m,j},B_{m,n},v_{\ast}\bigg) 
-\sum_mG_m\bigg(\tilde{a}_{s_{m-1}},\Sigma_{m,j},B_{m,n},v_{\ast}\bigg)=o_p(b_n^{1/4})
\end{equation}
for $j\in\{1,2\}$.
\begin{discuss}
{\colorr
\begin{eqnarray}
&&{\rm tr}((A(a)BA(a))^{1/2})-{\rm tr}((A(b)BA(b))^{1/2}) \nonumber \\
&&\quad \leq C\lVert (A(b)BA(b))^{1/2}\rVert \lVert A(a)BA(a) - A(b)BA(b)\rVert  \nonumber \\
&&\quad =O_p(1)\times \max_j(\hat{a}_m^j-\tilde{a}_{s_{m-1}}^j)=o_p(b_n^{1/4}). \nonumber
\end{eqnarray}
${\rm tr}(A(a)(C-B)A(a)(A(a)BA(a))^{-1/2})$の項の評価も同様．
}
\end{discuss}

Together with Lemmas~\ref{tildeSigma-est-lemma} and~\ref{Bmn-error-est-lemma}, it is sufficient to show that
\begin{eqnarray}\label{log-likelihood-eq-lemma-eq3}
&&\sum_mE_m(\tilde{a}_{s_{m-1}},\tilde{\Sigma}_m(\sigma_{j,n}),\tilde{B}_{m,n},v_{\ast}) \\
&& \quad -\sum_mE_m(\tilde{a}_{s_{m-1}},\tilde{\Sigma}_m(\sigma_{j,n}),\tilde{\Sigma}_{m,\dagger},v_{\ast})=o_p(b_n^{1/4}) \nonumber
\end{eqnarray}
for $j\in\{1,2\}$.
\begin{discuss}
{\colorr $E_m(\tilde{a}_{s_{m-1}},\Sigma_m(\sigma_{j,n}),B_{m,n},v_{\ast})-E_m(\tilde{a}_{s_{m-1}},\Sigma_m(\sigma_{j,n}),\tilde{B}_{m,n},v_{\ast})$
の評価は，Lemma \ref{Bmn-error-est-lemma}から
\begin{equation*}
{\rm tr}(\tilde{S}_m^{-1}(\sigma_{j,n})D'_m)(B_{m,n}-\tilde{B}_{m,n})=O_p(b_n^{1/2}\ell_n^{-1}b_n^{1/2}k_n^{-2})=o_p(b_n^{1/4}\ell_n^{-1}),
\end{equation*}
\begin{equation*}
b_n^{1/2}\ell_n^{-1}{\rm tr}(A(\tilde{a})(B_{m,n}-\tilde{B}_{m,n})A(\tilde{a})(A(\tilde{a})\Sigma_mA(\tilde{a}))^{-1/2})=O_p(b_n^{1/2}\ell_n^{-1}b_n^{1/2}k_n^{-2})=o_p(b_n^{1/4}\ell_n^{-1})
\end{equation*}
よりOK.
}
\end{discuss}

The left-hand side of (\ref{log-likelihood-eq-lemma-eq3}) can be rewritten as
\begin{equation*}
\sum_m\sum_{i,j}([\tilde{B}_{m,n}]_{ij}-[\tilde{\Sigma}_{m,\dagger}]_{ij})\Phi_{i,j,m}+o_p(b_n^{1/4}).
\end{equation*}
\begin{discuss}
{\colorr $E_m$の第二項を$s_{m-1}\to \int^{s_m}_{s_{m-1}}$とするからここも$A^{-1/2}-B^{-1/2}$のマルチンゲール評価が必要．よって$\partial_\sigma \Sigma$のりぷしっつ連続性が必要．
\begin{equation*}
\frac{T}{2}b_n^{1/2}\ell_n^{-1}[\tilde{a}_{s_{m-1}}^i\tilde{a}_{s_{m-1}}^j)^{1/2}(\mathcal{D}(s_{m-1},\Sigma(\sigma_{1,n}))^{-1/2}]_{ij}
-\frac{1}{2}b_n^{1/2}\int^{s_m}_{s_{m-1}}[\tilde{a}_t^i\tilde{a}_t^j)^{1/2}(\mathcal{D}(t,\Sigma(\sigma_{1,n}))^{-1/2}]_{ij}dt
\end{equation*}
の評価．}
\end{discuss}
Lemmas \ref{Bmn-error-est-lemma} and \ref{Phi-est-lemma} yield
\begin{eqnarray}
&&E\bigg[\bigg|\sum_m\sum_{i,j}([\tilde{B}_{m,n}]_{ij}-[\tilde{\Sigma}_{m,\dagger}]_{ij})\partial_\sigma^l\Phi_{i,j,m}(\sigma)\bigg|^q\bigg] \nonumber \\
&&\quad \leq C_qE\bigg[\bigg|\sum_m\sum_{i,j}E_m[[\tilde{B}_{m,n}]_{ij}-[\tilde{\Sigma}_{m,\dagger}]_{ij}]\partial_\sigma^l\Phi_{i,j,m}(\sigma)\bigg|^q\bigg] \nonumber \\
&&\quad\quad  +C_qE\bigg[\bigg|\sum_m\bigg(\sum_{i,j}([\tilde{B}_{m,n}]_{ij}-[\tilde{\Sigma}_{m,\dagger}]_{ij})\partial_\sigma^l\Phi_{i,j,m}(\sigma)\bigg)^2\bigg|^{q/2}\bigg] \nonumber \\
&&\quad = \bar{R}_n((k_n^{-1/2}+\ell_n^{-1})^2\ell_n^2(b_n^{-5/8}k_n+1)^2)+\bar{R}_n(\ell_n(b_n^{-5/8}k_n+1)^2)=o_p(b_n^{1/2}) \nonumber
\end{eqnarray}
for any $\sigma$ and $l\in\{0,1\}$. Sobolev's inequality then provides the desired results.
\begin{discuss}
{\colorr Radon-Nikodymを使えば互いに絶対連続な測度間でtightnessが引き継がれることがわかる．}
\end{discuss}
\end{proof}

\noindent
{\bf Proof of Theorem \ref{optConvTheorem}}
By localization techniques, we can additionally assume $[B1']$.
\begin{discuss}
{\colorr $\Delta_n$の評価ではCLTで使うpredictableでない係数の除去の技術が必要.}
\end{discuss}
For the $\bar{\Lambda}$-valued random variable $\sigma_\ast$ satisfying (\ref{sigma-ast-eq}),
Lemma \ref{log-likelihood-eq-lemma} yields
\begin{eqnarray}
&&-b_n^{1/4}(D(\Sigma(\check{\sigma}_n),\Sigma_\dagger)-D(\Sigma(\sigma_{\ast}),\Sigma_\dagger)) \nonumber \\
&&\quad =b_n^{-1/4}(\check{H}_n(\check{\sigma}_n,\hat{v}_n)-\check{H}_n(\sigma_{\ast},\hat{v}_n)) 
-b_n^{-1/4}\Delta_n(\check{\sigma}_n,\sigma_{\ast})+o_p(1). \nonumber
\end{eqnarray}
Moreover, the definition of $\check{\sigma}_n$ yields $\check{H}_n(\check{\sigma}_n,\hat{v}_n)-\check{H}_n(\sigma_{\ast},\hat{v}_n)\geq 0$, 
and applying (\ref{consistency-eq2}) and (\ref{tr2-est}) yields
\begin{equation}\label{Delta-est}
E_\Pi\Big[\sup_{\sigma_1,\sigma_2}|b_n^{-\frac{1}{4}}\Delta_n(\sigma_1,\sigma_2)|^q\Big]=O_p(b_n^{-\frac{q}{4}}(\ell_nb_n^{-1/2}k_n)^{\frac{q}{2}})+o_p(1)=O_p(1)
\end{equation}
for a sufficiently large $q$.
\begin{discuss}
{\colorr 
\begin{equation*}
\partial_\sigma \Delta_n(\sigma,\sigma')=-\frac{1}{2}\sum_m{\rm tr}(\partial_\sigma \tilde{S}_m^{-1}(\sigma)(\tilde{Z}_m\tilde{Z}_m^\top-\tilde{S}_{m,\dagger})).
\end{equation*}
}
\end{discuss}

Hence we have
\begin{equation*}
P[b_n^{1/4}(D(\Sigma(\check{\sigma}_n),\Sigma_\dagger)-D(\Sigma(\sigma_{\ast}),\Sigma_\dagger))>K]
\leq P[|b_n^{1/4}\Delta_n(\check{\sigma}_n,\sigma_\ast)|>K-1]<\epsilon
\end{equation*}
for sufficiently large $K$ and $n$.

\qed

\subsection{Proof of Theorem \ref{fast-estimator-thm}}\label{fast-estimator-proof-section}
We can assume $[B1']$.
Since $(\hat{a}_m^j)^{1/2}=(\tilde{a}^j_{s_{m-1}})^{1/2}+O_p(k_n^{-1/2})$ by [B2], we obtain
\begin{discuss}
{\colorr 残差はexplicitではないが問題ない．ここでA2ではなくB2が必要．}
\end{discuss}
\begin{eqnarray}\label{fast-est-proof-eq1}
&& \\
\dot{H}_n(\sigma)&=&
-\frac{T\ell_n^{-1}}{4}\sum_{m=2}^{\ell_n}{\rm tr}\Big((\mathcal{D}(s_{m-1},\tilde{B}_{m,n})+\mathcal{D}(s_{m-1},\Sigma(\sigma)))\mathcal{D}(s_{m-1},\Sigma(\sigma))^{-1/2}\Big) \nonumber \\
&&+o_p(b_n^{-1/4}) \nonumber
\end{eqnarray}
by Lemmas \ref{rootA-diff-lemma} and~\ref{Bmn-error-est-lemma}.
\begin{discuss}
{\colorr $(\hat{a})^{-1}$の評価は$\underbar{k}_n$の評価からでるか}
\end{discuss}

Moreover, Lemma \ref{Bmn-error-est-lemma}, the Burkholder--Davis--Gundy inequality, and Sobolev's inequality together yield
\begin{equation}\label{fast-est-proof-eq2}
\sup_\sigma\bigg|\ell_n^{-1}\sum_{m=2}^{\ell_n}{\rm tr}((\mathcal{D}(s_{m-1},\tilde{B}_{m,n})-\mathcal{D}(s_{m-1},\Sigma_{\dagger}))\mathcal{D}(s_{m-1},\Sigma(\sigma))^{-1/2})\bigg|=O_p(\ell_n^{-1/2}).
\end{equation}
\begin{discuss}
{\colorr ソボレフを使うために$\partial_\sigma \Sigma$が必要だから[B1]がいる．}
\end{discuss}

Equations (\ref{fast-est-proof-eq1}) and (\ref{fast-est-proof-eq2}) yield
\begin{equation}\label{fast-est-proof-eq3}
\dot{H}_n(\sigma_{j,n})=-D(\Sigma(\sigma_{j,n}),\Sigma_\dagger)-\frac{1}{2}\int^T_0{\rm tr}(\mathcal{D}(t,\Sigma_\dagger)^{1/2})dt+O_p(\ell_n^{-1/2}).
\end{equation}

Since $\dot{H}_n(\dot{\sigma}_n)\geq \dot{H}_n(\sigma_\ast)$, for any $\epsilon>0$, (\ref{fast-est-proof-eq3}) yields
\begin{equation*}
P[\ell_n^{1/2}(D(\Sigma(\dot{\sigma}_n),\Sigma_\dagger)-D(\Sigma(\sigma_\ast),\Sigma_\dagger)>K]<\epsilon
\end{equation*}
for sufficiently large $K$ and $n$.

\qed

\subsection{Proof of Theorem \ref{mixed-normality-thm}}\label{mixed-normality-proof-section}

Similar to the case with Lemma~\ref{log-likelihood-eq-lemma}, we have
\begin{discuss}
{\colorr Lemma~\ref{HtoTildeH-lemma}や(\ref{consistency-eq2}), Lemmas~\ref{Phi-est-lemma},~\ref{Psi-est-lemma}を$\partial_\sigma^2$をつけて証明できるから，
Lemma~\ref{log-likelihood-eq-lemma}を微分しても成り立つことが分かる}
\end{discuss}
\begin{eqnarray}\label{delH-eq}
&& \\
b_n^{-1/4}\partial_\sigma \check{H}_n(\sigma_\ast)&=& b_n^{-1/4}\partial_{\sigma_1} \Delta_n(\sigma_\ast,\sigma_{2,n})+o_p(1) \nonumber \\
&=&-\frac{1}{2}b_n^{-1/4}\sum_m{\rm tr}(\partial_\sigma \tilde{S}_m^{-1}(\sigma_\ast)(\tilde{Z}_m\tilde{Z}_m^\top-\tilde{S}_{m,\dagger}))+o_p(1), \nonumber
\end{eqnarray}
since $\partial_\sigma D(\Sigma(\sigma),\Sigma_\dagger)|_{\sigma=\sigma_\ast}=0$.

Using this, we apply Proposition~\ref{random-param-convergence} 
with $x=\sigma$ and 
\begin{equation*}
Z(x)=-\frac{1}{2}b_n^{-1/4}\sum_m{\rm tr}(\partial_\sigma \tilde{S}_m^{-1}(\sigma)(\tilde{Z}_m\tilde{Z}_m^\top-\tilde{S}_{m,\dagger}))
\end{equation*}
to show that (\ref{score-st-conv}) holds.

First, we show the convergence under the assumption that $\sigma$ is deterministic. 
Recall that $\mathcal{N}_k$ is defined in Section~\ref{asym-mixed-normality-section}.
\begin{lemma}\label{deterministic-st-conv-lemma}
Assume [B1$'$] and [B2]. Then for any deterministic $\sigma$，
\begin{equation*}
-\frac{1}{2}b_n^{-1/4}\sum_m{\rm tr}(\partial_\sigma \tilde{S}_m^{-1}(\sigma)(\tilde{Z}_m\tilde{Z}_m^\top-\tilde{S}_{m,\dagger}))\to^{s\mathchar`-\mathcal{L}}\Gamma_{1,\sigma}^{1/2}\mathcal{N}_\gamma.
\end{equation*}
\end{lemma}
\begin{proof}
By Theorem 3.2 in Jacod~\cite{jac97}, it is sufficient to show that
\begin{equation}\label{deterministic-st-conv-lemma-eq1}
\frac{b_n^{-1/2}}{4}\sum_mE_m[{\rm tr}(\partial_\sigma \tilde{S}_m^{-1}(\tilde{Z}_m\tilde{Z}_m^\top-\tilde{S}_{m,\dagger}))^2] \overset{P}\to  \Gamma_{1,\sigma},
\end{equation}
and
\begin{equation}\label{deterministic-st-conv-lemma-eq2}
b_n^{-1}\sum_mE_m[{\rm tr}(\partial_\sigma \tilde{S}_m^{-1}(\tilde{Z}_m\tilde{Z}_m^\top-\tilde{S}_{m,\dagger}))^4] \overset{P}\to  0.
\end{equation}

That equation (\ref{deterministic-st-conv-lemma-eq2}) holds can be easily shown by using Lemma \ref{ZSZ-est-lemma}.
\begin{discuss}
{\colorr $E_m[{\rm tr}(\partial_\sigma \tilde{S}_m^{-1}(\tilde{Z}_m\tilde{Z}_m^\top-\tilde{S}_{m,\dagger}))^4]=R_n((b_n^{-4}k_n^7)\vee (b_n^{-2}k_n^4))$.}
\end{discuss}
To show (\ref{deterministic-st-conv-lemma-eq1}), 
Lemma \ref{ZSZ-est-lemma} yields
\begin{eqnarray}
&&\frac{1}{4}\sum_mE_m[{\rm tr}(\partial_\sigma \tilde{S}_m^{-1}(\tilde{Z}_m\tilde{Z}_m^\top-\tilde{S}_{m,\dagger}))^2]
=\frac{1}{2}\sum_m{\rm tr}((\partial_\sigma \tilde{S}_m^{-1}\tilde{S}_{m,\dagger})^2)+\bar{R}_n(\ell_n) \nonumber \\
&&\quad =\frac{1}{2}\sum_m{\rm tr}((\tilde{S}_m^{-1}\partial_\sigma \tilde{S}_m(\mathcal{E}_{\mathcal{K}_m^\gamma}+\tilde{S}_m^{-1}(S_{m,\dagger}-\tilde{S}_m)))^2)+\bar{R}_n(\ell_n) \nonumber \\
&&\quad =\frac{1}{2}\sum_m\bigg\{{\rm tr}((\tilde{S}_m^{-1}\partial_\sigma \tilde{S}_m)^2)+2{\rm tr}((\tilde{S}_m^{-1}\partial_\sigma \tilde{S}_m)^2\tilde{S}_m^{-1}(\tilde{S}_{m,\dagger}-\tilde{S}_m)) \nonumber \\
&&\quad \quad \quad \quad \quad 
+{\rm tr}((\tilde{S}_m^{-1}\partial_\sigma \tilde{S}_m\tilde{S}_m^{-1}(\tilde{S}_{m,\dagger}-\tilde{S}_m))^2)\bigg\}+\bar{R}_n(\ell_n). \nonumber
\end{eqnarray}
We only show that 
\begin{equation*}
\frac{1}{2}b_n^{-1/2}\sum_m{\rm tr}((\tilde{S}_m^{-1}\partial_\sigma \tilde{S}_m\tilde{S}_m^{-1}(\tilde{S}_{m,\dagger}-\tilde{S}_m))^2)
\overset{P}\to \frac{1}{2}\int^T_0\mathfrak{K}_2(\mathcal{D}_{t,\sigma},\partial_\sigma \mathcal{D}_{t,\sigma},\mathcal{D}_{t,\dagger}-\mathcal{D}_{t,\sigma})dt.
\end{equation*}
Other terms are calculated similarly.

Let $\mathfrak{H}_m(x)=\mathfrak{F}_m(x)+b_n^{-1}(\tilde{\Sigma}_m-\tilde{\Sigma}'_m)$. Then similarly to (\ref{tr-invD-est}), we obtain
\begin{eqnarray} 
&&\frac{1}{2}b_n^{-1/2}\sum_m{\rm tr}((\tilde{S}_m^{-1}\partial_\sigma \tilde{S}_m\tilde{S}_m^{-1}(\tilde{S}_{m,\dagger}-\tilde{S}_m))^2) \nonumber \\
&&\quad \approx\frac{T\check{a}_m^1\ell_n^{-1}}{2\pi b_n^{7/2}}\sum_m\int^\pi_0{\rm tr}((\partial_\sigma \tilde{\Sigma}_m\mathfrak{H}_m^{-1}(x)(\tilde{\Sigma}_{m,\dagger}-\tilde{\Sigma}_m)\mathfrak{H}_m^{-1}(x))^2)\bigg\}dx. \nonumber
\end{eqnarray}
\begin{discuss}
{\colorr $\tilde{S}_m^{-1}$に対するLemma~\ref{invS-eq-lemma}の展開を適用し，(\ref{tr-invD-est})式を使って変形していく．
\begin{eqnarray}
&&{\rm tr}(\tilde{S}_m^{-1}\partial_\sigma \tilde{S}_m\tilde{S}_m^{-1}(\tilde{S}_{m,\dagger}-\tilde{S}_m))^2) \nonumber \\
&&\quad =\sum_{p_1,\cdots,p_4=0}^\infty (-1)^{p_1+\cdots +p_4}\sum_{(i^j_0,\cdots,i^j_{p_j})_{1\leq j\leq 4}}\prod_{j=1}^4\bigg(\prod_{r_j=1}^{p_j}[\tilde{\Sigma}_m]_{i^j_{r_j-1},i^j_{r_j}}\bigg) \nonumber \\
&&\quad \quad \times \partial_\sigma[\tilde{\Sigma}_m]_{i^1_{p_1},i^2_0}\partial_\sigma[\tilde{\Sigma}_m]_{i^3_{p_3},i^4_0}[\tilde{\Sigma}_m-\tilde{\Sigma}_{m,\dagger}]_{i^2_{p_2},i^3_0}[\tilde{\Sigma}_m-\tilde{\Sigma}_{m,\dagger}]_{i^4_{p_4},i^1_0} \nonumber \\
&&\quad \quad \times {\rm tr}\bigg(\prod_{j=1}^4\bigg\{\bigg(\prod_{r_j=1}^{p_j}\tilde{D}_{i^j_{r_j-1},m}\mathfrak{G}_{i^j_{r_j-1},i_{r_j}}\bigg)\tilde{D}_{i^j_{p_j},m}^{-1}\mathfrak{G}_{i^j_{p_j},i^{j+1}_0}\bigg\}\bigg), \nonumber
\end{eqnarray}
where $i^5_0=i^1_0$.
\begin{eqnarray}
&&{\rm tr}\bigg(\prod_{j=1}^4\bigg\{\bigg(\prod_{r_j=1}^{p_j}\tilde{D}_{i^j_{r_j-1},m}\mathfrak{G}_{i^j_{r_j-1},i_{r_j}}\bigg)\tilde{D}_{i^j_{p_j},m}^{-1}\mathfrak{G}_{i^j_{p_j},i^{j+1}_0}\bigg\}\bigg) \nonumber \\
&&\quad =b_n^{-p_1-\cdots -p_4-4}\frac{T\check{a}_m^1k_n}{\pi}\int^\pi_0\prod_{j=1}^4\prod_{r_j=0}^{p_j}[\mathfrak{F}_m(x)]_{i_{r_j}^j}^{-1}dx+o_p. \nonumber
\end{eqnarray}
残差項の評価はLemma~\ref{Phi-est-lemma}と同様．
}
\end{discuss}

Let $\mathfrak{D}_m=\mathfrak{A}_m^{-1/2}\tilde{\Sigma}_m\mathfrak{A}_m^{-1/2}$
and $\hat{\mathfrak{D}}_m=\mathfrak{A}_m^{-1/2}\partial_\sigma \tilde{\Sigma}_m\mathfrak{A}_m^{-1/2}$.
Then, we obtain
\begin{eqnarray} 
&&\frac{1}{2}b_n^{-1/2}\sum_m{\rm tr}((\tilde{S}_m^{-1}\partial_\sigma \tilde{S}_m\tilde{S}_m^{-1}(\tilde{S}_{m,\dagger}-\tilde{S}_m))^2) \nonumber \\
&&\quad \approx \frac{T\check{a}_m^1\ell_n^{-1}}{2\pi b_n^{7/2}}\sum_m\int^\infty_{-\infty}(1+t^2)^3 \nonumber \\
&&\quad \quad \times {\rm tr}(\{\partial_\sigma \tilde{\Sigma}_m(\mathfrak{A}_mt^2+b_n^{-1}\tilde{\Sigma}_m)^{-1}
(\tilde{\Sigma}_{m,\dagger}-\tilde{\Sigma}_m)(\mathfrak{A}_mt^2+b_n^{-1}\tilde{\Sigma}_m)^{-1}\}^2)dt. \nonumber
\end{eqnarray}
Since we have
\begin{equation*}
\frac{(1+t^2)^3}{(t^2+x_1^2)(t^2+x_2^2)(t^2+x_3^2)(t^2+x_4^2)}
=\bigg(1+\frac{1-x_4^2}{t^2+x_4^2}\bigg)\bigg(1+\frac{1-x_3^2}{t^2+x_3^2}\bigg)\bigg(1+\frac{1-x_2^2}{t^2+x_2^2}\bigg)\frac{1}{t^2+x_1^2}
\end{equation*}
for $x_1,\cdots,x_4\in\mathbb{R}$, we can calculate the integral by using Lemma~\ref{residueCalc-lemma}, 
and the terms except $\prod_{j=1}^4(t^2+x_j^2)^{-1}$ can be ignored for the asymptotic case. Then, we obtain
\begin{eqnarray}
&&\frac{1}{2}b_n^{-1/2}\sum_m{\rm tr}((\tilde{S}_m^{-1}\partial_\sigma \tilde{S}_m\tilde{S}_m^{-1}(\tilde{S}_{m,\dagger}-\tilde{S}_m))^2) \nonumber \\
&&\quad \approx \frac{T\check{a}_m^1\ell_n^{-1}}{2\pi b_n^{7/2}}\sum_m\int^\infty_{-\infty}{\rm tr}(\{\hat{\mathfrak{D}}_m(t^2+b_n^{-1}\mathfrak{D}_m)^{-1}
(\mathfrak{D}_{m,\dagger}-\mathfrak{D}_m)(t^2+b_n^{-1}\mathfrak{D}_m)^{-1}\}^2)dt \nonumber \\
&&\quad \approx \frac{T\check{a}_m^1\ell_n^{-1}}{b_n^{7/2}}\sum_m\mathfrak{K}_2(b_n^{-1}\mathfrak{D}_m,\hat{\mathfrak{D}}_m,\mathfrak{D}_{m,\dagger}-\mathfrak{D}_m) \nonumber \\
&&\quad \approx \frac{T\ell_n^{-1}}{2}\sum_m\mathfrak{K}_2(\mathcal{D}_{s_{m-1},\sigma},\partial_\sigma \mathcal{D}_{s_{m-1},\sigma},\mathcal{D}_{s_{m-1},\dagger}-\mathcal{D}_{s_{m-1},\sigma}) \nonumber \\
&&\quad \approx \frac{1}{2}\int^T_0\mathfrak{K}_2(\mathcal{D}_{t,\sigma},\partial_\sigma \mathcal{D}_{t,\sigma},\mathcal{D}_{t,\dagger}-\mathcal{D}_{t,\sigma})dt. \nonumber
\end{eqnarray}
\begin{discuss}
{\colorr $\mathfrak{A}_m\approx 4(\check{a}_m^1)^2{\rm diag}((v_{j,\ast}/\check{a}_m^j)_j)$, $\mathfrak{D}_m\approx (1/4)(\check{a}^1_m)^{-2}\mathcal{D}(s_{m-1},\Sigma_m)$, $\varphi_{\mathfrak{D}_m}(\partial_\sigma \mathfrak{D}_m)\approx (2\check{a}_m^1)^{-1}\varphi_{\mathcal{D}_m}(\partial_\sigma \mathcal{D}_m)$.
$\varphi_{aB}(A)=a^{-1/2}\varphi_B(A)$, $\varphi_B(aA)=a\varphi_B(A)$.
\begin{eqnarray}
&&\mathfrak{K}_2(b_n^{-1}\mathfrak{D}_m,\hat{\mathfrak{D}}_m,\mathfrak{D}_{m,\dagger}-\mathfrak{D}_m) \nonumber \\
&&\quad =\frac{b_n^{7/2}}{2\check{a}_m^1}\mathfrak{K}_2(\mathcal{D}(s_{m-1},\Sigma_m),\partial_\sigma \mathcal{D}(s_{m-1},\Sigma_m),\mathcal{D}(s_{m-1},\Sigma_{m,\dagger})-\mathcal{D}(s_{m-1},\Sigma_m)). \nonumber
\end{eqnarray}
最後に$B^{-1/2}$のヘルダー連続性が必要なのでLemma~\ref{rootA-diff-lemma}を用いる．
}
\end{discuss}
\begin{discuss}
{\colorr 
\begin{eqnarray}
\frac{1}{2}\sum_m{\rm tr}((\tilde{S}_m^{-1}\partial_\sigma \tilde{S}_m)^2)
&\approx& \frac{T\hat{a}_m^1\ell_n^{-1}}{2\pi b_n}\sum_m\int^\pi_0{\rm tr}((\partial_\sigma \mathfrak{D}_m\mathfrak{H}_m^{-1}(x))^2) \nonumber \\
&\approx& \frac{T\hat{a}_m^1\ell_n^{-1}}{2\pi b_n}\sum_m\int^\infty_{-\infty}{\rm tr}((\{\partial_\sigma \mathfrak{D}_m(\mathfrak{A}_mt^2+b_n^{-1}\mathfrak{D}_m)^{-1})^2)(1+t^2)dt, \nonumber
\end{eqnarray}
\begin{eqnarray}
&&\sum_m{\rm tr}((\tilde{S}_m^{-1}\partial_\sigma \tilde{S}_m)^2\tilde{S}_m^{-1}(\tilde{S}_{m,\dagger}-\tilde{S}_m)) \nonumber \\
&&\quad \approx \frac{T\hat{a}_m^1\ell_n^{-1}}{\pi b_n^2}\sum_m\int^\pi_0{\rm tr}((\partial_\sigma \mathfrak{D}_m\mathfrak{H}_m^{-1}(x))^2(\mathfrak{D}_{m,\dagger}-\mathfrak{D}_m)\mathfrak{H}_m^{-1}(x)) \nonumber \\
&&\quad \approx \frac{T\hat{a}_m^1\ell_n^{-1}}{\pi b_n^2}\sum_m\int^\infty_{-\infty}{\rm tr}(\{\partial_\sigma \mathfrak{D}_m(\mathfrak{A}_mt^2+b_n^{-1}\mathfrak{D}_m)^{-1}\}^2(\mathfrak{D}_{m,\dagger}-\mathfrak{D}_m)(\mathfrak{A}_mt^2+b_n^{-1}\mathfrak{D}_m)^{-1})(1+t^2)^2dt. \nonumber
\end{eqnarray}
あとは同様．
}
\end{discuss}
\end{proof}

To apply Proposition~\ref{random-param-convergence}, we consider a series expansion of $\tilde{S}_m^{-1}$.
Recall that $D'_{j,m}$ is defined in (\ref{D'-def-eq}), and ${\bf E}_{i,m}$ and $\mathfrak{G}$ are defined just before Lemma~\ref{invS-eq-lemma}.
Let $R\geq 1$ be a constant such that $\sup_{t,x,\sigma,j}[\Sigma(t,x,\sigma)]_{jj}\leq R$. Let $\check{D}_{j,m,1}=(RD'_{j,m}+v_{j,\ast}M_{j,m})^{-1}$, $\check{D}_{j,m,q}=\check{D}_{j,m,1}(D'_{j,m}\check{D}_{j,m,1})^{q-1}$ for $q>1$ and 
\begin{equation*}
\mathcal{B}=\{(p,i_0,\cdots,i_p,q_0,\cdots,q_p);p\in\mathbb{Z}_+,1\leq i_j\leq \gamma, i_j\neq i_{j-1},q_0,\cdots,q_p\in\mathbb{N}\}.
\end{equation*} 
Then, Lemma~\ref{invS-eq-lemma} and the expansion $\tilde{D}_{j,m}^{-1}=\sum_{p=1}^\infty(R-[\tilde{\Sigma}_m]_{jj})^{p-1}\check{D}_{j,m,p}$ yield
\begin{equation}\label{invS-expansion}
\tilde{S}_m^{-1}(\dot{\sigma})=\sum_{\alpha\in\mathcal{B}}c_{\alpha,m}(\dot{\sigma})S_{\alpha} 
\end{equation}
for any random variable $\dot{\sigma}$,
where $S_{\alpha}={\bf E}_{i_0,m}^\top\check{D}_{i_0,m,q_0}\prod_{j=1}^p(\mathfrak{G}_{i_{j-1},i_j}\check{D}_{i_j,m,q_j}){\bf E}_{i_p,m}^\top$ and
\begin{equation*}
c_{\alpha,m}(\sigma)=(-1)^p\prod_{j=0}^p(R-[\tilde{\Sigma}_m(\sigma)]_{i_j,i_j})^{q_j-1}\prod_{j=1}^p[\tilde{\Sigma}_m(\sigma)]_{i_{j-1},i_j}.
\end{equation*}

\begin{lemma}\label{st-conv-lemma2}
Assume [B1$'$] and [B2]. Then there exist continuous ${\bf F}$-adapted stochastic processes $(E_{\alpha,\alpha'}(t))_{t\in [0,T], \alpha,\alpha'\in \mathcal{B}}$
such that for any $K\in\mathbb{N}$ and continuous functions $(F_k(t,x))_{k=1}^K$ and $(\alpha_k)_{k=1}^K\subset \mathcal{B}$, 
$\Upsilon$ is a symmetric nonnegative definite matrix almost surely and
\begin{eqnarray}
&&\Big(b_n^{-\frac{1}{4}}\sum_mF_k(s_{m-1},X_{s_{m-1}}){\rm tr}(S_{\alpha_k}(\tilde{Z}_m\tilde{Z}_m^\top-\tilde{S}_{m,\dagger}))\Big)_{k=1}^K \nonumber \\
&&\quad \to^{s\mathchar`-\mathcal{L}} \Upsilon^{1/2}\mathcal{N}_K, \nonumber
\end{eqnarray}
where $\Upsilon=([\Upsilon]_{k,k'})_{1\leq k,k'\leq K}$ and $[\Upsilon]_{k,k'}=\int^T_0F_k(t,X_t)F_{k'}(t,X_t)E_{\alpha_k,\alpha_{k'}}(t)dt$.
\end{lemma}
\begin{proof}
Let $F_{k,m}=F_k(s_{m-1},X_{s_{m-1}})$. It is sufficient to show that
\begin{equation*}
b_n^{-\frac{1}{4}}\sum_{k=1}^Ka_k\sum_mF_{k,m}{\rm tr}(S_{\alpha_k}(\tilde{Z}_m\tilde{Z}_m^\top-\tilde{S}_{m,\dagger}))
\to^{s\mathchar`-\mathcal{L}}(a_k)_{1\leq k\leq K}^\top\Upsilon^{1/2}\mathcal{N}_K
\end{equation*}
for any constants $(a_k)_{k=1}^K$.
\begin{discuss}
{\colorr ここは有限和なので上からの評価を気にせず収束を示せる．}
\end{discuss}

For any $t\in (0,T]$, Remark~\ref{ZSZ-est-rem} yields
\begin{eqnarray}
&&\sum_{m=2}^{[\ell_nt]}E_m\bigg[\bigg(b_n^{-1/4}\sum_{k=1}^Ka_kF_{k,m}{\rm tr}(S_{\alpha_k}(\tilde{Z}_m\tilde{Z}_m^\top-\tilde{S}_{m,\dagger}))\bigg)^2\bigg] \nonumber \\
&&\quad =b_n^{-1/2}\sum_{m=2}^{[\ell_nt]}\sum_{k,k'=1}^Ka_ka_{k'}F_{k,m}F_{k',m}{\rm tr}(\tilde{S}_{m,\dagger}S_{\alpha_k}\tilde{S}_{m,\dagger}S_{\alpha_{k'}})+o_p(b_n^{-1/2}), \nonumber
\end{eqnarray}
where $[x]$ denotes the largest integer not larger than $x$.
\begin{discuss}
{\colorr $\epsilon^3$項は相方が$\epsilon$なので消える． 
$\sum_{\alpha,\alpha'\in \mathcal{B}_P}a_\alpha a_{\alpha'}F_{\alpha,m}F_{\alpha',m}\sum_i(E[\epsilon_i^4]-v_\ast^2)(S_\alpha)_{ii}(S_{\alpha'})_{ii}$の評価を
しないといけないからLemma~\ref{ZSZ-est-lemma}を使う．
$p=0$かつ$q_0=1$だとtrのオーダーが変わるが，この時係数が$1$なので$\sigma$の微分で$0$になるから無視できる．}
\end{discuss}
By an argument similar to the proof of Lemma \ref{Phi-est-lemma}, there exists a stochastic process $(E_{\alpha,\alpha'}(t))_{t\in [0,T]}$ with a continuous path such that
\begin{discuss}
{\colorr
$c'_j=Rb_n^{-1}\hat{a}^j_m(\hat{a}_m^1)^{-2}v_{j,\ast}^{-1}$, $\ddot{D}_{j,m,0}=c'_j\mathcal{E}_{k^1_m}+M_{1,m}$とすると，
$S_{m,\dagger}=RD'_m\mathcal{E}+M_{m,\ast}+{\rm diag}((([\Sigma_{m,\dagger}]_{jj}-R)|I^k_j|)_j)$の展開を使うと，
\begin{eqnarray}
&&b_n^{-1/2}{\rm tr}(S_{m,\dagger}S_{\alpha}S_{m,\dagger}S_{\alpha'})=b_n^{-1/2}\sum_{i,j,k,l}(S_{m,\dagger})_{ij}(S_\alpha)_{jk}(S_{m,\dagger})_{kl}(S_{\alpha'})_{li} \nonumber \\
&&\quad \approx \frac{(b_{m,\dagger}^{i'_p}\cdot b_{m,\dagger}^{i_0}-R1_{i'_p=i_0})(b_{m,\dagger}^{i_p}\cdot b_{m,\dagger}^{i'_0}-R1_{i_p=i'_0})}{b_n^{5/2+p+p'}(\hat{a}_m^1)^{p+p'+2}}
{\rm tr}\bigg(\prod_{j=0}^p\frac{(\hat{a}_m^{i_j})^{q_j}\ddot{D}_{i_j,m,0}^{-q_j}}{(\hat{a}_m^1)^{2q_j-1}b_n^{q_j-1}}\prod_{j'=0}^{p'}\frac{(\hat{a}_m^{i'_{j'}})^{q'_{j'}}\ddot{D}_{i'_{j'},m,0}^{-q'_{j'}}}{(\hat{a}_m^1)^{2q'_{j'}-1}b_n^{q'_{j'}-1}}\bigg) \nonumber \\
&&\quad \quad +\frac{b_{m,\dagger}^{i'_p}\cdot b_{m,\dagger}^{i_0}1_{i_p=i'_0}}{b_n^{3/2+p+p'}(\hat{a}_m^1)^{p+p'+1}}
{\rm tr}\bigg(\prod_{j=0}^p\frac{(\hat{a}_m^{i_j})^{q_j}\ddot{D}_{i_j,m,0}^{-q_j}}{(\hat{a}_m^1)^{2q_j-1}b_n^{q_j-1}}\prod_{j'=0}^{p'}\frac{(\hat{a}_m^{i'_{j'}})^{q'_{j'}}\ddot{D}_{i'_{j'},m,0}^{-q'_{j'}}}{(\hat{a}_m^1)^{2q'_{j'}-1}b_n^{q'_{j'}-1}}\ddot{D}_{i_p,m,0}\frac{\hat{a}_m^1}{\hat{a}_m^{i_p}}\bigg) \nonumber \\
&&\quad \quad +\frac{b_{m,\dagger}^{i_p}\cdot b_{m,\dagger}^{i'_0}1_{i'_p=i_0}}{b_n^{3/2+p+p'}(\hat{a}_m^1)^{p+p'+1}}
{\rm tr}\bigg(\prod_{j=0}^p\frac{(\hat{a}_m^{i_j})^{q_j}\ddot{D}_{i_j,m,0}^{-q_j}}{(\hat{a}_m^1)^{2q_j-1}b_n^{q_j-1}}\prod_{j'=0}^{p'}\frac{(\hat{a}_m^{i'_{j'}})^{q'_{j'}}\ddot{D}_{i'_{j'},m,0}^{-q'_{j'}}}{(\hat{a}_m^1)^{2q'_{j'}-1}b_n^{q'_{j'}-1}}\ddot{D}_{i_0,m,0}\frac{\hat{a}_m^1}{\hat{a}_m^{i_0}}\bigg) \nonumber \\
&&\quad \quad +\frac{1_{i'_p=i_0,i_p=i'_0}}{b_n^{1/2+p+p'}(\hat{a}_m^1)^{p+p'}}
{\rm tr}\bigg(\prod_{j=0}^p\frac{(\hat{a}_m^{i_j})^{q_j}\ddot{D}_{i_j,m,0}^{-q_j}}{(\hat{a}_m^1)^{2q_j-1}b_n^{q_j-1}}\prod_{j'=0}^{p'}\frac{(\hat{a}_m^{i'_{j'}})^{q'_{j'}}\ddot{D}_{i'_{j'},m,0}^{-q'_{j'}}}{(\hat{a}_m^1)^{2q'_{j'}-1}b_n^{q'_{j'}-1}}\ddot{D}_{i_0,m,0}\ddot{D}_{i_p,m,0}\frac{(\hat{a}_m^1)^2}{\hat{a}_m^{i_0}\hat{a}_m^{i_p}}\bigg). \nonumber 
\end{eqnarray}
（$RD'_m+M_{m,\ast}$の方は$\check{D}$が一つ減り，$\delta_{ij}$がつく）
($b_n$は$S_{m,\dagger}$からも計２個くる．
$\tilde{G}_{i,j,m}\approx (\hat{a}_m^1)^{-1}b_n^{-1}\mathcal{E}_{k^1_m}$, $\check{D}_{i,m,q}\approx (\hat{a}_m^i)^q/(\hat{a}_m^1)^{2q-1}\ddot{D}_{i,m,0}^{-q}$.
$D_{j,m}^{-1}M_{j,m}\approx \mathcal{E}_{k^j_m}-R(\hat{a}_m^1)^{-1}b_n^{-1}D_{j,m}^{-1}$．$D_{j,m,0}^{-1}\approx (\hat{a}_m^j/\hat{a}_m^1)\ddot{D}_{j,m,0}^{-1}$, $D'_{j,m}\approx (\hat{a}^1_m)^{-1}b_n^{-1}\mathcal{E}_{k^1_m}$)
であり，ある$(c''_q)_{q=1}^M$と$(\beta_q)_{q=1}^M$があって
\begin{eqnarray}
&&{\rm tr}\bigg(\prod_{j=0}^p\ddot{D}_{i_j,m,0}^{-q_j}\prod_{j'=0}^{p'}\ddot{D}_{i'_{j'},m,0}^{-q'_{j'}}\bigg)
\approx \frac{T\hat{a}_m^1k_n}{\pi}\int^\pi_0\prod_{j=0}^p\frac{1}{(c'_{i_j}+2(1-\cos x))^{q_j}}\prod_{j'=0}^{p'}\frac{1}{(c'_{i'_{j'}}+2(1-\cos x))^{q'_{j'}}}dx \nonumber \\
&&\quad \approx \frac{T\hat{a}_m^1k_n}{\pi}\int^\infty_{-\infty}\frac{1}{1+t^2}\prod_{j=0}^p\frac{(1+t^2)^{q_j}}{(4t^2+c'_{i_j})^{q_j}}\prod_{j'=0}^{p'}\frac{(1+t^2)^{q'_{j'}}}{(4t^2+c'_{i'_{j'}})^{q'_{j'}}}dt \nonumber \\
&&\quad \approx \frac{T\hat{a}_m^1k_n}{\pi}\int^\infty_{-\infty}\prod_{q=1}^M\frac{1}{(4t^2+c''_qb_n^{-1})^{\beta_q}}(1+t^2)^{q_0+\cdots +q_p+q'_0+\cdots +q'_{p'}-1}dt \nonumber \\
&&\quad \approx 2iT\hat{a}_m^1k_n\sum_{q=1}^M\bigg(\frac{d}{dt}\bigg)^{\beta_q-1}\bigg(\frac{1}{(4t+2\sqrt{c''_q}ib_n^{-1/2})^{\beta_q}}\prod_{r\neq q}\frac{1}{(4t^2+c''_rb_n^{-1})^{\beta_r}}\bigg)\bigg|_{t=\sqrt{c''_q}b_n^{-1/2}i/2}. \nonumber
\end{eqnarray}
($(1+t^2)^o$はオーダーが下がるから消える．)\\
☆この収束は$b_n$のオーダーが合っているからどこに行くかはともかく収束することはわかる．そして極限が$s_{m-1}$の形で書かれて連続であることも．
}
\end{discuss}
\begin{equation*}
b_n^{-1/2}{\rm tr}(S_{m,\dagger}S_{\alpha}S_{m,\dagger}S_{\alpha'})=T\ell_n^{-1}E_{\alpha,\alpha'}(s_{m-1})+o_p(1)
\end{equation*}
for any $\alpha,\alpha'\in\mathcal{B}$.
\begin{discuss}
{\colorr 係数の$b_n$は$-(q_0+\cdots +q_p+q'_0+\cdots +q'_{p'}+1/2)$個で${\rm tr}()$のオーダーは$b_n^{q_0+\cdots +q_p+q'_0+\cdots +q'_{p'}-1/2}$なので$b_n^{-1}$が残る．}
\end{discuss}
Therefore, we obtain
\begin{eqnarray}\label{second-moment-est}
&& \\
&&\sum_{m=2}^{[\ell_nt]}E_m\bigg[\bigg(b_n^{-1/4}\sum_{k=1}^Ka_kF_{k,m}{\rm tr}(S_{\alpha_k}(\tilde{Z}_m\tilde{Z}_m^\top-\tilde{S}_{m,\dagger}))\bigg)^2\bigg] \nonumber \\
&&\quad \overset{P}\to \int^t_0\sum_{k,k'=1}^Ka_ka_{k'}F_k(s,X_s)F_{k'}(s,X_s)E_{\alpha_k,\alpha'_k}(s)ds
=\sum_{k,k'=1}^Ka_ka_{k'}[\tilde{\Upsilon}(t)]_{k,k'}, \nonumber 
\end{eqnarray}
where $[\tilde{\Upsilon}(t)]_{k,k'}=\int^t_0F_k(s,X_s)F_{k'}(s,X_s)E_{\alpha_k,\alpha_k'}(s)ds$.
The relation (\ref{second-moment-est}) shows that $\sum_{k,k'=1}^Ka_ka_{k'}[\Upsilon]_{k,k'}\geq 0$ almost surely, which implies that $\Upsilon$ is nonnegative definite.
\begin{discuss}
{\colorr $s_{m-1}\to t$は一様連続性から出す．上の収束先になっていることから$(\gamma_{\alpha,\alpha'})_{\alpha,\alpha'\in\mathcal{B}_P}$がnonnegativeになっていることが分かる．}
\end{discuss}

Similarly to the case with (\ref{deterministic-st-conv-lemma-eq2}), we can readily show 
\begin{equation}\label{forth-moment-est}
b_n^{-1}\sum_mE_m\bigg[\bigg(\sum_{k=1}^Ka_k F_{k,m}{\rm tr}(S_{\alpha_k}(\tilde{Z}_m\tilde{Z}_m^\top-\tilde{S}_{m,\dagger}))\bigg)^4\bigg]\overset{P}\to  0.
\end{equation}
\begin{discuss}
{\colorr Lemma~\ref{ZSZ-est-lemma}より(\ref{forth-moment-est})の左辺は
\begin{eqnarray}
&\leq& C_Kb_n^{-1}\sum_m\sum_ka_k^4F_{k,m}^4E_m[{\rm tr}(S_{\alpha_k}(\tilde{Z}_m\tilde{Z}_m^\top - \tilde{S}_{m,\ast}))^4] \nonumber \\
&\leq& C_Kb_n^{-1}\sum_m\sum_ka_k^4F_{k,m}^4E_m[(\tilde{Z}_m^\top S_{\alpha_k}\tilde{Z}_m)^4] \nonumber \\
&=&O_p(b_n^{-1}\ell_n((b_n^{-4}k_n^7)\vee (b_n^{-2}k_n^4)))\overset{P}\to 0. \nonumber
\end{eqnarray}
\begin{eqnarray}
\int^T_0\sqrt{\sum_{k,k'}a_ka_{k'}F_kF_{k'}E_{\alpha_k,\alpha'_k}}dW_t\overset{d}=\sqrt{\sum_{k,k'}a_ka_{k'}[\Upsilon]_{k,k'}}N
\overset{d}=(a_k)_k^\top\Upsilon^{1/2}N_K. \nonumber
\end{eqnarray}
}
\end{discuss}
Equations (\ref{second-moment-est}) and (\ref{forth-moment-est}) with Theorem 3.2 of Jacod~\cite{jac97} yield the desired result.
\end{proof}

\begin{discuss}
{\colorr {\rm Claim}. $\mathcal{X}$: Poland space, $U\subset \mathcal{X}$: openとすると，ある$S\subset \mathcal{X}$ :dense, 可算があって，$\bar{S\cap U}=\bar{U}$.
\begin{proof}
$\{U_n\}_n\subset \mathcal{X}$を開基として，$x_n\in U_n$に対して$S=\{x_n\}_n$とおく．任意の$y\in\bar{U}$と$V\in y$: openに対して$V\cap U\neq \emptyset$なので
$S\cap (V\cap U)\neq \emptyset$. よって$S\cap U$は$\bar{U}$でdense.
\end{proof}
}
\end{discuss}

\begin{lemma}\label{Hnl-conv-lemma}
Assume [B1$'$], [B2], and further assume that there exist $M\in \mathbb{N}$ and continuous functions $\{\Sigma_{i,1}(t,x)\}_{i=1}^M$ and $\{\Sigma_{i,2}(\sigma)\}_{i=1}^M$ 
such that a continuous $\partial_\sigma\Sigma_{i,2}(\sigma)$ exists on $\bar{\Lambda}$ for $1\leq i\leq M$, 
and that $\Sigma(t,x,\sigma)=\sum_{i=1}^M\Sigma_{i,1}(t,x)\Sigma_{i,2}(\sigma)$. Then,
\begin{equation*}
b_n^{-1/4}\partial_\sigma \check{H}_n(\sigma_\ast)\to^{s\mathchar`-\mathcal{L}}\Gamma_{1,\sigma_\ast}^{1/2}\mathcal{N}_\gamma.
\end{equation*}
\end{lemma}
\begin{proof}
By the assumptions, there exist $L_\alpha\in\mathbb{N}$, continuous functions $(F_{\alpha,l}(t,x))_{l=1}^{L_\alpha}$ 
and $(c'_{\alpha,l}(\sigma))_{l=1}^{L_\alpha}$ that are $C^1$ with respect to $\sigma$, such that 
\begin{equation*}
c_{\alpha,m}(\sigma)=\sum_{l=1}^{L_\alpha}F_{\alpha,l}(s_{m-1},X_{s_{m-1}})c'_{\alpha,l}(\sigma).
\end{equation*}
Together with (\ref{delH-eq}) and (\ref{invS-expansion}), we have
\begin{eqnarray}
\partial_\sigma \check{H}_n(\sigma_\ast)
&=&-\frac{1}{2}\sum_m{\rm tr}(\partial_\sigma \tilde{S}_m^{-1}(\sigma_\ast)(\tilde{Z}_m\tilde{Z}_m^\top - \tilde{S}_{m,\dagger}))+o_p(b_n^{\frac{1}{4}}) \nonumber \\
&=&-\frac{1}{2}\sum_{\alpha\in\mathcal{B}} \sum_{l=1}^{L_\alpha}\partial_\sigma c'_{\alpha,l}(\sigma_\ast)\sum_mF_{\alpha,l,m}{\rm tr}(S_\alpha(\tilde{Z}_m\tilde{Z}_m^\top -\tilde{S}_{m,\dagger}))+o_p(b_n^{\frac{1}{4}}), \nonumber
\end{eqnarray}
where $F_{\alpha,l,m}=F_{\alpha,l}(s_{m-1},X_{s_{m-1}})$.
\begin{discuss}
{\colorr 微分しても絶対収束していることはわかるから無限和も微分できるか．もしくは下の証明からもいえる．}
\end{discuss}

Let $|\alpha|=p\vee q_0\vee \cdots \vee q_p$, $\mathcal{B}_0=\emptyset$ and $\mathcal{B}_P=\{\alpha\in\mathcal{B}:|\alpha|\leq P\}\setminus\{(0,1)\}$ for $P\in\mathbb{N}$.
Let $\mathcal{S}_{P_1,P_2}^{(j)}(\sigma)=\sum_{\alpha\in\mathcal{B}_{P_2}\setminus \mathcal{B}_{P_1}}\partial_\sigma^jc_{\alpha,m}(\sigma)S_\alpha$,
$c_\alpha(t,\sigma)=\sum_{l=1}^{L_\alpha}F_{\alpha,l}(t,X_t)c'_{\alpha,l}(\sigma)$ and 
\begin{equation*}
{\bf V}_{P_1,P_2}^{(j,k)}(\sigma)=\sum_{\alpha,\alpha'\in\mathcal{B}_{P_2}\setminus \mathcal{B}_{P_1}}\int^T_0F_{\alpha,\alpha'}(t)\partial_\sigma^j c_\alpha(t,\sigma)\partial_\sigma^k c_\alpha(t,\sigma)dt 
\end{equation*}
for $1\leq j,k\leq 2$, $P_1,P_2\in\mathbb{Z}_+$ and $P_2>P_1$.
By Proposition~\ref{random-param-convergence} and Lemmas~\ref{deterministic-st-conv-lemma} and ~\ref{st-conv-lemma2}, it is sufficient to show that for any $\epsilon,\delta>0$ there exists some $P\in\mathbb{N}$ such that
\begin{discuss}
{\colorr \begin{equation*}
{\bf Z}_{n,k}=b_n^{-1/4}\sum_mF_{\alpha,l}(s_{m-1},X_{s_{m-1}}){\rm tr}(S_\alpha(\tilde{Z}_m\tilde{Z}_m^\top-\tilde{S}_{m,\dagger})).
\end{equation*}
}
\end{discuss}
\begin{equation}\label{tail-est1}
P\bigg[\sup_{\sigma\in\Lambda}\bigg|b_n^{-\frac{1}{4}}\sum_m{\rm tr}(\mathcal{S}_{P,P'}^{(1)}(\sigma)(\tilde{Z}_m\tilde{Z}_m^\top-\tilde{S}_{m,\dagger}))\bigg|> \delta\bigg]<\epsilon
\end{equation}
and 
\begin{equation}\label{tail-est2}
P\Big[\sup_{\sigma\in\Lambda}|{\bf V}_{0,P'}^{(1,1)}(\sigma)-{\bf V}_{0,P}^{(1,1)}(\sigma)|> \delta\Big]<\epsilon
\end{equation}
for $P'>P$.

First we show that (\ref{tail-est1}) holds.
Let $\mathfrak{M}_a=aD'_m+M_{m,\ast}$ for $a>0$. Lemma A.7 in~\cite{ogi18} yields
\begin{eqnarray}\label{Salpha-norm-est}
&& \\
&&\lVert \mathfrak{M}_1^{1/2}S_\alpha\mathfrak{M}_1^{1/2}\rVert \nonumber \\
&&\quad \leq \lVert D'^{-1/2}_m\mathfrak{M}_RS_\alpha\mathfrak{M}_RD'^{-1/2}_m\rVert
\lVert \mathfrak{M}_1^{1/2}\mathfrak{M}_R^{-1}D'_m\mathfrak{M}_R^{-1}\mathfrak{M}_1^{1/2}\rVert. \nonumber 
\end{eqnarray}
\begin{discuss}
{\colorr Lemma~\ref{ZSZ-est-lemma}の条件をチェックするには$\lVert \mathfrak{M}_1S_\alpha\mathfrak{M}_1\rVert$を評価したいが，
${\rm tr}((\mathfrak{S}\tilde{S}_{m,\ast}^2)$の評価でこれを使うと$r_n/\underbar{r}_n$のロスが発生するからダメ．}
\end{discuss}
Since 
\begin{eqnarray}
\mathfrak{M}_R^{1/2}D'^{-1}_m\mathfrak{M}_R^{1/2}\geq \mathfrak{M}_R^{1/2}(\mathfrak{M}_R/R)^{-1}\mathfrak{M}_R^{1/2}=R\mathcal{E}_{\mathcal{K}_m^\gamma}, \nonumber
\end{eqnarray}
we have
\begin{eqnarray}\label{Salpha-norm-est2}
\mathfrak{M}_1^{-1/2}\mathfrak{M}_RD'^{-1}_m\mathfrak{M}_R\mathfrak{M}_1^{-1/2}
\geq R\mathfrak{M}_1^{-1/2}\mathfrak{M}_R\mathfrak{M}_1^{-1/2}\geq R\mathcal{E}_{\mathcal{K}_m^\gamma}.
\end{eqnarray}
Equation (\ref{Salpha-norm-est}) and (\ref{Salpha-norm-est2}) with Lemma 2 of~\cite{ogi-yos14} and the inequality 
$\lVert (RD'_m+M_{m,\ast})^{-1/2}D'_m(RD'_m+M_{m,\ast})^{-1/2}\rVert\leq R^{-1}$ yield
\begin{equation}\label{Salpha-norm-est3}
\lVert \mathfrak{M}_1^{1/2}S_\alpha\mathfrak{M}_1^{1/2}\rVert\leq R^{-1}\times R^{-q_0-\cdots -q_p+2}=R^{-q_0-\cdots -q_p+1}.
\end{equation}
\begin{discuss}
{\colorr (\ref{Salpha-norm-est})の右辺第一因子は$\lVert D'^{-1/2}_m\mathfrak{S}D'^{-1/2}_m\rVert$と$\lVert (RD'_m+M_{m,\ast})^{-1/2}D'_m(RD'_m+M_{m,\ast})^{-1/2}\rVert$
のいくつかの積で抑えられる．}
\end{discuss}
Similarly, Lemmas A.7 and 5.2 of~\cite{ogi18} yield
\begin{eqnarray}\label{Salpha-tr-est}
{\rm tr}(\mathfrak{M}_1S_\alpha)&\leq& R^{-q_0-\cdots -q_p+2}{\rm tr}(\mathfrak{M}_1^{1/2}\mathfrak{M}_R^{-1}D'_m\mathfrak{M}_R^{-1}\mathfrak{M}_1^{1/2}) \\
&\leq& R^{-q_0-\cdots -q_p+2}\lVert \mathfrak{M}_R^{-1/2}\mathfrak{M}_1\mathfrak{M}_R^{-1/2}\rVert {\rm tr}(\mathfrak{M}_R^{-1}D'_m) \nonumber \\
&\leq& R^{-q_0-\cdots -q_p}O_p(b_n^{1/2}\ell_n^{-1}). \nonumber
\end{eqnarray}

Moreover, Lemma~\ref{invS-eq-lemma} yields
\begin{eqnarray}\label{calpha-sum-est}
&&\sum_{P=P_1+1}^\infty \sum_{\alpha\in\mathcal{B}_P\setminus\mathcal{B}_{P-1}}\frac{|\partial_\sigma^jc_{\alpha,m}(\sigma)|}{R^{q_0+\cdots +q_p}} \\
&&\quad \leq \gamma^2\bigg(\sum_{P=P_1+1}^\infty {\bf a}_P
+\bigg(1-\frac{\epsilon}{R}\bigg)^{P_1}\sum_{P=1}^\infty{\bf a}_P\bigg)
=:\mathcal{C}_{P_1} \nonumber
\end{eqnarray}
for $j\leq 2$, where ${\bf a}_P=P!((P-2)!)^{-1}(1-1/R)^P$.
\begin{discuss}
{\colorr $\alpha\in\mathcal{B}\setminus \mathcal{B}_{P-1}$ならば$p\geq P$ or $\max_jq_j\geq P$. $p\geq P$の評価はOK.
$p\leq P-1, q_j\geq P$ならば$(1-\epsilon/R)^P$がかかるから$P\to\infty$のとき$\to 0$. Lemma~\ref{invS-eq-lemma}の最後の級数評価も使う．
微分も$[\tilde{\Sigma}_m]_{ij}$にかかるのは問題ない．$(R-[\tilde{\Sigma}_m]_{ii})^{q_j}$にかかったら$[\tilde{\Sigma}_m]_{ii}^{-1}$がつくが$\epsilon^{-1}$で評価できる．}
\end{discuss}
Then, since (\ref{Salpha-norm-est3}), (\ref{Salpha-tr-est}) and (\ref{calpha-sum-est}) yield
\begin{eqnarray}
{\rm tr}((\mathcal{S}_{P_1,P_2}^{(j)}\tilde{S}_{m,\dagger})^2)
&\leq& C\mathcal{C}_{P_1}\sum_{\alpha\in\mathcal{B}_{P_2}\setminus\mathcal{B}_{P_1}}|\partial_\sigma^jc_{\alpha,m}(\sigma)|{\rm tr}(\mathfrak{M}_1^{-1}\tilde{S}_{m,\dagger}S_\alpha\tilde{S}_{m,\dagger}) \nonumber \\
&\leq& C_{P_1}^2O_p(b_n^{1/2}\ell_n^{-1}), \nonumber
\end{eqnarray}
for $P'>P$ and $q>4$, Remark~\ref{ZSZ-est-rem2} yields
\begin{eqnarray}\label{trSm-est}
&&\sup_{\sigma\in\Lambda}\sum_{j=1}^2E_\Pi\bigg[\bigg|b_n^{-\frac{1}{4}}\sum_m{\rm tr}(\mathcal{S}_{P,P'}^{(j)}(\sigma)(\tilde{Z}_m\tilde{Z}_m^\top-\tilde{S}_{m,\dagger}))\bigg|^q\bigg] \\
&&\quad \leq C_q\sup_{\sigma\in\Lambda}\sum_{j=1}^2E_\Pi\bigg[\bigg(b_n^{-\frac{1}{2}}\sum_m{\rm tr}((\mathcal{S}_{P,P'}^{(j)}(\sigma)\tilde{S}_{m,\dagger})^2)\bigg)^{q/2}\bigg] \nonumber \\
&&\quad \quad +\mathcal{C}_P^q\bar{R}_n(b_n^{-\frac{q}{4}}\ell_n^{\frac{q}{2}}+b_n^{-q}k_n^{\frac{7}{4}q}) \nonumber \\
&&\quad \leq C_q\mathcal{C}_P^q\to 0 \nonumber
\end{eqnarray}
as $P\to \infty$.

Therefore, Sobolev's inequality yields
\begin{equation*}
\sup_nE_\Pi\bigg[\sup_{\sigma\in\Lambda}\bigg|b_n^{-\frac{1}{4}}\sum_m{\rm tr}(\mathcal{S}_{P,P'}^{(1)}(\sigma)(\tilde{Z}_m\tilde{Z}_m^\top-\tilde{S}_{m,\dagger}))\bigg|^q\bigg] \to 0
\end{equation*}
as $P\to \infty$ for sufficiently large $q$, and this implies (\ref{tail-est1}).

We show next that (\ref{tail-est2}) holds.
Similarly to (\ref{trSm-est}), we obtain
\begin{equation}\label{trSm-est2}
\sup_{P\in\mathbb{N}}\sup_n\sup_{\sigma\in\Lambda}\sum_{j=1}^2E_\Pi\bigg[\bigg|b_n^{-\frac{1}{4}}\sum_m{\rm tr}(\mathcal{S}_{0,P}^{(j)}(\tilde{Z}_m\tilde{Z}_m^\top-\tilde{S}_{m,\dagger}))\bigg|^q\bigg]<\infty.
\end{equation}

Lemma~\ref{st-conv-lemma2} yields
\begin{equation}\label{trSm-conv}
b_n^{-\frac{1}{4}}\sum_{\alpha\in\mathcal{B}_{P_2}\setminus\mathcal{B}_{P_1}}\sum_m\partial_\sigma^jc_{\alpha,m}(\sigma){\rm tr}(S_\alpha(\tilde{Z}_m\tilde{Z}_m^\top-\tilde{S}_{m,\dagger}))
\to^{s\mathchar`-\mathcal{L}}\sqrt{{\bf V}_{P_1,P_2}^{j,j}(\sigma)}\mathcal{N}_1.
\end{equation}
Then, Fatou's lemma, Skorohod's representation theorem, (\ref{trSm-conv}), and (\ref{trSm-est}) together yield
\begin{equation}\label{tail-est3}
E\big[\big|{\bf V}_{P_1,P_2}^{j,j}(\sigma)\big|^{q/2}\big]\leq C_qE\Big[\Big|\sqrt{{\bf V}_{P_1,P_2}^{j,j}(\sigma)}\mathcal{N}_1\Big|^q\Big]\leq C_q\mathcal{C}_{P_1}^q
\end{equation}
for $P_2>P_1$, $1\leq j\leq 2$, $q>4$ and $\sigma\in\Lambda$. Similarly, (\ref{trSm-est2}) yields
\begin{equation}\label{tail-est4}
\max_{1\leq j\leq 2}\sup_{P_1\in\mathbb{N},\sigma\in\Lambda}E\big[\big|{\bf V}_{0,P_1}^{j,j}(\sigma)\big|^{q/2}\big]<\infty.
\end{equation}
\begin{discuss}
{\colorr $A_n$を$r_n,k_n$等のindicatorで$1_{A_n}\overset{P}\to 1$となるものとすると，$X_n1_{A_n}\to^d X$.
Fatou+Skorohodで$E[X]\leq \liminf_{n\to\infty}E[X_n1_{A_n}]<\infty$.
$a_t^j$のモーメント条件もB2にあるからOK
}
\end{discuss}

Furthermore, the Cauchy--Schwarz inequality yields
\begin{eqnarray}\label{tail-est5}
&& \\
|{\bf V}_{0,P_2}^{(1,1)}-{\bf V}_{0,P_1}^{(1,1)}|&\leq& |{\bf V}_{P_1,P_2}^{(1,1)}|+2\bigg|\sum_{\substack{\alpha\in\mathcal{B}_{P_1} \\\alpha'\in\mathcal{B}_{P_2}\setminus \mathcal{B}_{P_1}}}
\int^T_0F_{\alpha,\alpha'}(t)\partial_\sigma c_\alpha(t,\sigma)\partial_\sigma c_{\alpha'}(t,\sigma)dt\bigg| \nonumber \\
&\leq &|{\bf V}_{P_1,P_2}^{(1,1)}|+2\sqrt{{\bf V}_{0,P_1}^{(1,1)}}\sqrt{{\bf V}_{P_1,P_2}^{(1,1)}}. \nonumber
\end{eqnarray}
\begin{discuss}
{\colorr $\int (s a_t+b_t)^\top A_t(s a_t+b_t)dt\geq 0$より
\begin{equation*}
\bigg|\int a_t^\top A_tb_tdt\bigg|\leq \sqrt{\int a_t^\top A_ta_tdt}\sqrt{\int b_t^\top A_t b_tdt}.
\end{equation*}
}
\end{discuss}
Equations (\ref{tail-est3}), (\ref{tail-est4}) and (\ref{tail-est5}) yield $\sup_\sigma E[|{\bf V}_{0,P_2}^{(1,1)}-{\bf V}_{0,P_1}^{(1,1)}|^q]\to 0$ as $P_1\to \infty$.
By using $|{\bf V}_{P_1,P_2}^{(1,2)}|\leq \sqrt{{\bf V}_{P_1,P_2}^{(1,1)}}\sqrt{{\bf V}_{P_1,P_2}^{(2,2)}}$ in a similar way, we have
\begin{equation*}
\sup_\sigma E[|\partial_\sigma ({\bf V}_{0,P_2}^{(1,1)}-{\bf V}_{0,P_1}^{(1,1)})|^q]\to 0
\end{equation*}
as $P_1\to \infty$, letting Sobolev's inequality yield $E[\sup_\sigma |{\bf V}_{0,P_2}^{(1,1)}-{\bf V}_{0,P_1}^{(1,1)}|^q]\to 0$ as $P_1\to \infty$
for sufficiently large $q$.
This implies (\ref{tail-est2}).

\begin{discuss}
{\colorr 分布収束から概収束の定理：skorohod representation theorem. 可分な位相空間で成り立つ．

あとはこれで$|\alpha|>p$のところを一様評価できれば良い．左辺で$p$が大きいところはLemma~\ref{invS-eq-lemma}とソボレフの不等式から一様評価が可能．
$q_j$が大きいところは$D_m^{-1}$の展開(\ref{invD-exp})で$q_j\leq p$の部分を$D'^{-1}_m$, $q_j> p$の部分を$D''^{-1}_m$と書くと
$D''^{-1}_mD_m^{-1}\cdots D_m^{-1}+D'^{-1}_mD''^{-1}_mD_m^{-1}\cdots D_m^{-1}+\cdots + D'^{-1}_m\cdots D'^{-1}_mD''^{-1}_m$と分解されるので
ソボレフと$\lVert D_m^{-1/2}\tilde{L}D_m^{-1/2}\rVert^p\leq \rho^p$の評価を使えば
\begin{equation*}
b_n^{-1/4}\sum_{p\leq P}p\rho^{p-1}\sum_{p'=P+1}^\infty(R-(\Sigma_m)_{jj})^{p'}{\rm tr}(D_m^{-1})
\end{equation*}
で抑えられるので，$P$が大きい時一様評価可能（今はソボレフが使えるように$k_n$のオーダーを決めているので問題ない）．右辺の残差項の評価は，$\Gamma'_P$の定義から各$(\alpha,l,\alpha',l')$に対して
\begin{equation*}
b_n^{-1/2}\partial_\sigma c''_{\alpha,l}\partial_\sigma c''_{\alpha',l'}
\sum_mc'_{\alpha,m,l}c'_{\alpha',m,l'}E_m[{\rm tr}(S_\alpha (\tilde{Z}_m\tilde{Z}_m^\top -\tilde{S}_{m,\dagger})S_{\alpha'} (\tilde{Z}_m\tilde{Z}_m^\top -\tilde{S}_{m,\dagger}))]
\end{equation*}
が$\partial_\sigma c''_{\alpha,l}\partial_\sigma c''_{\alpha',l'}(\Gamma)_{\alpha,l,\alpha',l'}$に確率収束することとこの式の残差項が同様に一様評価できることとFatou's lemmaを使えば良い．
（恒等式で戻すことができる）無限和だが有限和ごとのFatou's lemmaを使った後に足し合わせればいい．}
\end{discuss}

\begin{discuss}
{\colorr ちゃんと$\omega$と$\sigma$の関数になっていることが大事か．他の変数が入ってきて$\sigma$を動かしたときとか$\dot{\sigma}$を代入する時二勝手に変わると困る．
$\Gamma$の$\sigma$への依存は$\Sigma_m(\sigma)$を通してだけで$\sigma$を変えたり$\dot{\sigma}$を入れる時に他の変数を止めていられるので問題ない．
$\sigma$に関する連続性は必要で$\sum\sum \partial_\sigma c''\partial_\sigma c''\Gamma'$の方はOK. $\Gamma$もdefからOK.}
\end{discuss}
\end{proof}

\begin{proposition}\label{st-conv-prop}
Assume [B1$'$], [B2], [V], and [C]. Then, 
\begin{equation*}
b_n^{-1/4}\partial_\sigma \check{H}_n(\sigma_\ast)\to^{s\mathchar`-\mathcal{L}} \Gamma_{1,\sigma_\ast}^{1/2}\mathcal{N}_\gamma.
\end{equation*}
\end{proposition}
\begin{proof}
\begin{discuss}
{\colorr $\partial_\sigma\tilde{S}_m^{-1}(\sigma_\ast)$のrandomnessをどうコントロールするかがカギ}
\end{discuss}
Let $(\Sigma_{i,l,1}(t,x))_{1\leq i\leq I,l\in\mathbb{N}}$, $(\Sigma_{i,l,2}(\sigma))_{1\leq i\leq I,l\in\mathbb{N}}$ and $(\Sigma_l(t,x,\sigma))_{l\in\mathbb{N}}$ 
be continuous, symmetric, matrix-valued functions such that $\partial_\sigma \Sigma_{i,l,2}$ exists and is continuous, let
$[\Sigma_l(t,x,\sigma)]_{jk}=\sum_{i=1}^I[\Sigma_{i,l,1}(t,x)]_{jk}[\Sigma_{i,l,2}(\sigma)]_{jk}$, 
and assume each $\Sigma_l(t,x,\sigma)$ satisfies the conditions of $\Sigma(t,x,\sigma)$ in [B1$'$],
\begin{discuss}
{\colorr $\partial_x\Sigma_{i,l,1}$が存在してLipschitz continuous, etc.}
\end{discuss}
and $\partial_\sigma^j \Sigma_l\to \partial_\sigma^j \Sigma$ 
as $l\to\infty$ uniformly in any compact set for $j\in\{0,1,2\}$.
\begin{discuss}
{\colorr $\Sigma$はcontinuousだからコンパクト空間上でStone Wierstrassの定理から近似できる．りぷしっつ連続の積はコンパクト空間上ではりぷしっつ連続．
コンパクト列で$[0,T]\times \mathbb{R}^{\gamma_X}\times \bar{\Lambda}$
に収束するものを取れば近似できる．
(\ref{fast-est-proof-eq1})を満たすことが重要で，$\Sigma_l$を局所化すればコンパクトで十分近似しているからOK.
$\partial_\sigma^2\Sigma$を近似する関数を作ってその積分でdefするか．コンパクトの増大列で任意のコンパクトをやがて含むもの$K_l$をとって
$\sup_{K_l}|\partial_\sigma^2\Sigma_l-\partial_\sigma^2 \Sigma|<1/l$となるようにする．
{\colorg リバイズ時：微分まで近似するのはStone-Wierstrassだと厳しいか．$\mathbb{R}^d$のopenだから$\mathbb{R}^d$に特化した方法とか
探してちゃんとやる．最悪$\Lambda$を弧状連結とかにすれば$\partial_\sigma^2\Sigma$を近似しておいて積分で$\Sigma$をdefできる}
}
\end{discuss}
Let $\check{H}_{n,l}$, $\tilde{S}_{m,l}$, $\tilde{\Sigma}_{m,l}$ and $\Gamma_{1,\sigma,l}$ 
be obtained similarly to $\check{H}_n$, $\tilde{S}_m$, $\tilde{\Sigma}_m$, and $\Gamma_{1,\sigma}$, respectively,
replacing $\Sigma(t,x,\sigma)$ with $\Sigma_l(t,x,\sigma)$.
Let $\sigma_{l,\ast}$ be a random variable that minimizes $D((\Sigma_l(t,X_t,\sigma))_{t\in[0,T]},\Sigma_\dagger)$ for $l\in\mathbb{N}$.

Then, (\ref{delH-eq}) yields
\begin{eqnarray}
&&b_n^{-1/4}\partial_\sigma \check{H}_n(\sigma_\ast)-b_n^{-1/4}\partial_\sigma \check{H}_{n,l}(\sigma_{l,\ast}) \nonumber \\
&&\quad =-\frac{1}{2}b_n^{-1/4}\sum_m{\rm tr}((\partial_\sigma \tilde{S}_m^{-1}(\sigma_\ast)-\partial_\sigma\tilde{S}_{m,l}^{-1}(\sigma_{l,\ast}))(\tilde{Z}_m\tilde{Z}_m^\top-\tilde{S}_{m,\dagger}))+o_p(1). \nonumber
\end{eqnarray}
Moreover, an argument similar to that for (\ref{trSm-est2}) and Sobolev's inequality yield
\begin{eqnarray}\label{Sm-Sml-diff}
&& \\
&&E_\Pi\bigg[\sup_{\substack{\sigma\in\Lambda \\ |u|\leq 1}}\bigg|b_n^{-\frac{1}{4}}\sum_m{\rm tr}((\partial_\sigma \tilde{S}_m^{-1}(\sigma)-\partial_\sigma \tilde{S}_{m,l}^{-1}(\sigma+\eta u))(\tilde{Z}_m\tilde{Z}_m^\top -\tilde{S}_{m,\dagger}))\bigg|^q\bigg] \nonumber \\
&&\quad \leq C\sup_{\substack{\sigma\in\Lambda \\ |u|\leq 1}} \sum_{0\leq j,k\leq 1,j+k\leq 1}
E_\Pi\bigg[\bigg|b_n^{-\frac{1}{4}}\sum_m{\rm tr}(\partial_\sigma^{1+j}\partial_u^k (\tilde{S}_m^{-1}(\sigma)-\tilde{S}_{m,l}^{-1}(\sigma+\eta u)) \nonumber \\
&&\hspace{5cm} \times(\tilde{Z}_m\tilde{Z}_m^\top -\tilde{S}_{m,\dagger}))\bigg|^q\bigg] \nonumber \\
&&\quad \leq O_p(1)\times \max_{m,\sigma\in\Lambda,|u|\leq 1}
(|\partial_\sigma\tilde{\Sigma}_m(\sigma)-\partial_\sigma \tilde{\Sigma}_{m,l}(\sigma+\eta u)|^q \nonumber \\
&&\hspace{4cm}+\eta|\partial_\sigma \tilde{\Sigma}_{m,l}(\sigma+\eta u)|^q) \nonumber \\
&&\quad \overset{P}\to 0 \nonumber
\end{eqnarray}
as $l\to\infty$ and $\eta\to 0$ for sufficiently large $q>0$.
\begin{discuss}
{\colorr Remark~\ref{ZSZ-est-rem2}と$\partial_\sigma(\tilde{S}_m^{-1}(\sigma)-\tilde{S}_{m,l}^{-1}(\sigma+u))$のノルム評価を使えばよい．}
\end{discuss}
\begin{discuss}
{\colorr 
\begin{equation*}
\partial_\sigma \tilde{S}_m^{-1}(\sigma_\ast)=-\tilde{S}_m^{-1}\partial_\sigma \tilde{S}_m\tilde{S}_m^{-1}
\end{equation*}
\begin{equation*}
\partial_\sigma \tilde{S}_m^{-1}-\partial_\sigma \tilde{S}_{m,l}^{-1}=-\tilde{S}_m^{-1}\partial_\sigma \tilde{S}_m\tilde{S}_m^{-1}+\tilde{S}_{m,l}^{-1}\partial_\sigma \tilde{S}_{m,l}\tilde{S}_{m,l}^{-1}
\end{equation*}
\begin{equation*}
\partial_\sigma^j \tilde{S}_m-\partial_\sigma^j \tilde{S}_{m,l}=([\partial_\sigma^j\tilde{\Sigma}_m-\partial_\sigma^j\tilde{\Sigma}_{m,l}]_{ij}D'_{ij,m})_{ij}
\end{equation*}
}
\end{discuss}

Moreover, since $\sigma_\ast$ is the unique point that minimizes $D(\Sigma(\sigma),\Sigma_\dagger)$
and $\Sigma_l\to \Sigma$ uniformly in any compact set, we obtain
\begin{equation}\label{sigma_l-conv}
\sigma_{l,\ast}\to\sigma_\ast
\end{equation}
as $l\to \infty$ almost surely.

Thanks to (\ref{Sm-Sml-diff}) and (\ref{sigma_l-conv}),
for any $\epsilon,\delta>0$, there exist $L,N\in\mathbb{N}$ such that
\begin{eqnarray}
&&P\Big[|\exp(iub_n^{-1/4}\partial_\sigma \check{H}_n(\sigma_\ast))-\exp(iub_n^{-1/4}\partial_\sigma \check{H}_{n,l}(\sigma_{l,\ast}))|\geq \delta\Big]<\epsilon, \nonumber \\
&&P[|\exp(iu\Gamma_{1,\sigma_\ast}^{1/2}\mathcal{N}_\gamma)-\exp(iu\Gamma_{1,\sigma_{l,\ast},l}^{1/2}\mathcal{N}_\gamma)|\geq \delta]<\epsilon \nonumber
\end{eqnarray}
for $l\geq L$ and $n\geq N$.

Therefore, Lemma \ref{Hnl-conv-lemma} implies that there exists some $N'\in\mathbb{N}$ such that
\begin{eqnarray}
&&|E[\exp(iub_n^{-1/4}\partial_\sigma \check{H}_n(\sigma_\ast))e^{ivF}]-E[\exp(iu\Gamma_{1,\sigma\ast}^{1/2}\mathcal{N}_\gamma)e^{ivF}]| \nonumber \\
&&\quad \leq |E[\exp(iub_n^{-1/4}\partial_\sigma \check{H}_{n,L}(\sigma_{L,\ast}))e^{ivF}]-E[\exp(iu\Gamma_{1,\sigma_{L,\ast},L}^{1/2}\mathcal{N}_\gamma)e^{ivF}]| \nonumber \\
&&\quad \quad +2\epsilon+2\delta \nonumber \\
&&\quad <3\epsilon+2\delta \nonumber
\end{eqnarray}
for any $u\in\mathbb{R}^\gamma$, $v\in\mathbb{R}$, $n\geq N'$, and an $\mathcal{F}$-measurable random variable $F$.

\end{proof}

\noindent
{\bf Proof of Theorem~\ref{mixed-normality-thm}.}
By localization techniques, we may additionally assume [B1$'$]. Under [C], we can show 
\begin{equation}\label{Gamma-conv}
b_n^{-1/2}\partial_\sigma^2\check{H}_n(\xi_n)=-\partial_\sigma^2D(\Sigma(\xi_n),\Sigma_\dagger)+o_p(1)\overset{P}\to \Gamma_2,
\end{equation}
similarly to (\ref{Hn-limit-eq}).
\begin{discuss}
{\colorr (\ref{delH-eq})と同様にLemma~\ref{log-likelihood-eq-lemma}を微分できる．$\xi_n\overset{P}\to \sigma_\ast$も使う．$\partial_\sigma^2\Sigma$のLipschitz contiも}
\end{discuss}
Then, Proposition~\ref{st-conv-prop} and the discussions in Section~\ref{asym-mixed-normality-section} complete the proof.

\qed

\begin{small}
\bibliographystyle{abbrv}
\bibliography{referenceLibrary_mathsci,referenceLibrary_other}
\end{small}


\renewcommand{\thesection}{\Alph{section}}
\setcounter{section}{0}
\section{Appendix}

\subsection{Results from linear algebra}

For a symmetric matrix $A$, let $\{\lambda_j(A)\}_j$ be eigenvalues of $A$ and let $U(A)$ be an orthogonal matrix satisfying $U(A){\rm diag}((\lambda_j(A))_j)U(A)^\top=A$.

\begin{lemma}\label{inv-residue-lemma}
For symmetric, positive definite matrices $C_1$ and $C_2$ of the same size,
\begin{equation*}
\int^\infty_{-\infty}(C_1t^2+C_2)^{-1}dt=\pi C_1^{-1/2}(C_1^{-1/2}C_2C_1^{-1/2})^{-1/2}C_1^{-1/2}.
\end{equation*}
\end{lemma}
\begin{proof}
Let $A=C_1^{-1/2}C_2C_1^{-1/2}$. Then we have
\begin{eqnarray}
\int^\infty_{-\infty}(C_1t^2+C_2)^{-1}dt
&=&C_1^{-1/2}\int^\infty_{-\infty}(t^2+C_1^{-1/2}C_2C_1^{-1/2})^{-1}dtC_1^{-1/2} \nonumber \\
&=&C_1^{-1/2}U(A) \int^\infty_{-\infty}{\rm diag}((t^2+\lambda_j(A))_j^{-1})dt U(A)^\top C_1^{-1/2} \nonumber \\
&=&\pi C_1^{-1/2}U(A) {\rm diag}((\lambda_j(A)^{-1/2})_j) U(A)^\top C_1^{-1/2} \nonumber \\
&=&\pi C_1^{-1/2}(C_1^{-1/2}C_2C_1^{-1/2})^{-1/2}C_1^{-1/2}. \nonumber
\end{eqnarray}
\end{proof}

\begin{lemma}\label{log-residue-lemma}
For $l\times l$ symmetric, positive definite matrices $C_1$ and $C_2$,
\begin{eqnarray}
&&\int_{-\infty}^\infty \frac{1}{1+t^2}\log \det (C_1t^2+C_2)dt \nonumber \\
&&\quad =\pi \log\det C_1+2\pi \log\det (\mathcal{E}_l+(C_1^{-1/2}C_2C_1^{-1/2})^{1/2}). \nonumber
\end{eqnarray}
\end{lemma}
\begin{proof}
Let $A=C_1^{-1/2}C_2C_1^{-1/2}$. Then we have
\begin{eqnarray}
&&\int_{-\infty}^\infty \frac{1}{1+t^2}\log \det (C_1t^2+C_2)dt \nonumber \\
&&\quad =\pi \log\det C_1+\sum_{j=1}^n\int_{-\infty}^\infty \frac{1}{1+t^2}\log(t^2+\lambda_j(A))dt \nonumber \\
&&\quad =\pi \log\det C_1+\sum_{j=1}^n2\pi \log(1+\sqrt{\lambda_j(A)}) \nonumber \\
&&\quad =\pi \log\det C_1+2\pi \log\det (\mathcal{E}_l+A^{1/2}). \nonumber
\end{eqnarray}
\begin{discuss}
{\colorr logの積分公式はOgihara~\cite{ogi18}のLemma A.9の証明にある．}
\end{discuss}
\end{proof}

\begin{lemma}\label{rootA-diff-lemma}
Let $A$ and $B$ be $l\times l$ symmetric, nonnegative definite matrices such that $A+B$ is invertible. Then
\begin{equation*}
\lVert A-B\rVert \leq \sqrt{l}\lVert (A+B)^{-1}\rVert \lVert A^2-B^2\rVert.
\end{equation*}
Moreover, if further $A$ and $B$ are individually invertible, then
\begin{equation*}
\lVert A^{-1}-B^{-1}\rVert 
\leq \sqrt{l}\lVert A^2-B^2\rVert \lVert A^{-1}\rVert \lVert B^{-1}\rVert  \lVert (A+B)^{-1}\rVert.
\end{equation*}
\end{lemma}

\begin{proof}
Since
\begin{equation*}
(A+B)(A-B)+(A-B)(A+B)=2A^2-2B^2,
\end{equation*}
we obtain 
\begin{equation*}
U \Lambda U^\top(A-B)+(A-B)U \Lambda U^\top=2A^2-2B^2.
\end{equation*}
\begin{equation*}
\Lambda U^\top(A-B)U +U^\top(A-B)U \Lambda =2U^\top(A^2-B^2)U.
\end{equation*}
\begin{equation*}
[U^\top(A-B)U]_{ij}(\lambda_i(A+B)+\lambda_j(A+B))=[2U^\top(A^2-B^2)U]_{ij},
\end{equation*}
where $U=U(A+B)$ and $\Lambda={\rm diag}((\lambda_j(A+B))_j)$.
Therefore, we have
\begin{eqnarray}
|U^\top(A-B)U|^2&=&4\sum_{i,j}\frac{[U^\top(A^2-B^2)U]_{ij}^2}{(\lambda_i+\lambda_j)^2} \nonumber \\
&\leq &4\sup_{i,j}\frac{1}{(\lambda_i+\lambda_j)^2}|U^\top(A^2-B^2)U|^2 \nonumber \\
&\leq & \lVert (A+B)^{-1}\rVert^2|U^\top(A^2-B^2)U|^2, \nonumber
\end{eqnarray}
where $\lambda_i=\lambda_i(A+B)$. 
\begin{discuss}
{\colorr $(\lambda_i+\lambda_j)^{-2}\leq \lVert (A+B)^{-1}\rVert^2$を出すためには$A,B$: nonnegativeが必要．}
\end{discuss}
Then we have
\begin{equation*}
\lVert A-B\rVert =\lVert U^\top(A-B)U \rVert \leq \sqrt{l}\lVert (A+B)^{-1}\rVert \lVert A^2-B^2\rVert.
\end{equation*}
If $A$ and $B$ are invertible, $A^{-1}-B^{-1}=B^{-1}(B-A)A^{-1}$ yields the desired result.
\end{proof}

\begin{discuss}
{\colorr $\lVert A\rVert^2=\sup_{|x|=1}|Ax|^2=\sup_{|x|=1}\sum_i(\sum_jA_{ij}x_j)^2\leq \sup_{|x|=1}\sum_i\sum_jA_{ij}^2\sum_jx_j^2=|A|^2$.
また，$x_j=A_{ij}/\sqrt{\sum_jA_{ij}^2}$とすれば$\lVert A\rVert^2\geq \sum_jA_{ij}^2$. よって$|A|^2\leq l\lVert A\rVert^2$.}
\end{discuss}
\begin{discuss}
{\colorg リバイズ時このLemmaどこかに載っていないか探す.}
{\colorr Rでこの不等式が成り立ちそうなことをチェック済．$\sqrt{l}$がなくても成り立ちそう}
\end{discuss}

\begin{lemma}\label{residueCalc-lemma}
Let $x_1,x_2,x_3,x_4>0$. Denote $F^j_n=\sum_{1\leq i_1<i_2<\cdots <i_j\leq n}\prod_{l=1}^jx_{i_l}$. Then 
\begin{equation*}
\frac{1}{\pi}\int^\infty_{-\infty}\frac{dt}{(t^2+x_1^2)(t^2+x_2^2)}=\frac{1}{F_2^1F^2_2}, 
\end{equation*}
\begin{equation*}
\frac{1}{\pi}\int^\infty_{-\infty}\frac{dt}{(t^2+x_1^2)(t^2+x_2^2)(t^2+x_3^2)}=\frac{F_3^1}{F^3_3\prod_{1\leq i<j\leq 3}(x_i+x_j)},
\end{equation*}
\begin{equation*}
\frac{1}{\pi}\int^\infty_{-\infty}\frac{dt}{(t^2+x_1^2)(t^2+x_2^2)(t^2+x_3^2)(t^2+x_4^2)}=\frac{F_4^1F^2_4-F^3_4}{F^4_4\prod_{1\leq i<j\leq 4}(x_i+x_j)}.
\end{equation*}
\end{lemma}
\begin{proof}
If $x_1=x_2=x_3=x_4$, then we can easily calculate 
\begin{equation*}
\frac{1}{\pi}\int^\infty_{-\infty}\frac{1}{(t^2+x_1^2)^4}dt=\frac{2i}{3!}\bigg(\frac{d}{dt}\bigg)^3\bigg(\frac{1}{(t+x_1i)^4}\bigg)\bigg|_{t=x_1i}=\frac{5}{16}x_1^{-7}
\end{equation*}
by the residue theorem.
If $x_1\neq x_2$, then 
\begin{eqnarray}
&&\frac{1+t^2}{(t^2+x_1^2)(t^2+x_2^2)(t^2+x_3^2)(t^2+x_4^2)} \nonumber \\
&&\quad =\frac{(1+t^2)(x_2^2-x_1^2)^{-1}}{(t^2+x_1^2)(t^2+x_3^2)(t^2+x_4^2)}-\frac{(1+t^2)(x_2^2-x_1^2)^{-1}}{(t^2+x_2^2)(t^2+x_3^2)(t^2+x_4^2)}. \nonumber
\end{eqnarray}
We can thus inductively calculate the integrals. We can obtain the results for the other cases in a similar fashion.

\begin{discuss}
{\colorr 
まず
\begin{equation*}
\frac{1}{\pi}\int^\infty_{-\infty}\frac{dt}{(t^2+x_1^2)(t^2+x_2^2)}=\frac{1}{x_1x_2F_2^1}, 
\quad \frac{1}{\pi}\int^\infty_{-\infty}\frac{dt}{(t^2+x_1^2)(t^2+x_2^2)(t^2+x_3^2)}=\frac{F_3^1}{\prod_{i=1}^3x_i\prod_{1\leq i<j\leq 3}(x_i+x_j)},
\end{equation*}
\begin{equation*}
\frac{1}{\pi}\int^\infty_{-\infty}\frac{dt}{(t^2+x_1^2)(t^2+x_2^2)(t^2+x_3^2)(t^2+x_4^2)}=\frac{F_4^1F^2_4-F^3_4}{\prod_{i=1}^4x_i\prod_{1\leq i<j\leq 4}(x_i+x_j)}.
\end{equation*}
を示す．

$x_3\neq x_4$の時
\begin{eqnarray}
\frac{1}{\pi}\int^\infty_{-\infty}\frac{dt}{\prod_{j=1}^4(t^2+x_j^2)}
&=&\frac{(x_1+x_2+x_3)(x_4^2-x_3^2)^{-1}}{x_1x_2x_3(x_1+x_2)(x_1+x_3)(x_2+x_3)}-\frac{(x_1+x_2+x_4)(x_4^2-x_3^2)^{-1}}{x_1x_2x_4(x_1+x_2)(x_1+x_4)(x_2+x_4)} \nonumber \\
&=&\frac{x_4(x_1+x_2+x_3)(x_1+x_4)(x_2+x_4)-x_3(x_1+x_2+x_4)(x_1+x_3)(x_2+x_3)}{x_1x_2x_3x_4(x_1+x_2)(x_1+x_3)(x_2+x_3)(x_1+x_4)(x_2+x_4)(x_4^2-x_3^2)}. \nonumber
\end{eqnarray}
右辺の分子は
\begin{eqnarray}
&&x_4^3x_3-x_3^3x_4+(x_1+x_2)x_4^3-(x_1+x_2)x_3^3+x_4^2x_3(x_1+x_2)-x_3^2x_4(x_1+x_2)+x_4^2(x_1+x_2)^2-x_3^2(x_1+x_2)^2 \nonumber \\
&&\quad \quad +x_4x_1x_2(x_1+x_2)-x_3x_1x_2(x_1+x_2) \nonumber \\
&&\quad = (x_4-x_3)\big\{x_3x_4(x_3+x_4)+(x_1+x_2)(x_4^2+x_3x_4+x_3^2)+x_3x_4(x_1+x_2)+(x_3+x_4)(x_1+x_2)^2+x_1x_2(x_1+x_2)\big\} \nonumber \\
&&\quad =(x_4-x_3)\big\{x_3x_4(x_3+x_4)+(x_1+x_2)(x_4+x_3)^2+(x_3+x_4)(x_1+x_2)^2+x_1x_2(x_1+x_2)\big\} \nonumber \\
&&\quad =(x_4-x_3)\bigg\{F_4^1\big\{(x_1+x_2)(x_3+x_4)+x_3x_4+x_1x_2\big\}-F_4^3\bigg\}
=(x_4-x_3)(F_4^1F_4^2-F_4^3) \nonumber
\end{eqnarray}
よりOK.}
\end{discuss}
\end{proof}

\begin{lemma}\label{varphi-lemma}
Let $n\in\mathbb{N}$. Let $A$ and $B$ be matrices with size $n\times n$. 
Assume that $B$ is symmetric, positive definite. Then there exists a unique matrix $C$ such that
\begin{equation}\label{varphi-identity}
B^{1/2}C+CB^{1/2}=A.
\end{equation}
Let $\varphi_B(A)=C$. Then $\varphi_B$ is a linear operator on a space of $n\times n$ matrices equipped with the operator norm,
$\varphi_B$ is invertible on ${\rm Im}\varphi_B$,
\begin{equation}
\lVert \varphi_B\rVert \leq (1/2)\sqrt{n}\lVert B^{-1}\rVert^{1/2}, \quad \lVert \varphi_B^{-1}\rVert \leq 2\sqrt{n}\lVert B\rVert^{1/2},
\end{equation} 
and for a symmetric, positive definite matrix $B'$,
\begin{equation*}
\lVert \varphi_B(A)-\varphi_{B'}(A)\rVert\leq 2n\sqrt{n} \lVert B'\rVert^{1/2} \lVert B'^{-1}\rVert \lVert A\rVert \lVert B-B'\rVert.
\end{equation*}
\end{lemma}

\begin{proof}
Let $U=U(B)$ and $\Lambda={\rm diag}((\lambda_j(B))_j)$.
If a matrix $C$ satisfies (\ref{varphi-identity}), we obtain
\begin{equation*}
\Lambda^{1/2}U^\top CU+U^\top CU\Lambda^{1/2}=U^\top AU,
\end{equation*}
and hence
\begin{equation}\label{varphi-identity2}
[U^\top CU]_{ij}=\frac{[U^\top AU]_{ij}}{\lambda_i(B)^{1/2}+\lambda_j(B)^{1/2}}.
\end{equation}
Therefore, $C$ exists and is unique.

We can easily check that $\varphi_B$ is linear. We also have
\begin{eqnarray}
\lVert \varphi_B(A)\rVert ^2&=&\lVert U^\top \varphi_B(A)U\rVert^2=\sup_{|x|=1}\sum_i\bigg(\sum_j[U^\top\varphi_B(A)U]_{ij}x_j\bigg)^2 \nonumber \\
&=&\sup_{|x|=1}\sum_i\bigg(\sum_j\frac{[U^\top AU]_{ij}}{\lambda_i^{1/2}+\lambda_j^{1/2}}x_j\bigg)^2 \nonumber \\
&\leq& \frac{\lVert B^{-1}\rVert}{4}\sup_{|x|=1}\sum_i\bigg(\sum_j|[U^\top AU]_{ij}||x_j|\bigg)^2 \nonumber \\
&\leq& \frac{\lVert B^{-1}\rVert}{4}\sup_{|x|=1}\sum_i\sum_j[U^\top AU]_{ij}^2
\leq \frac{\lVert B^{-1}\rVert}{4}n\lVert U^\top AU\rVert^2. \nonumber
\end{eqnarray}
Then, we have $\lVert \varphi_B\rVert \leq (1/2)\sqrt{n}\lVert B^{-1}\rVert^{1/2}$.

Moreover, if $C=\varphi_B(A)$, we have
\begin{eqnarray}
\lVert A\rVert^2&=&\lVert U^\top AU\rVert^2
=\sup_{|x|=1}\sum_i\bigg(\sum_j[U^\top CU]_{ij}(\lambda_i^{1/2}+\lambda_j^{1/2})x_j\bigg)^2 \nonumber \\
&\leq &4\lVert B\rVert n\lVert U^\top CU\rVert^2. \nonumber 
\end{eqnarray}
Then, $\varphi_B$ is invertible and $\lVert \varphi_B^{-1}\rVert \leq 2\sqrt{n}\lVert B\rVert^{1/2}$.

Furthermore, we have
\begin{eqnarray}
0&=&\varphi_B(A)B^{1/2}-\varphi_{B'}(A)B'^{1/2}+B^{1/2}\varphi_B(A)-B'^{1/2}\varphi_{B'}(A) \nonumber \\
&=&\varphi_B(A)(B^{1/2}-B'^{1/2})+(B^{1/2}-B'^{1/2})\varphi_B(A)+\varphi_{B'}(\varphi_B(A)-\varphi_{B'}(A)). \nonumber
\end{eqnarray}
Therefore, Lemma~\ref{rootA-diff-lemma} yields
\begin{eqnarray}
&&\lVert \varphi_B(A)-\varphi_{B'}(A)\rVert \nonumber \\
&&\quad \leq \lVert \varphi_{B'}^{-1}\rVert \lVert \varphi_{B'}(\varphi_B(A)-\varphi_{B'}(A))\rVert \nonumber \\
&&\quad \leq 4\sqrt{n}\lVert B'\rVert^{1/2} \lVert \varphi_{B'}(A)\rVert \lVert B^{1/2}-B'^{1/2}\rVert \nonumber \\
&&\quad \leq 2n \lVert B'\rVert^{1/2} \lVert B'^{-1}\rVert^{1/2} \lVert A\rVert \times \sqrt{n}\lVert (B^{1/2}+B'^{1/2})^{-1}\rVert\lVert B-B'\rVert. \nonumber
\end{eqnarray}
\end{proof}

\begin{lemma}\label{G4-integral-lemma}
Let $B$ be a symmetric, positive definite matrix, and let $A$ and $C$ be symmetric matrices, all of the same size. Then
\begin{equation*}
\int^\infty_{-\infty}{\rm tr}((A(t^2+B)^{-1})^2)dt=2\pi {\rm tr}(B^{-\frac{1}{2}}\varphi_B(A)^2),
\end{equation*}
\begin{equation*}
\int^\infty_{-\infty}{\rm tr}((A(t^2+B)^{-1})^2C(t^2+B)^{-1})dt=\pi \mathfrak{K}_1(B,A,C),
\end{equation*}
\begin{equation*}
\int^\infty_{-\infty}{\rm tr}((A(t^2+B)^{-1}C(t^2+B)^{-1})^2)dt=2\pi \mathfrak{K}_2(B,A,C).
\end{equation*}
\end{lemma}

\begin{proof}
We show only the last equation.
Let $U=U(B)$ and $\lambda_j=\lambda_j(B^{1/2})$.
Since $(U^\top\varphi_B(M)U)_{ij}=(U^\top MU)_{ij}/(\lambda_i+\lambda_j)$ for a matrix $M$, Lemma~\ref{residueCalc-lemma} yields
\begin{eqnarray}
&&\int^\infty_{-\infty}{\rm tr}((A(t^2+B)^{-1}C(t^2+B)^{-1})^2)dt \nonumber \\
&&\quad =\int^\infty_{-\infty}\sum_{i(1),\cdots,i(4)}\prod_{j=1}^2\frac{[U^\top AU]_{i(j),i(j+1)}[U^\top AU]_{i(j+1),i(j+2)}}{(t^2+\lambda_{i(j)}^2)(t^2+\lambda_{i(j+1)}^2)}dt \nonumber \\
&&\quad =\pi \sum_{i(1),\cdots,i(4)}\frac{(F_1F_2-F_3)\prod_{j=1}^2([U^\top AU]_{i(j),i(j+1)}[U^\top CU]_{i(j+1),i(j+2)})}{F_4\prod_{1\leq j<k\leq 4}(\lambda_{i(j)}+\lambda_{i(k)})}  \nonumber \\
&&\quad =\pi \sum_{i(1),\cdots,i(4)}\frac{(F_1F_2-F_3)}{F_4(\lambda_{i(1)}+\lambda_{i(3)})(\lambda_{i(2)}+\lambda_{i(4)})} \nonumber \\
&&\quad \quad \quad \quad \times \prod_{j=1}^2([U^\top \varphi_B(A)U]_{i(j),i(j+1)}[U^\top \varphi_B(C)U]_{i(j+1),i(j+2)}), \nonumber 
\end{eqnarray}
where $F_j=\sum_{1\leq k_1<k_2<\cdots <k_j\leq 4}\prod_{l=1}^j\lambda_{i(k_l)}$ and we regard $i(5)=i(1)$. 

Since $F_3=\lambda_{i(1)}\lambda_{i(3)}(\lambda_{i(2)}+\lambda_{i(4)})+\lambda_{i(2)}\lambda_{i(4)}(\lambda_{i(1)}+\lambda_{i(3)})$ 
and $F_2=(\lambda_{i(1)}+\lambda_{i(3)})(\lambda_{i(2)}+\lambda_{i(4)})+\lambda_{i(1)}\lambda_{i(3)}+\lambda_{i(2)}\lambda_{i(4)}$,
we obtain
\begin{eqnarray}
&&\frac{(F_1F_2-F_3)}{F_4(\lambda_{i(1)}+\lambda_{i(3)})(\lambda_{i(2)}+\lambda_{i(4)})} \nonumber \\
&&\quad = \frac{F_2}{F_4(\lambda_{i(1)}+\lambda_{i(3)})}+\frac{F_2}{F_4(\lambda_{i(2)}+\lambda_{i(4)})}
-\frac{\lambda_{i(1)}\lambda_{i(3)}}{F_4(\lambda_{i(1)}+\lambda_{i(3)})}
-\frac{\lambda_{i(2)}\lambda_{i(4)}}{F_4(\lambda_{i(2)}+\lambda_{i(4)})} \nonumber \\
&&\quad = \frac{F_2-\lambda_{i(1)}\lambda_{i(3)}}{F_4(\lambda_{i(1)}+\lambda_{i(3)})}
+\frac{F_2-\lambda_{i(2)}\lambda_{i(4)}}{F_4(\lambda_{i(2)}+\lambda_{i(4)})} \nonumber \\
&&\quad = \frac{1}{\lambda_{i(1)}\lambda_{i(3)}\lambda_{i(4)}}+\frac{1}{\lambda_{i(1)}\lambda_{i(2)}\lambda_{i(3)}}
+\frac{1}{\lambda_{i(1)}\lambda_{i(3)}(\lambda_{i(1)}+\lambda_{i(3)})} \nonumber \\
&&\quad \quad +\frac{1}{\lambda_{i(1)}\lambda_{i(2)}\lambda_{i(4)}}+\frac{1}{\lambda_{i(2)}\lambda_{i(3)}\lambda_{i(4)}}
+\frac{1}{\lambda_{i(2)}\lambda_{i(4)}(\lambda_{i(2)}+\lambda_{i(4)})}. \nonumber
\end{eqnarray}
Therefore, we have
\begin{eqnarray}
&&\frac{1}{\pi}\int^\infty_{-\infty}{\rm tr}((A(t^2+B)^{-1}C(t^2+B)^{-1})^2)dt \nonumber \\
&&\quad =2{\rm tr}(B^{-1/2}\varphi_B(A)\varphi_B(C)B^{-1/2}\varphi_B(A)B^{-1/2}\varphi_B(C)) \nonumber \\
&&\quad \quad +2{\rm tr}(\varphi_B(A)B^{-1/2}\varphi_B(C)B^{-1/2}\varphi_B(A)B^{-1/2}\varphi_B(C)) \nonumber \\
&&\quad \quad +2{\rm tr}(\varphi_B(\varphi_B(A)\varphi_B(C))B^{-1/2}\varphi_B(A)\varphi_B(C)B^{-1/2}) \nonumber \\
&&\quad =2\mathfrak{K}_2(B,A,C). \nonumber
\end{eqnarray}
\begin{discuss}
{\colorr 
\begin{eqnarray}
&&\frac{1}{\pi}\int^\infty_{-\infty}{\rm tr}((A(t^2+B)^{-1})^2)dt=\frac{1}{\pi}\int^\infty_{-\infty}
\sum_{i,j}\frac{(U^\top AU)_{ij}(U^\top A U)_{ji}}{(t^2+\lambda_i^2)(t^2+\lambda_j^2)}dt \nonumber \\
&&\quad = \sum_{i,j}\frac{(U^\top AU)_{ij}(U^\top A U)_{ji}}{\lambda_i\lambda_j(\lambda_i+\lambda_j)}
=2{\rm tr}(AB^{-1/2}\varphi_B(A)B^{-1/2})={\rm tr}(B^{-1/2}(\varphi_B(A)B^{-1/2})^2). \nonumber
\end{eqnarray}
\begin{eqnarray}
&&\frac{1}{\pi}\int^\infty_{-\infty}{\rm tr}((A(t^2+B)^{-1})^2C(t^2+B)^{-1})dt \nonumber \\
&&\quad =\frac{1}{\pi}\int^\infty_{-\infty} \sum_{i_1,i_2,i_3}\frac{(U^\top AU)_{i_1,i_2}(U^\top A U)_{i_2,i_3}(U^\top C U)_{i_3,i_1}}{(t^2+\lambda_{i_1}^2)(t^2+\lambda_{i_2}^2)(t^2+\lambda_{i_3}^2)}dt \nonumber \\
&&\quad =\sum_{i_1,i_2,i_3}\frac{F^1_3(U^\top AU)_{i_1,i_2}(U^\top A U)_{i_2,i_3}(U^\top C U)_{i_3,i_1}}{F^3_3(\lambda_{i_1}+\lambda_{i_2})(\lambda_{i_1}+\lambda_{i_3})(\lambda_{i_2}+\lambda_{i_3})} \nonumber \\
&&\quad = {\rm tr}(\varphi_B(A)B^{-1/2}\varphi_B(A)B^{-1/2}\varphi_B(C)+B^{-1/2}\varphi_B(A)^2B^{-1/2}\varphi_B(C)+(B^{-1/2}\varphi_B(A))^2\varphi_B(C)) \nonumber \\
&&\quad = \mathfrak{K}_1(B,A,C). \nonumber
\end{eqnarray}
}
\end{discuss}
\end{proof}

\subsection{Some auxiliary lemmas}

\begin{lemma}\label{st-conv-lemma}
Let $(\Omega,\mathcal{F},P)$ be a probability space. Let $\mathcal{G}$ be a $\sigma$-subfield of $\mathcal{F}$, 
let ${\bf X}_1,{\bf X}_2$ be continuous random fields 
on a separable metric space $\mathcal{X}$, and let ${\bf Y}$ be a $\mathcal{G}$-measurable random variable.
Assume that $({\bf X}_1(x),{\bf U})\overset{d}=({\bf X}_2(x),{\bf U})$ for any $x\in\mathcal{X}$ and $\mathcal{G}$-measurable random variable ${\bf U}$.
Then, $({\bf X}_1({\bf Y}),{\bf U})\overset{d}=({\bf X}_2({\bf Y}),{\bf U})$ for any bounded $\mathcal{G}$-measurable random variable ${\bf U}$.
\end{lemma}
\begin{discuss}
{\colorr $U$:多次元に対しても$\langle U,u\rangle$は一次元なので成立．}
\end{discuss}

\begin{proof}
Let $d$ be the metric of $\mathcal{X}$. Since $\mathcal{X}$ is a separable metric space, there exists a countable set $\mathcal{X}_0=(y_k)_k\subset \mathcal{X}$ and a Borel function 
$F_n:\mathcal{X}\to \mathcal{X}_0$ such that $d(F_n(y),y)<1/n \ (y\in \mathcal{X})$ for $n\in\mathbb{N}$.
\begin{discuss}
{\colorr キョリ空間では可分なら第二可算公理がなりたつ．$\mathcal{X}$は可算個の$\epsilon$ボールで覆われるので順番を付けて差をとることでdisjointにすればいい．}
\end{discuss}
By the assumption, $E[e^{iu{\bf X}_1(x)}{\bf U}]=E[e^{iu{\bf X}_2(x)}{\bf U}]$ for any $x$ and bounded $\mathcal{G}$-measurable random variable ${\bf U}$.
Then, for any $u=(u_1,u_2)$, we obtain
\begin{eqnarray}
&&E[\exp(iu({\bf X}_1(F_n({\bf Y})),{\bf U}))]
=\sum_k E[e^{iu_2{\bf U}}e^{iu_1{\bf X}_1(y_k)}1_{F_n^{-1}(y_k)}({\bf Y})] \nonumber \\
&&\quad =\sum_k E[e^{iu({\bf X}_2(y_k),{\bf U})}1_{F_n^{-1}(y_k)}({\bf Y})] 
=E[\exp(iu({\bf X}_2(F_n({\bf Y})),{\bf U}))]. \nonumber
\end{eqnarray}
\begin{discuss}
{\colorr StConvStudyのTheorem 1から$g(X_1(e_k),{\bf U})$もstable convergence}
\end{discuss}
By letting $n\to\infty$ in the above equation, the continuity of ${\bf X}_1$ and of ${\bf X}_2$ yields 
$({\bf X}_1({\bf Y}),{\bf U})\overset{d}=({\bf X}_2({\bf Y}),{\bf U})$.
\end{proof}

\subsection{Proof of results in Section~\ref{misspe-theory-section}}\label{Section3-proof-section}

$ $

\noindent
{\bf Proof of (\ref{L2-equivalence}).}

(\ref{D-def}) and Lemma~A.1 in~\cite{ogi18} yield
\begin{eqnarray}
D(\Sigma(\sigma),\Sigma_\dagger)&\leq &\frac{1}{4}\int^T_0\lVert \mathcal{D}_{t,\sigma}^{-1/2}\rVert {\rm tr}((\mathcal{D}_{t,\sigma}^{1/2}-\mathcal{D}_{t,\dagger}^{1/2})^2)dt \nonumber \\
&\leq &\frac{1}{4}\sup_t\lVert \Sigma_t^{-1}(\sigma)\rVert^{1/2}\bigg(\inf_{j,t}a^j_t\bigg)^{-1/2}
\int^T_0{\rm tr}((\mathcal{D}_{t,\sigma}^{1/2}-\mathcal{D}_{t,\dagger}^{1/2})^2)dt. \nonumber
\end{eqnarray}
Moreover,
\begin{equation*}
2(\mathcal{D}-\mathcal{D}_\dagger)=(\mathcal{D}^{1/2}-\mathcal{D}_\dagger^{1/2})(\mathcal{D}^{1/2}+\mathcal{D}_\dagger^{1/2})
+(\mathcal{D}^{1/2}+\mathcal{D}_\dagger^{1/2})(\mathcal{D}^{1/2}-\mathcal{D}_\dagger^{1/2})
\end{equation*}
implies
\begin{eqnarray}
&&4{\rm tr}((\mathcal{D}-\mathcal{D}_\dagger)^2) \nonumber \\
&&\quad ={\rm tr}\Big(\big\{(\mathcal{D}^{1/2}-\mathcal{D}_\dagger^{1/2})(\mathcal{D}^{1/2}+\mathcal{D}_\dagger^{1/2})
+(\mathcal{D}^{1/2}+\mathcal{D}_\dagger^{1/2})(\mathcal{D}^{1/2}-\mathcal{D}_\dagger^{1/2})\big\}^2\Big) \nonumber \\
&&\quad =2{\rm tr}\Big(\big\{(\mathcal{D}^{1/2}-\mathcal{D}_\dagger^{1/2})(\mathcal{D}^{1/2}+\mathcal{D}_\dagger^{1/2})\big\}^2\Big) \nonumber \\
&&\quad \quad +2{\rm tr}((\mathcal{D}^{1/2}-\mathcal{D}_\dagger^{1/2})^2(\mathcal{D}^{1/2}+\mathcal{D}_\dagger^{1/2})^2) \nonumber \\
&&\quad \geq 4\lVert (\mathcal{D}^{1/2}+\mathcal{D}_\dagger^{1/2})^{-1}\rVert^{-2} {\rm tr}((\mathcal{D}^{1/2}-\mathcal{D}_\dagger^{1/2})^2). \nonumber
\end{eqnarray}
\begin{discuss}
{\colorr $A+B$が正定値の時，
\begin{equation*}
{\rm tr}((A-B)^2)={\rm tr}((A+B)(A-B)^2(A+B)(A+B)^{-2})\leq {\rm tr}((A-B)^2(A+B)^2)\lVert (A+B)^{-1}\rVert^2.
\end{equation*}
\begin{eqnarray}
{\rm tr}((A-B)^2)&=&{\rm tr}((A+B)^{1/2}(A-B)^2(A+B)^{1/2}(A+B)^{-1}) \nonumber \\
&\leq& {\rm tr}((A-B)(A+B)(A-B))\lVert (A+B)^{-1}\rVert \nonumber \\
&\leq& {\rm tr}(\{(A-B)(A+B)\}^2)\lVert (A+B)^{-1}\rVert^2. \nonumber
\end{eqnarray}
}
\end{discuss}
Then, thanks to Lemma~A.1 of~\cite{ogi18} and the inequality
\begin{equation*}
\lVert (\mathcal{D}_{t,\sigma}^{1/2}+\mathcal{D}_{t,\dagger}^{1/2})^{-1}\rVert \leq (\inf_{j,t}a^j_t)^{-1/2}\sup_t\lVert \Sigma_t^{-1}(\sigma)\rVert^{1/2},
\end{equation*}
there exists a constant $C_2$ such that
\begin{equation*}
D(\Sigma(\sigma),\Sigma_\dagger)\leq C_2\int^T_0|\Sigma_t(\sigma)-\Sigma_{t,\dagger}|^2dt.
\end{equation*}

To show that $C_1\int^T_0|\Sigma_t(\sigma)-\Sigma_{t,\dagger}|^2dt\leq D(\Sigma(\sigma),\Sigma_\dagger)$,
we similarly obtain
\begin{equation*}
D(\Sigma(\sigma),\Sigma_\dagger)\geq \frac{1}{4}\int^T_0{\rm tr}((\mathcal{D}_{t,\sigma}^{1/2}-\mathcal{D}_{t,\dagger}^{1/2})^2)\lVert \mathcal{D}_{t,\sigma}^{1/2}\rVert^{-1}dt
\end{equation*}
and
\begin{equation*}
{\rm tr}((\mathcal{D}_{t,\sigma}-\mathcal{D}_{t,\dagger})^2)
\leq \lVert \mathcal{D}_{t,\sigma}^{1/2}+\mathcal{D}_{t,\dagger}^{1/2}\rVert^2{\rm tr}((\mathcal{D}_{t,\sigma}^{1/2}-\mathcal{D}_{t,\dagger}^{1/2})^2).
\end{equation*}
Then the estimates for $\lVert \mathcal{D}_{t,\sigma}^{1/2}\rVert$ and $\lVert \mathcal{D}_{t,\sigma}^{1/2}+\mathcal{D}_{t,\dagger}^{1/2}\rVert$ yield the desired result.

\qed

$ $

\noindent
{\bf Calculation of (\ref{bias-limit}).}
Let ${\bf M}_{m,\sigma}=b_n^{-1}\tilde{\Sigma}_m(\sigma)\mathcal{E}_{k^1_m}+v_\ast M_m$, ${\bf M}_{m,\ast}={\bf M}_{m,\sigma_\ast}$, ${\bf M}_{m,\dagger}=b_n^{-1}\tilde{\Sigma}_{m,\dagger}\mathcal{E}_{k^1_m}+v_\ast M_m$
${\bf F}_{m,\sigma}(x)=b_n^{-1}\tilde{\Sigma}_m+2v_\ast(1-\cos x)$ and ${\bf F}_{m,\ast}={\bf F}_{m,\sigma_\ast}$.
Then, by the method given in the discussions of the proofs of Lemma~\ref{Phi-est-lemma} and~\ref{Psi-est-lemma},
\begin{eqnarray}
&&\Gamma_2 (b_n^{1/4}(\hat{\sigma}_n-\sigma_\ast)) \nonumber \\
&&\quad \approx -\frac{b_n^{-1/4}}{2}\sum_m\partial_\sigma G_m(\tilde{a}_{s_{m-1}},\tilde{\Sigma}_m(\sigma),\tilde{\Sigma}_{m,\dagger},v_\ast)|_{\sigma=\sigma_\ast} \nonumber \\
&&\quad \approx -\frac{b_n^{-1/4}}{2}\sum_m\partial_\sigma\bigg\{ {\rm tr}({\bf M}_{m,\sigma}^{-1}{\bf M}_{m,\dagger})
-\frac{k^1_m+1}{\pi b_n}\int^\pi_0\frac{(\tilde{\Sigma}_m-\tilde{\Sigma}_{m,\dagger})}{{\bf F}_{m,\sigma}(x)}dx \nonumber \\
&&\quad \quad \quad \quad \quad \quad \quad  +\log\det {\bf M}_{m,\sigma}-\frac{k_m^1+1}{\pi}\int^\pi_0\log(v_\ast^{-1}{\bf F}_{m,\sigma}(x))dx\bigg\}\bigg|_{\sigma=\sigma_\ast} \nonumber \\
&&\quad = -\frac{b_n^{-1/4}}{2}\sum_m\bigg\{ -\frac{\partial_\sigma \tilde{\Sigma}_{m,\ast}}{b_n}{\rm tr}({\bf M}_{m,\ast}^{-2}{\bf M}_{m,\dagger})
+b_n^{-1}\partial_\sigma \tilde{\Sigma}_{m,\ast}{\rm tr}({\bf M}_{m,\ast}^{-1}) \nonumber \\
&&\quad \quad \quad \quad \quad \quad \quad \quad  +\frac{(k_m^1+1)(\tilde{\Sigma}_{m,\dagger}-\tilde{\Sigma}_{m,\ast})}{\pi b_n}
\int^\pi_0\frac{b_n^{-1}\partial_\sigma \tilde{\Sigma}_{m,\ast}}{{\bf F}_{m,\ast}^2(x)}dx\bigg\}. \nonumber
\end{eqnarray}
Therefore, we have
\begin{eqnarray}
&&\Gamma_2 (b_n^{1/4}(\hat{\sigma}_n-\sigma_\ast)) \nonumber \\
&&\quad \approx \frac{b_n^{-9/4}}{2}\sum_m(\tilde{\Sigma}_{m,\dagger}-\tilde{\Sigma}_{m,\ast})\partial_\sigma \tilde{\Sigma}_{m,\ast}
\bigg\{ {\rm tr}({\bf M}_{m,\ast}^{-2})-\int^\pi_0\frac{(k_m^1+1)}{\pi{\bf F}_{m,\ast}^2(x)}dx\bigg\} \nonumber \\
&&\quad \approx \frac{b_n^{-9/4}}{2}\sum_m(\tilde{\Sigma}_{m,\dagger}-\tilde{\Sigma}_{m,\ast})\partial_\sigma \tilde{\Sigma}_{m,\ast} \nonumber \\
&&\quad \quad \times\sum_{k=1}^{k_m}\bigg(\frac{1}{{\bf F}_{m,\ast}^2(t_k)}-\int^{t_k}_{t_{k-1}}\frac{(t_k-t_{k-1})^{-1}}{{\bf F}_{m,\ast}^2(x)}dx\bigg) \nonumber \\
&&\quad =\frac{b_n^{-9/4}}{2}\sum_m(\tilde{\Sigma}_{m,\dagger}-\tilde{\Sigma}_{m,\ast})\partial_\sigma \tilde{\Sigma}_{m,\ast}
\sum_{k=1}^{k_m}\int^{t_k}_{t_{k-1}}\int^{t_k}_x\frac{-4v_\ast\sin y}{(t_k-t_{k-1}){\bf F}_{m,\ast}^3(y)}dydx. \nonumber
\end{eqnarray}
\begin{discuss}
{\colorr 
\begin{equation*}
{\rm tr}(\mathcal{D}(s_{m-1},\Sigma(\sigma))^{1/2}\approx \frac{k^1_m+1}{\pi}\int^\pi_0\log(\tilde{\Sigma}_{m,\ast}b_n^{-1}v_\ast^{-1}+2(1-\cos x))dx
\end{equation*}
の導出はより直接的にやる．}
\end{discuss}

\noindent
{\bf Proof of Lemma~\ref{uniqueness-lemma}.}
For any $\epsilon,\delta>0$, the assumption and continuity of a map $\sigma\mapsto D(\Sigma(\sigma),\Sigma_\dagger)$,
there exists some $\eta>0$ such that
\begin{equation*}
P\bigg[\min_{|\sigma-\sigma_\ast|\geq \delta}D(\Sigma(\sigma),\Sigma_\dagger)\geq \min_{\sigma}D(\Sigma(\sigma),\Sigma_\dagger)+\eta\bigg]\geq 1-\frac{\epsilon}{2}.
\end{equation*}
Moreover, thanks to Theorem~\ref{optConvTheorem}, there exists $N\in\mathbb{N}$ such that
\begin{equation*}
P\bigg[D(\Sigma(\check{\sigma}_n),\Sigma_\dagger)<\min_{\sigma}D(\Sigma(\sigma),\Sigma_\dagger)+\eta\bigg]\geq 1-\frac{\epsilon}{2}
\end{equation*}
for $n\geq N$.
Therefore, we obtain
\begin{eqnarray}
P[|\check{\sigma}_n-\sigma_\ast|\geq \delta]
&\leq& P\bigg[\min_{|\sigma-\sigma_\ast|\geq \delta}D(\Sigma(\sigma),\Sigma_\dagger)< \min_{\sigma}D(\Sigma(\sigma),\Sigma_\dagger)+\eta\bigg] \nonumber \\
&&+P\bigg[D(\Sigma(\check{\sigma}_n),\Sigma_\dagger)\geq \min_{\sigma}D(\Sigma(\sigma),\Sigma_\dagger)+\eta\bigg] \nonumber \\
&<&\epsilon. \nonumber
\end{eqnarray}

\qed

\noindent
{\bf Proof of Proposition~\ref{random-param-convergence}.}
Let $\mathcal{V}(x)=P\mathchar`-\lim_{M\to\infty}\mathcal{V}_M(x)$.
\begin{discuss}
{\colorr 極限$\mathcal{V}$の存在は下のLemmaから．${\bf Z}_n(x)$の存在も同じ. ${\bf Z}_n(x)$の連続性は仮定している．}
\end{discuss}
By the assumptions, we have
\begin{equation*}
{\bf Z}_n({\bf Y})=\sum_{m=1}^\infty\sum_{k=K_{m-1}+1}^{K_m}{\bf Z}_{n,k}f_k({\bf Y}) \to^{{\bf G}\mathchar`-\mathcal{L}}\sqrt{\mathcal{V}({\bf Y})}\mathcal{N}_1
\end{equation*}
and
\begin{equation*}
{\bf Z}_n(x)=\sum_{m=1}^\infty\sum_{k=K_{m-1}+1}^{K_m} {\bf Z}_{n,k}f_k(x)\to^{{\bf G}\mathchar`-\mathcal{L}}\sqrt{\mathcal{V}(x)}\mathcal{N}_1
\end{equation*}
for any $x\in\mathcal{X}$.
\begin{discuss}
{\colorr $\mathcal{N}$と$V_{k,k'}$が独立でないとこの表現が出ないから$V_{k,k'}$は$({\bf A},{\bf G})$の確率変数であることを仮定する．
下は(\ref{Zf-tail-est})と3からすぐ．上は$Y$のtightnessから$K$:cptで$P[Y\not\in K]<\epsilon$となるものをがとれ，
$K$は有限個の$U_x$でおおわれるので，$P[\sup_{x'\in K}\sum_{m=M}^\infty|\sum_k \cdots |>\delta]<\epsilon$よりOK.
（これを言うのに完備可分が必要）}
\end{discuss}
Moreover, by the assumptions and Proposition 1 of Aldous and Eagleson~\cite{ald-eag78}, we have
\begin{equation*}
({\bf Z}(x),{\bf U})\overset{d}= (\sqrt{\mathcal{V}(x)}\mathcal{N}_1,{\bf U})
\end{equation*}
for any $\mathbb{R}^k$-valued random variable ${\bf U}$ on $({\bf A}, {\bf G})$. 
\begin{discuss}
{\colorr ${\bf X}_n\to^{s\mathchar`-\mathcal{L}}{\bf X}$ならば$({\bf X}_n,U)\overset{d}\to ({\bf X},U)$}
\end{discuss}

(\ref{Zf-tail-est}) yields continuity of $\mathcal{V}(x)$ with respect to $x$ almost surely.
\begin{discuss2}
{\colorg リバイズの時証明チェックする}
\end{discuss2}
\begin{discuss}
{\colorr 
まず任意の$U_{x'}$で連続になり，第二可算公理から可算個の$U_{x'}$でおおわれる．

ある$\{M_{l,n}\}_{l,n}$: increasing on $n$があって
\begin{equation*}
P\bigg[\sup_{x'}\sum_{m=M_{l,n}+1}^\infty \bigg|\sum_{k=K_{m-1}+1}^{K_m}V_kf_k(x')\bigg|\geq \frac{1}{l}\bigg]\leq \frac{1}{n2^l}
\end{equation*}
for any $l$とできるので
\begin{equation*}
P\bigg[\mbox{ある$l$があって}\sup_{x'}\sum_{m=M_{l,n}+1}^\infty \bigg|\sum_{k=K_{m-1}+1}^{K_m}V_kf_k(x')\bigg|\geq \frac{1}{l}\bigg]\leq \frac{1}{n}.
\end{equation*}
ゆえに
\begin{equation*}
P\bigg[\cup_{n=1}^\infty\bigg\{\sup_{x'}\sum_{m=M_{l,n}+1}^\infty\bigg|\sum_{k=K_{m-1}+1}^{K_m}V_kf_k(x')\bigg|<\frac{1}{l} \ {\rm for \ any} \ l\bigg\}\bigg]=1.
\end{equation*}
$\omega\in \{\sup_{x'}\sum_{m=M_{l,n}+1}^\infty|\sum_{k=K_{m-1}+1}^{K_m} V_kf_k(x')|<\frac{1}{l} \ {\rm for \ any} \ l\}$の時，
任意の$\epsilon>0$に対してある$l$があって$\sup_{x'}\sum_{m=M_{l,n}+1}^\infty|\sum_{k=K_{m-1}+1}^{K_m}V_kf_k(x')|<\epsilon/4$で，
ある$\delta>0$があって$|x-x'|\leq \delta$ならば
\begin{equation*}
\bigg|\sum_{k=1}^{K_{m,n}}V_k(f_k(x)-f_k(x'))\bigg|<\frac{\epsilon}{2}
\end{equation*}
($V_k(\omega)$:fixで考えるからOK)より，$|x-x'|\leq \delta$ならば
\begin{equation*}
\bigg|\sum_{k=1}^\infty V_k(f_k(x)-f_k(x'))\bigg|<\epsilon.
\end{equation*}
よって$\sum_{k=1}^\infty V_kf_k(x)$は連続．

${\bf Z}({\bf Y})$の可測性：第二可算公理から$\mathcal{X}$は可算個の$\epsilon$ボールで覆われるので順番を付けて差をとることで
disjointにすれば$F_n:\mathcal{X}\to \mathcal{X}_0$, Borelを$d(F_n(y),y)<1/n$となるように定めることができる．この時
$F_n({\bf Y})=\sum_{k=1}^\infty e_k1_{E_k}({\bf Y})$と書け，${\bf Z}(F_n({\bf Y}))=\sum_{k=1}^\infty Z(e_k)1_{E_k}({\bf Y})$
は可測なので${\bf Z}({\bf Y})=\lim_{n\to\infty}{\bf Z}(F_n({\bf Y}))$も可測．
}
\end{discuss}

Then, Lemma \ref{st-conv-lemma} yields the desired results.
\begin{discuss}
{\colorr $\sum_kV_k{\bf Z}_{n,k}(x)$が収束することは$\omega\in \{\ \}$のとき$\sum_{m,k}V_kf_k(x)$がコーシー列になることからわかる．}
\end{discuss}
\qed

\begin{discuss}
{\colorr
\begin{lemma}
$(X_m)_{m\in\mathbb{N}}$を確率変数とし，任意の$\epsilon>0$と$\eta>0$に対し，ある$N$があって，
\begin{equation*}
P[|X_m-X_n|>\delta]<\epsilon
\end{equation*}
for $m,n\geq N$とする．この時，ある確率変数$X_\infty$があって$X_n\to X_\infty$ in probability.
\end{lemma}
\begin{proof}
ある$(n_k)_k$があって，$P[|X_{n_{k+1}}-X_{n_k}|>2^{-k}]\leq 2^{-k}$. ボレル・カンテリより
\begin{equation*}
P\bigg[\cap_{n=1}^\infty\cup_{l\geq k}\{|X_{n_{l+1}}-X_{n_l}|>2^{-k}\}\bigg]=0.
\end{equation*}
よってa.s.で$X_\infty=\lim_{k\to\infty} X_{n_k}$が存在．
$P[|X_{n_k}-X_n|>\delta]<\epsilon$において$k\to\infty$とすれば$P[|X_\infty-X_n|>\delta]<\epsilon$なのでOK.
\end{proof}

\noindent
＜$\mathcal{V}(x)$の連続性＞
(\ref{Zf-tail-est})より上のLemmaと同様に$\sup_{x'}$をつけて議論すれば
\begin{equation*}
P\bigg[\sup_{x'}|\mathcal{V}_M(x')-\mathcal{V}(x')|>\delta\bigg]<\epsilon.
\end{equation*}
よって$C(U_x)$のノルムで確率収束しているので部分列がa.s.収束し，$\mathcal{V}\in C(U_x)$ a.s.
$\mathcal{X}$は可分なので可算個の$U_x$でおおわれるのでOK.
}
\end{discuss}

\end{document}